\documentclass{amsart}

\usepackage{amssymb,stmaryrd}

\usepackage{graphicx}

\usepackage{color}

\usepackage{wasysym,xspace,subcaption}
\usepackage{verbatim}

\newtheorem{theorem}{Theorem}
\newtheorem{lemma}[theorem]{Lemma}
\newtheorem{proposition}[theorem]{Proposition}
\newtheorem{corollary}[theorem]{Corollary}

\theoremstyle{definition}

\theoremstyle{remark}
\newtheorem{remark}[theorem]{Remark}

\newcommand\todo[1]{\par\bigskip\noindent TODO: {\color{magenta}#1}\par\bigskip}

\newcommand\RE{\mathbb{R}}

\newcommand\dual[2]{{}_{#2'}\langle#1\rangle_{#2}}
\newcommand\sfA{\mathsf{A}}
\newcommand\sfB{\mathsf{B}}
\newcommand\sfM{\mathsf{M}}
\newcommand\sfK{\mathsf{K}}
\newcommand\sfAkk{\sfA_{\sfK\sfK}}
\newcommand\sfu{\mathsf{u}}
\newcommand\sfp{\mathsf{p}}
\newcommand\sff{\mathsf{f}}
\newcommand\sfg{\mathsf{g}}
\newcommand\sfx{\mathsf{x}}
\newcommand\sfy{\mathsf{y}}
\newcommand\jump[2]{ \llbracket     #1 \rrbracket_{#2}}
\newcommand\sfzero{\mathsf{0}}
\newcommand\kerh{\mathsf{K}_{\sfB_h}}
\newcommand\Ker{\mathsf{K}_{\sfB}}

\newcommand\bfsigma{\boldsymbol{\sigma}}
\newcommand\bfpsi{\boldsymbol{\psi}}

\newcommand\bftau{\boldsymbol{\tau}}
\newcommand\bfvarsigma{\boldsymbol{\varsigma}}
\newcommand\bx{\mathbf{x}}
\newcommand\bn{\mathbf{n}}
\newcommand\grad{\operatorname\nabla}
\renewcommand\div{\operatorname{\mathrm{div}}}
\newcommand\Hdiv{\mathbf{H}(\div;\Omega)}
\newcommand\Vh{\Sigma_h}
\newcommand\Qh{U_h}

\newcommand\Po{\mathcal{P}_0}
\newcommand\Wone{W_1}
\newcommand\Qzero{U_0}
\newcommand\Qone{U_1}
\newcommand\Qtwo{U_2}
\newcommand\Vone{\Sigma_1}
\newcommand\Vtwo{\Sigma_2}
\newcommand\Vt{\tilde\Sigma}
\newcommand\Span{\text{span}}
\newcommand\mut{\underline\mu}
\newcommand\Nht{N_{\mut}}
\newcommand\bfsigmaone{\boldsymbol{\varphi}_{1,h}}
\newcommand\bfsigmatwo{\boldsymbol{\varphi}_{2,h}}
\newcommand\uone{w_{1,h}}
\newcommand\utwo{w_{2,h}}
\newcommand\projo{\Pi_0}

\DeclareRobustCommand{\meshcrossed}{\textsc{Crossed}\xspace}
\DeclareRobustCommand{\meshright}{\textsc{Right}\xspace}
\DeclareRobustCommand{\meshnonstructured}{\textsc{Non Structured}\xspace}

\begin{document}

\title[Is the inf-sup condition necessary?]{On the necessity of the inf-sup
condition for a mixed finite element formulation}


\author{Fleurianne Bertrand}
\address{University of Twente, The Netherlands}
\curraddr{}
\email{}
\thanks{}

\author{Daniele Boffi}
\address{King Abdullah University of Science and Technology (KAUST), Saudi
Arabia and University of Pavia, Italy}
\curraddr{}
\email{}
\thanks{}


\date{}

\begin{abstract}

We study a non standard mixed formulation of the Poisson problem, sometimes
known as dual mixed formulation.
For reasons related to the equilibration of the flux, we use finite elements
that are conforming in $\Hdiv$ for the approximation of the gradients, even
if the formulation would allow for discontinuous finite elements.
The scheme is not uniformly inf-sup stable, but we can show existence and
uniqueness of the solution, as well as optimal error estimates for the
gradient variable when suitable regularity assumptions are made. Several
additional remarks complete the paper, shedding some light on the sources of
instability for mixed formulations.

\end{abstract}

\maketitle

\section{Introduction}
\label{se:intro}

In this paper we discuss the numerical approximation of saddle point problems
of the following form: given two Hilbert spaces $V$ and $Q$, two continuous
bilinear forms $a(\cdot,\cdot):V\times V\to\RE$ and $b(\cdot,\cdot):V\times
Q\to\RE$, and two functionals $f\in V'$ and $g\in Q'$, find $u\in V$ and $p\in
Q$ such that
\begin{equation}
\left\{
\aligned
&a(u,v)+b(v,p)=\dual{f,v}{V}&&\forall v\in V\\
&b(u,q)=\dual{g,q}{Q}&&\forall q\in Q.
\endaligned
\right.
\label{eq:mixed}
\end{equation}

It is well-known that a conforming approximation of the problem relies on
suitable stability conditions, which are usually referred to as \emph{inf-sup
conditions} (see Section~\ref{se:infsup}).
While it is universally understood that the inf-sup conditions are sufficient
for the quasi-optimal convergence of any Galerkin discretization, the question
whether such conditions are also necessary is less studied; nevertheless it is
a common belief that in general the inf-sup conditions are essentially
sufficient and necessary for the optimal behavior of a numerical scheme and
everybody agrees that inf-sup unstable formulations should be avoided unless
special tricks are adopted (stabilizations, filtering of spurious modes,
special meshes, etc.).

In this paper we discuss the necessity of the inf-sup conditions by studying
the approximation of a non standard mixed formulation for the Poisson
equation. In particular, we present a scheme which, under suitable
conditions, is optimally convergent even if the inf-sup constant goes to zero
as the mesh is refined. This (counter-)example can be consider as an extension
of a one dimensional toy problem that has been studied in~\cite{1D}. The
interested reader is also referred to the abstract setting sketched
in~\cite[Section~5.6.2]{bbf} where it is shown what can happen when the
inf-sup condition goes wrong.

After an introductory section about the inf-sup conditions, in
Section~\ref{se:dualmixed} we present the dual mixed formulation of the
Poisson equation. Section~\ref{se:infsuph} deals with discrete inf-sup
constant, including some numerical tests showing the mesh dependent behavior
of the stability condition. In Section~\ref{se:splitting} we show how the
finite element spaces can be spit into a stable part and an unstable one.
Sections~\ref{se:convergence} and~\ref{se:num} present the convergence
theoretical results and some numerical tests confirming the theory. Finally,
two appendices conclude the paper with some links between the considered
problem and a flux equilibration strategy.

\section{Generalities about the inf-sup conditions}
\label{se:infsup}

In this section and in the sequel of this paper we follow the framework
of~\cite{bbf}; we recall some relevant and well known results for completeness
and for setting our notation.

The conforming Galerkin approximation of the mixed formulation presented
in~\eqref{eq:mixed} consists in choosing appropriate finite element subspaces
$V_h\subset V$ and $Q_h\subset Q$, and in finding $u_h\in V_h$ and $p_h\in
Q_h$ such that
\begin{equation}
\left\{
\aligned
&a(u_h,v)+b(v,p_h)=\dual{f,v}{V}&&\forall v\in V_h\\
&b(u_h,q)=\dual{g,q}{Q}&&\forall q\in Q_h.
\endaligned
\right.
\label{eq:mixedh}
\end{equation}

Let $N_V$ be the dimension of $V_h$ and $N_Q$ the one of $Q_h$, the matrix
form of the discrete problem expressed in~\eqref{eq:mixedh} is given by
\begin{equation}
\left(
\begin{matrix}
\sfA&\sfB^\top\\
\sfB&\sfzero
\end{matrix}
\right)
\left(
\begin{matrix}
\sfu\\
\sfp
\end{matrix}
\right)
=
\left(
\begin{matrix}
\sff\\
\sfg
\end{matrix}
\right),
\label{eq:matrix}
\end{equation}
where $\sfA$ is a square matrix of size $N_V\times N_V$, $\sfB$ is a
rectangular matrix of size $N_Q\times N_V$, $\sfu\in\RE^{N_V}$ and
$\sfp\in\RE^{N_Q}$ are column vector representations of $u\in V_h$ and $p\in
Q_h$, respectively, and $\sff\in\RE^{N_V}$ and $\sfg\in\RE^{N_Q}$ are column
vector realizations of the right hand sides $f$ and $g$, respectively.

The necessary and sufficient conditions for the solvability
of~\eqref{eq:matrix} are summarized in~\cite[Theorem~3.2.1]{bbf}, which we now
recall for the reader's convenience.
We denote by $\sfK$ the kernel of $\sfB$
\[
\sfK=\ker\sfB.
\]
The restriction of $\sfA$ to $\sfK$ is denoted by $\sfAkk$
\[
\sfAkk=\Pi_\sfK\sfA E_\sfK,
\]
where $\Pi_\sfK$ and $E_\sfK$ are the projection from $\RE^{N_V}$ onto $\sfK$
and the embedding from $\sfK$ into $\RE^{N_V}$, respectively. The two
conditions equivalent to the solvability of~\eqref{eq:matrix} (for all possible
right hand sides) are the following ones.

\begin{description}
\item[M1]
The matrix $\sfAkk$ is invertible.

\item[M2]
$N_V\ge N_Q$ and the matrix $\sfB$ is full rank.

\end{description}
\textbf{M1} can be expressed by saying that the operator
associated with $\sfAkk$ is surjective or, equivalently, injective;
analogously, \textbf{M2} states that the operator associated with
$\sfB$ is surjective or, equivalently, that the one associated with
$\sfB^\top$ is injective.

The essential ideas behind the inf-sup theory is that a uniform stability of
problem~\eqref{eq:mixedh} with respect to the parameter $h$ requires that
conditions \textbf{M1} and \textbf{M2} are made explicit and uniform with
respect to $h$. This is done by introducing suitable inf-sup conditions.

Let $\kerh$ be the subspace of $V_h$ associated with the kernel of the matrix
$\sfB$
\[
\kerh=\{v_h\in V_h:b(v_h,q)=0\ \forall q\in Q_h\}.
\]
Then hypotheses \textbf{M1} and \textbf{M2} correspond to the following two
inf-sup conditions, respectively (where, as usual, when there is an inf-sup
involving fractions, we understand that the infimum and the supremum are taken
over non vanishing functions).

\begin{description}
\item[IS1]
There exists a constant $\alpha_h>0$ such that
\begin{equation}
\inf_{v_h\in\kerh}\sup_{w_h\in\kerh}\frac{a(v_h,w_h)}{\|v_h\|_V\|w_h\|_V}
\ge\alpha_h.
\label{eq:infsup1h}
\end{equation}

\item[IS2]
There exists a constant $\beta_h>0$ such that
\begin{equation}
\inf_{q_h\in Q_h}\sup_{v_h\in V_h}\frac{b(v_h,q_h)}{\|v_h\|_V\|q_h\|_Q}
\ge\beta_h.
\label{eq:infsup2h}
\end{equation}

\end{description}

We assume that the continuous problem~\eqref{eq:mixed} is stable; in
particular, the following conditions analogue to \textbf{IS1} and \textbf{IS2}
guarantee this property, where the continuous kernel is defined as
\[
\Ker=\{v\in V:b(v,q)=0\ \forall q\in Q\}.
\]

\begin{itemize}
\item
There exists a constant $\alpha>0$ such that
\begin{equation}
\inf_{v\in\Ker}\sup_{w\in\Ker}\frac{a(v,w)}{\|v\|_V\|w\|_V}
\ge\alpha
\label{eq:infsup1a}
\end{equation}
and
\begin{equation}
\inf_{w\in\Ker}\sup_{v\in\Ker}\frac{a(v,w)}{\|v\|_V\|w\|_V}
\ge\alpha.
\label{eq:infsup1b}
\end{equation}

\item
There exists a constant $\beta>0$ such that
\begin{equation}
\inf_{q\in Q}\sup_{v\in V}\frac{b(v,q)}{\|v\|_V\|q\|_Q}
\ge\beta.
\label{eq:infsup2}
\end{equation}

\end{itemize}

The following theorem summarizes the stability and convergence result that is
obtained when the inf-sup conditions are uniform with respect to the parameter
$h$.

\begin{theorem}
If there exist $\alpha_0>0$ and $\beta_0>0$ such that the constants in
\textbf{IS1} and \textbf{IS2} are uniformly bounded below, that is,
$\alpha_h\ge\alpha_0$ and $\beta_h\ge\beta_0$ for all $h$, then the following
quasi-optimal error estimate holds true
\begin{equation}
\|u-u_h\|_V+\|p-p_h\|_Q\le C\inf_{\substack{v_h\in V_h\\ q_h\in Q_h}}
\left(\|u-v_h\|_V+\|p-q_h\|_Q\right),
\label{eq:quasiopt}
\end{equation}
where $(u,p)$ and $(u_h,p_h)$ are the solutions of~\eqref{eq:mixed}
and~\eqref{eq:mixedh}, respectively.

\label{th:infsup}
\end{theorem}

\begin{remark}
The constant $C$ in~\eqref{eq:quasiopt} could be made explicit in terms of
$\alpha_0$ and $\beta_0$. The interested reader is referred
to~\cite[Theorem~5.2.1]{bbf}.
\label{re:infsup}
\end{remark}

\section{A non-standard mixed formulation for the Laplace equation}
\label{se:dualmixed}

In this paper we consider the homogeneous Dirichlet problem for the Laplace
equation on a polygonal domain in $\mathbb{R}^2$: given $f$ find $u$ such that
\begin{equation}
\left\{
\aligned
&-\Delta u=f&&\text{in }\Omega\\
&u=0&&\text{on }\partial\Omega.
\endaligned
\right.
\label{eq:poisson}
\end{equation}

More general elliptic equations might be considered, but we believe that this
is the simplest and most effective setting for the presentation of our
results.

We split the second order equation as a system of two first order equations by
introducing the variable $\bfsigma=-\grad u$. As opposed to the standard mixed
formulation, we integrate by parts the equilibrium equation and not the
equation defining $\bfsigma$, and obtain the following problem: given $f$ in
$H^{-1}(\Omega)$, find $\bfsigma\in L^2(\Omega)^2$ and $u\in H^1_0(\Omega)$
such that
\begin{equation}
\left\{
\aligned
&(\bfsigma,\bftau)+(\bftau,\grad u)=0&&\forall\bftau\in L^2(\Omega)^2\\
&(\bfsigma,\grad v)=-\langle f,v\rangle&&\forall v\in H^1_0(\Omega).
\endaligned
\right.
\label{eq:dualmixed}
\end{equation}
Sometimes this formulation is called dual mixed formulation to differentiate
it from the standard primal mixed formulation.

\begin{remark}
Our interest in problem~\eqref{eq:dualmixed} is related to an analogue
formulation used in elasticity (see~\cite{joerg}). A numerical study of the
inf-sup condition in the case of the Laplace equation was presented
in~\cite{pamm2}.
\end{remark}

The mixed formulation is well posed in the chosen functional spaces. For
completeness, this is proved in the next theorem.

\begin{theorem}
The mixed formulation presented in~\eqref{eq:dualmixed} is well posed in
the sense that the inf-sup conditions~\eqref{eq:infsup1a},
\eqref{eq:infsup1b}, and~\eqref{eq:infsup2} are satisfied.
\label{th:dualmixed}
\end{theorem}

\begin{proof}

The inf-sup conditions~\eqref{eq:infsup1a} and~\eqref{eq:infsup1b} follow from
the fact that the bilinear form $a(\cdot,\cdot)$ corresponds to the identity
operator in $L^2(\Omega)^2$, which is clearly invertible on the whole space.
The inf-sup condition~\eqref{eq:infsup2} for the bilinear form
$b(\cdot,\cdot)$ follows from the inclusion $\grad H^1_0(\Omega)\subset
L^2(\Omega)^2$: given $v\in H^1_0(\Omega)$, the vectorfield $\bftau=\grad v$
satisfies $b(\bftau,v)=\|\grad v\|^2_{L^2(\Omega)}\ge
C_1\|v\|^2_{H^1_0(\Omega)}$ and $\|\bftau\|_{L^2(\Omega)}=\|\grad
v\|_{L^2(\Omega)}\le C_2\|v\|_{H^1_0(\Omega)}$, that is~\eqref{eq:infsup2}
with $\beta=C_1/C_2$.

\end{proof}

Given two discrete subspaces $\Vh\subset L^2(\Omega)^2$ and $\Qh\subset
H^1_0(\Omega)$, the discretization or problem~\eqref{eq:dualmixed} reads: find
$\bfsigma_h\in\Vh$ and $u_h\in\Qh$ such that
\begin{equation}
\left\{
\aligned
&(\bfsigma_h,\bftau)+(\bftau,\grad u_h)=0&&\forall\bftau\in\Vh\\
&(\bfsigma_h,\grad v)=-\langle f,v\rangle&&\forall v\in\Qh.
\endaligned
\right.
\label{eq:dualmixedh}
\end{equation}

\begin{remark}
Since the inf-sup conditions~\eqref{eq:infsup1a} and~\eqref{eq:infsup1b} are
satisfied on the entire space $L^2(\Omega)^2$, any choice of discrete spaces
$\Vh$ and $\Qh$ will satisfy uniformly the discrete inf-sup
condition~\eqref{eq:infsup1h}. It follows that the only condition to be shown
for the stability of the discretization, is the uniform bound of the discrete
inf-sup constant in~\eqref{eq:infsup2h} associated with the bilinear form
$b(\cdot,\cdot)$.
\end{remark}

A natural choice for the discrete spaces is given by discontinuous piecewise
polynomials of degree $k$ (in each component) for $\Vh$ and continuous
piecewise polynomials of degree $k+1$ for $\Qh$. In the next proposition we
state the stability and the quasi-optimal convergence of the resulting scheme.  

\begin{proposition}
For $k\ge0$ let $\Vh$ be the space of discontinuous piecewise polynomials of
degree $k$ in each component and $\Qh$ be the space of continuous piecewise
polynomials of degree $k+1$ with zero boundary conditions. Then the
approximation~\eqref{eq:dualmixedh} of the mixed
formulation~\eqref{eq:dualmixed} is uniformly stable in the sense of
Theorem~\ref{th:infsup} and the following quasi-optimal error estimate holds
true
\[
\|\bfsigma-\bfsigma_h\|_{L^2(\Omega)}+\|u-u_h\|_{H^1_0(\Omega)}\le
C\inf_{\substack{\bftau\in\Vh\\ v_h\in\Qh}}
\left(\|\bfsigma-\bftau\|_{L^2(\Omega)}+\|u-v\|_{H^1_0(\Omega)}\right).
\]
In particular, if the solution $(\bfsigma,u)$ is smooth enough, we have
\[
\aligned
\|\bfsigma-\bfsigma_h\|_{L^2(\Omega)}+\|u-u_h\|_{H^1_0(\Omega)}&\le
Ch^{k+1}\left(\|\bfsigma\|_{H^{k+1}(\Omega)}+\|u\|_{H^{k+2}(\Omega)}\right)\\
&\le Ch^{k+1}\|u\|_{H^{k+2}(\Omega)}.
\endaligned
\]
In general, if $u\in H^{1+s}(\Omega)$ with $s\le k+1$, we obtain
\[
\|\bfsigma-\bfsigma_h\|_{L^2(\Omega)}+\|u-u_h\|_{H^1_0(\Omega)}
\le Ch^s\|u\|_{H^{1+s}(\Omega)}.
\]
\label{pr:P0P1}
\end{proposition}

\begin{proof}

The stability proof follows the same lines as in Theorem~\ref{th:dualmixed}.
In particular, the uniform inf-sup condition~\eqref{eq:infsup1h} is guaranteed
by the global invertibility of the bilinear form $a(\cdot,\cdot)$ which
corresponds to the identity in $L^2(\Omega)$. The uniform inf-sup
condition~\eqref{eq:infsup2h} for the bilinear form $b(\cdot,\cdot)$ follows
from the inclusion $\grad\Qh\subset\Vh$.

\end{proof}

It can actually be easily observed that if $\grad\Qh\subset\Vh$ then the mixed
formulation~\eqref{eq:dualmixedh} is equivalent to the standard Galerkin
formulation where the space $\Qh$ is used. We state this result in the
following proposition.

\begin{proposition}
If the inclusion $\grad\Qh\subset\Vh$ is satisfied, then the mixed
formulation~\eqref{eq:dualmixedh} is well posed in the sense of
Theorem~\ref{th:infsup} and the component $u_h$ of its solution solves the
standard Galerkin formulation
\begin{equation}
(\grad u_h,\grad v)=\langle f,v\rangle\qquad\forall v\in U_h.
\label{eq:Galh}
\end{equation}
The other component of the solution is given by
$\bfsigma_h=\grad u_h$.
\label{pr:equiv}
\end{proposition}

\begin{proof}
It is easily seen that inserting $\bfsigma_h=\grad u_h$
into~\eqref{eq:dualmixedh}, the first equation is an identity and the second
equation corresponds precisely to~\eqref{eq:Galh}.
\end{proof}

The lowest-order element presented in Proposition~\ref{pr:P0P1} will be
referred to as the $P_0-P_1$ scheme.

In a more general context,~\cite{joerg} discusses the approximation of a
linear elasticity problem with a mixed scheme which has some analogies with
our formulation~\eqref{eq:dualmixed}. For particular reasons related to some
equilibration properties that will be made more precise later on, it is
proposed the use of Raviart--Thomas elements $RT_0$ for the definition of
$\Vh$. If \emph{discontinuous} $RT_0$ elements are used, then the same
stability proof as for the $P_0-P_1$ scheme applies. This follows from the
fact that $\Vh$ contains the space of piecewise constants. We state this
result in the following corollary for any degree $k$.

\begin{corollary}
Let $\Vh$ be the space of discontinuous Raviart--Thomas finite elements of
degree $k\ge0$ and $\Qh$ be the space of continuous piecewise polynomials of
degree $k+1$ with homogeneous boundary conditions.
Then the formulation~\eqref{eq:dualmixedh} is uniformly stable and the
following error estimate holds true if the solution $u$ is smooth enough
\[
\|\bfsigma-\bfsigma_h\|_{L^2(\Omega)}+\|u-u_h\|_{H^1_0(\Omega)}\le
Ch^s\|u\|_{H^{1+s}(\Omega)}
\]
with $s\le k+1$.
Moreover, the mixed problem is equivalent to the standard Galerkin
approximation~\eqref{eq:Galh} with continuous polynomials of degree $k+1$ and
$\bfsigma_h=\grad u_h$.
\end{corollary}

We now consider the lowest order case, that is $k=0$, so that the
approximation of $u$ is obtained by standard piecewise linear elements.

For reasons related to the equilibration property $\div\bfsigma=f$,
in~\cite{joerg} it is proposed to use $\Hdiv$-conforming $RT_0$ elements for
the definition of $\Vh$. We are going to denote this element by $RT_0-P_1$.
Numerical evidence seems to indicate that for the elasticity problem this
choice provides a uniformly stable scheme~\cite{joerg}, while this is not the
case for the Poisson problem~\cite{pamm2}. We are going to analyze in more
detail this element in the next sections.

\section{The inf-sup condition for the $RT_0-P_1$ scheme}
\label{se:infsuph}

The first inf-sup condition \textbf{IS1}~\eqref{eq:infsup1h} is automatically
satisfied for our mixed formulation (see Remark~\ref{re:infsup}), so that we
are only discussing the second inf-sup \textbf{IS2}~\eqref{eq:infsup2h} which
reads:
\begin{equation}
\inf_{v\in\Qh}\sup_{\bftau\in\Vh}
\frac{(\bftau,\grad v)}{\|\bftau\|_{L^2(\Omega)}\|v\|_{H^1(\Omega)}}\ge\beta_h.
\label{eq:betah}
\end{equation}

We start with a positive result, showing that for all $h$ the constant
$\beta_h$ is strictly greater than zero.

\begin{theorem}

For all $v\in\Qh$ with $v\ne0$ there exists $\bftau\in\Vh$ such that
\[
(\bftau,\grad v)>0,
\]
that is the inf-sup constant in~\eqref{eq:betah} satisfies $\beta_h>0$ for all
$h$.

\label{th:existence}
\end{theorem}

\begin{proof}
After integration by parts and taking into account the boundary conditions, we
have
\[
(\bftau,\grad v)=-(\div\bftau,v).
\]
We denote by $\projo$ the $L^2(\Omega)$ projection onto the space of
piecewise constant functions; the term $(\bftau,\grad v)$ is maximized by
taking $\bftau\in\Vh$ with $\div\bftau=\projo v$, so that we have
\[
(\bftau,\grad v)=-\|\projo v\|^2_{L^2(\Omega)}.
\]
Hence the result follows by observing that $\projo v=0$ implies $v=0$ if
$v$ is vanishing on $\partial\Omega$.
Indeed, it is not possible to construct a function $v\in\Qh$
that is zero mean valued in each element. This is easily seen by starting from
a boundary element $T$ (with two vertices on $\partial\Omega$): if $v$ is zero
mean valued on $T$, then necessarily it vanishes on $T$; the same argument can
then be applied to the neighboring elements sharing an edge with $T$ and so on
until it is seen that $v$ must vanish on all elements of the triangulation.
\end{proof}

The immediate consequence of the previous theorem is that
problem~\eqref{eq:dualmixedh} is solvable.

\begin{corollary}
For all $f\in H^{-1}(\Omega)$ and all $h$ there exists a unique solution to
problem~\eqref{eq:dualmixedh}.
\end{corollary}

We postpone to Appendix~\ref{ap:fleur} further theoretical investigations
about the behavior of the inf-sup constant. Here we continue this study
numerically.

It is well known that an estimate of the inf-sup constant $\beta_h$ appearing
in~\eqref{eq:betah} can be obtained by solving an algebraic problem.
In~\cite{bbf} a singular value decomposition is used. An essentially
equivalent approach was described in~\cite{chapelle} (see
Remark~\ref{re:bbf}), based on an idea from~\cite{malkus}.

\begin{proposition}
Let $\Vh$ and $\Qh$ be finite element spaces and consider the discrete saddle
point problem~\eqref{eq:dualmixedh}.
Let $\sfA$ and $\sfB$ be the corresponding matrices appearing
in~\eqref{eq:matrix} and introduce the matrix $\sfM$ corresponding to the
$H^1_0(\Omega)$ inner product $(\grad\cdot,\grad\cdot)$ in $\Qh$. Then the
inf-sup constant $\beta_h$ in~\eqref{eq:betah} is equal to $\sqrt{\mu_{min}}$,
where $\mu_{min}$ is the smallest eigenvalue $\mu$ of the
generalized eigenvalue problem
\begin{equation}
\sfB\sfA^{-1}\sfB^\top\sfx=\mu\sfM\sfx.
\label{eq:svd}
\end{equation}
\label{pr:svd}
\end{proposition}

\begin{remark}
In~\cite[Section~3.4.3]{bbf} it is shown that the constant $\beta_h$ is equal
to the smallest singular value of the matrix
$\sfM^{-1/2}\sfB\sfA^{-1/2}$. This statement is equivalent to
Proposition~\ref{pr:svd} since the singular values of
$\mathcal{A}=\sfM^{-1/2}\sfB\sfA^{-1/2}$ are the square roots of the
eigenvalues of
$\mathcal{A}\mathcal{A}^\top=\sfM^{-1/2}\sfB\sfA^{-1}\sfB^\top\sfM^{-1/2}$.
Please note the typo in~\cite[Proposition~3.4.5]{bbf}, where the involved
matrix should read $S_Y^{-1} B S_X^{-1}$ instead of $S_Y B S_X$.
\label{re:bbf}
\end{remark}

The eigenvalue problem~\eqref{eq:svd} has $N(h)=\dim(\Qh)$ eigensolutions
(taking into account possibly repeated eigenvalues); we number the eigenvalues
starting from the largest one, so that $\mu_{min}=\mu_{N(h)}$,
\[
\mu_1\ge\mu_2\ge\dots\ge\mu_{N(h)}\ge0
\]
and we denote the corresponding eigenvectors by $\{\sfx_i\}\subset\RE^{N(h)}$.
Each eigenvector $\sfx_i$ represents an element $g_i$ of $\Qh$ and we have
\[
\Qh=\Span\{g_1,\dots,g_{N(h)}\}.
\]

The eigenvalue problem~\eqref{eq:svd} associated with the inf-sup condition
has also a mixed equivalent formulation:
\[
\left(
\begin{matrix}
\sfA&\sfB^\top\\
\sfB&\sfzero
\end{matrix}
\right)
\left(
\begin{matrix}
\sfy\\
\sfx
\end{matrix}
\right)
=\mu
\left(
\begin{matrix}
\sfzero&\sfzero\\
\sfzero&-M
\end{matrix}
\right)
\left(
\begin{matrix}
\sfy\\
\sfx
\end{matrix}
\right).
\]
The two formulations are easily shown equivalent to each other by solving for
$\sfy$ the first equation $\sfA\sfy+\sfB^\top\sfx=0$ and substituting into the
second equation.

This allows the definition of associated eigenvectors
$\{\sfy_i\}\subset\RE^{N(h)}$ in addition to the $\{\sfx_i\}$'s satisfying the
relation
\[
\sfA\sfy_i+\sfB^\top\sfx_i=0\qquad (i=1,\dots,N(h)).
\]
Translating into the finite element notation, we have constructed two sets of
finite element functions $\{\bfsigma_{i,h}\}\subset\Vh$ and
$\{u_{i,h}\}\subset\Qh$ that satisfy the following variational problem
\begin{equation}
\left\{
\aligned
&(\bfsigma_{i,h},\bftau)+(\bftau,\grad u_{i,h})=0&&\forall\bftau\in\Vh\\
&(\bfsigma_{i,h},\grad v)=-\mu_i(\grad u_{i,h},\grad v)&&\forall v\in\Qh
\endaligned
\right.
\label{eq:eiginfsup}
\end{equation}
and such that
\[
\aligned
&\Span\{\bfsigma_{1,h},\dots,\bfsigma_{N(h),h}\}\subset\Vh\\
&\Span\{u_{1,h},\dots,u_{N(h),h}\}=\Qh.
\endaligned
\]
The eigenvectors $\{u_{i,h}\}$ can be chosen so that
\[
\aligned
&(\grad u_{i,h},\grad u_{j,h})_{L^2(\Omega)}=0&&\text{if }i\ne j\\
&\|\grad u_{i,h}\|_{L^2(\Omega)}=1.
\endaligned
\]
It follows that also the $\{\bfsigma_{i,h}\}$ are orthogonal, since
\[
(\bfsigma_{i,h},\bfsigma_{j,h})_{L^2(\Omega)}=
-(\bfsigma_{i,h},\grad u_{j,h})_{L^2(\Omega)}=
\mu_i(\grad u_{i,h},\grad u_{j,h})_{L^2(\Omega)}=0
\qquad\text{if }i\ne j.
\]

In~\cite{pamm2} it was shown numerically that the inf-sup constant $\beta_h$
is not uniformly bounded from below. More precisely, $\mu_{N(h)}$ (together
with other eigenvalues of~\eqref{eq:svd}) tends to zero when $h$ goes to zero.
We report here again this behavior and show how it may be different depending
on the mesh sequence that we are considering.

We use three sequences of mesh on the unit square: two structured and one
unstructured. With obvious meaning, we call them \meshcrossed, \meshright, and
\meshnonstructured mesh, respectively. An example of such meshes is shown in
Figure~\ref{fg:meshes}.

\begin{figure}
\includegraphics[width=12cm]{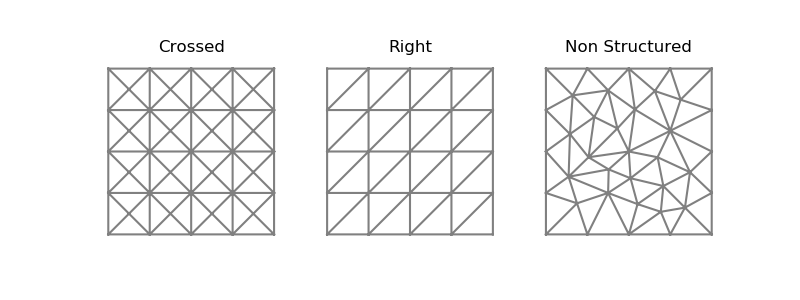}
\caption{\meshcrossed, \meshright, and \meshnonstructured meshes}
\label{fg:meshes}
\end{figure}

Table~\ref{tb:4} shows the last four computed eigenvalues
$\{\mu_{N(h)-3},\dots,\mu_{N(h)}\}$.

\begin{table}
\begin{center}
\meshcrossed mesh\nopagebreak\\
\begin{tabular}{ r|l|l|l|l }
$N(h)$ & $\mu_{N(h)-3}$ & $\mu_{N(h)-2}$ & $\mu_{N(h)-1}$ & $\mu_{N(h)}$\\ \hline
13 & 0.66666667 &        0.5 & 0.5 & 0.22222222\\
41 & 0.16521696 & 0.15643855 & 0.15643855 & 0.06604647\\
145 & 0.04880971 & 0.04191655 & 0.04191655 & 0.01698587\\
545 & 0.01268672 & 0.01065182 & 0.01065182 & 0.00427448\\
2113 & 0.00320245 & 0.00267372 & 0.00267372 & 0.00107035
\end{tabular}
\end{center}

\medskip

\begin{center}
\meshright mesh\nopagebreak\\
\begin{tabular}{ r|l|l|l|l }
$N(h)$ & $\mu_{N(h)-3}$ & $\mu_{N(h)-2}$ & $\mu_{N(h)-1}$ & $\mu_{N(h)}$\\ \hline
25 & 0.44698968 & 0.41649077 & 0.23888594 & 0.23720409\\
81 & 0.14099494 & 0.14089618 & 0.06715927 & 0.06707865\\
289 & 0.03714468 & 0.03714446 & 0.01720941 & 0.01720741\\
1089 & 0.00938762 & 0.00938762 & 0.00432346 & 0.00432341\\
4225 & 0.00235165 & 0.00235165 & 0.00108154 & 0.00108154
\end{tabular}
\end{center}

\medskip

\begin{center}
\meshnonstructured mesh\nopagebreak\\
\begin{tabular}{ r|l|l|l|l }
$N(h)$ & $\mu_{N(h)-3}$ & $\mu_{N(h)-2}$ & $\mu_{N(h)-1}$ & $\mu_{N(h)}$\\ \hline
38 & 0.25105003 & 0.22389003 & 0.16903879 & 0.11779969\\
140 & 0.10876822 & 0.0991606 & 0.08762351 & 0.05880215\\
531 & 0.0736162 & 0.06851846 & 0.06604323 & 0.05821584\\
2066 & 0.05617133 & 0.05523537 & 0.05366541 & 0.04984366\\
8128 & 0.04867193 &  0.04780397 & 0.0451059 & 0.04200516
\end{tabular}
\end{center}
\caption{Numerical estimate of the inf-sup constant for $RT_0-P_1$ scheme}
\label{tb:4}
\end{table}

It turns out that in the case of the structured meshes the smallest
eigenvalue $\mu_{N(h)}$ goes to zero quadratically in $h$, thus giving the
estimate $\beta_h=O(h)$. The behavior of the inf-sup constant on the
unstructured mesh is less critical, even if not optimal, showing a slower
decay of $\mu_{N(h)}$ as $h$ goes to zero.

It interesting to look at the eigenfunctions corresponding
to~\eqref{eq:eiginfsup}. Figure~\ref{fg:infsup} shows the eigenfunctions
$u_{1,h}$ and $u_{N(h),h}$ corresponding to the maximum and minimum eigenvalue
for the three meshes.
\begin{figure}

\meshcrossed mesh\nopagebreak\\
\includegraphics[width=5cm]{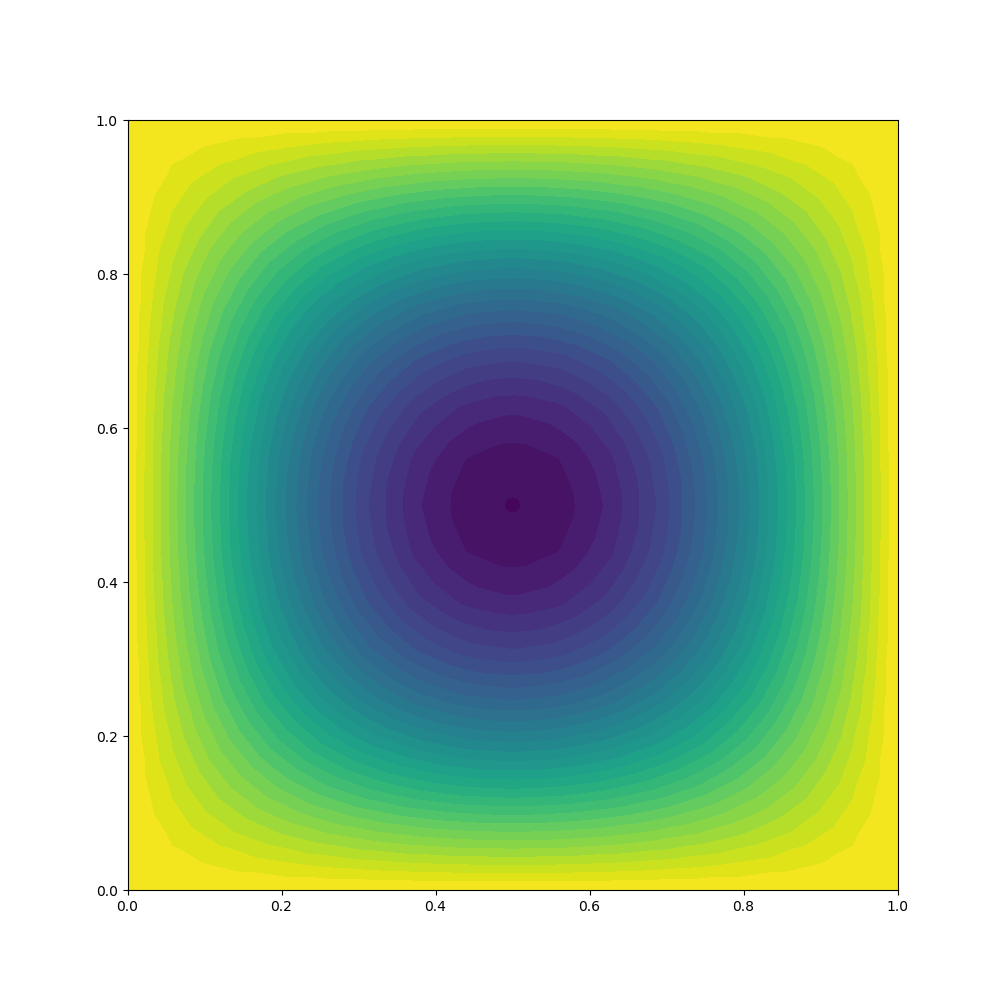}
\includegraphics[width=5cm]{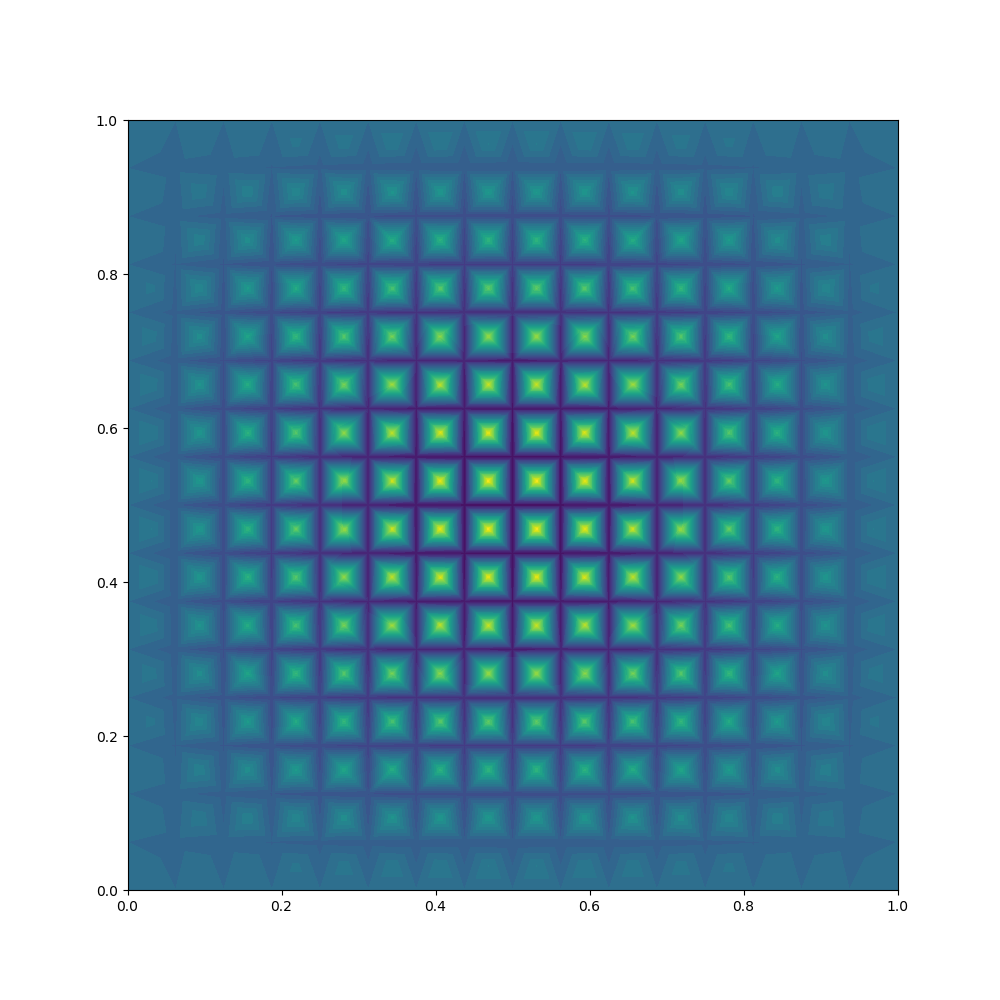}

\meshright mesh\nopagebreak\\
\includegraphics[width=5cm]{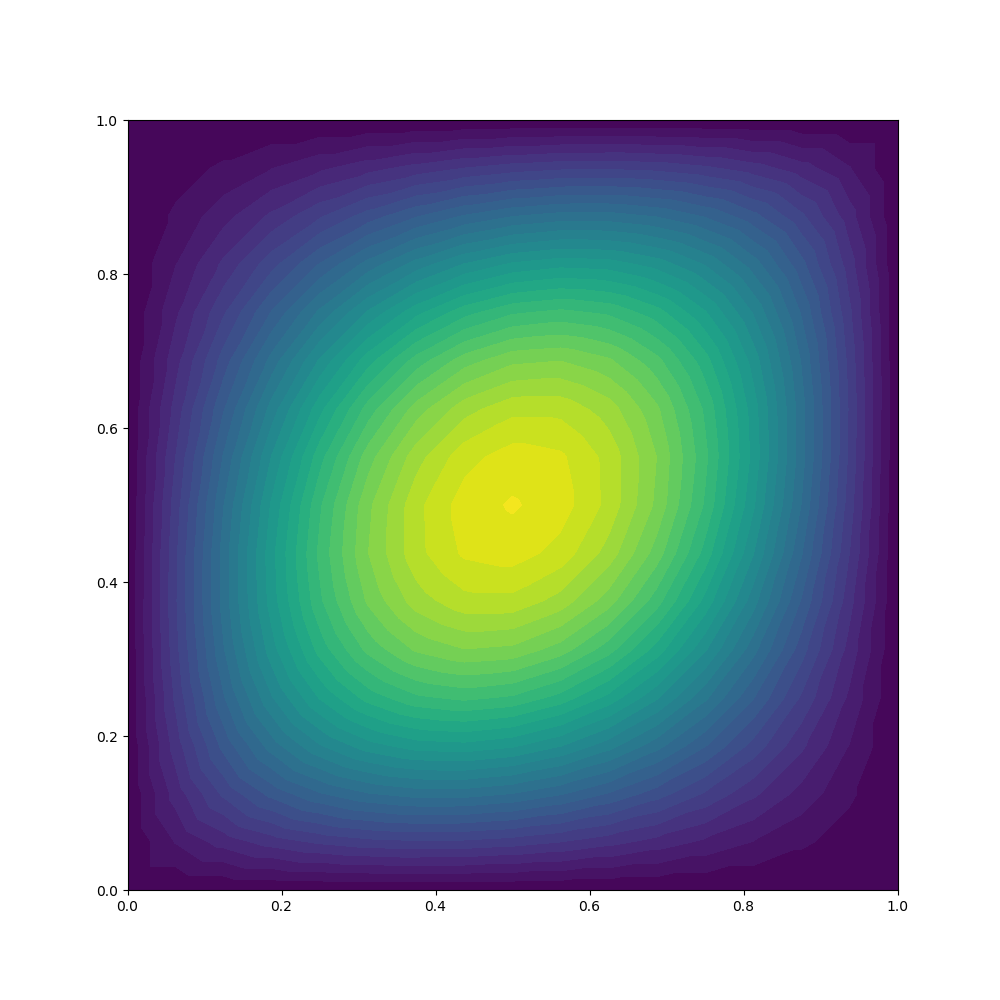}
\includegraphics[width=5cm]{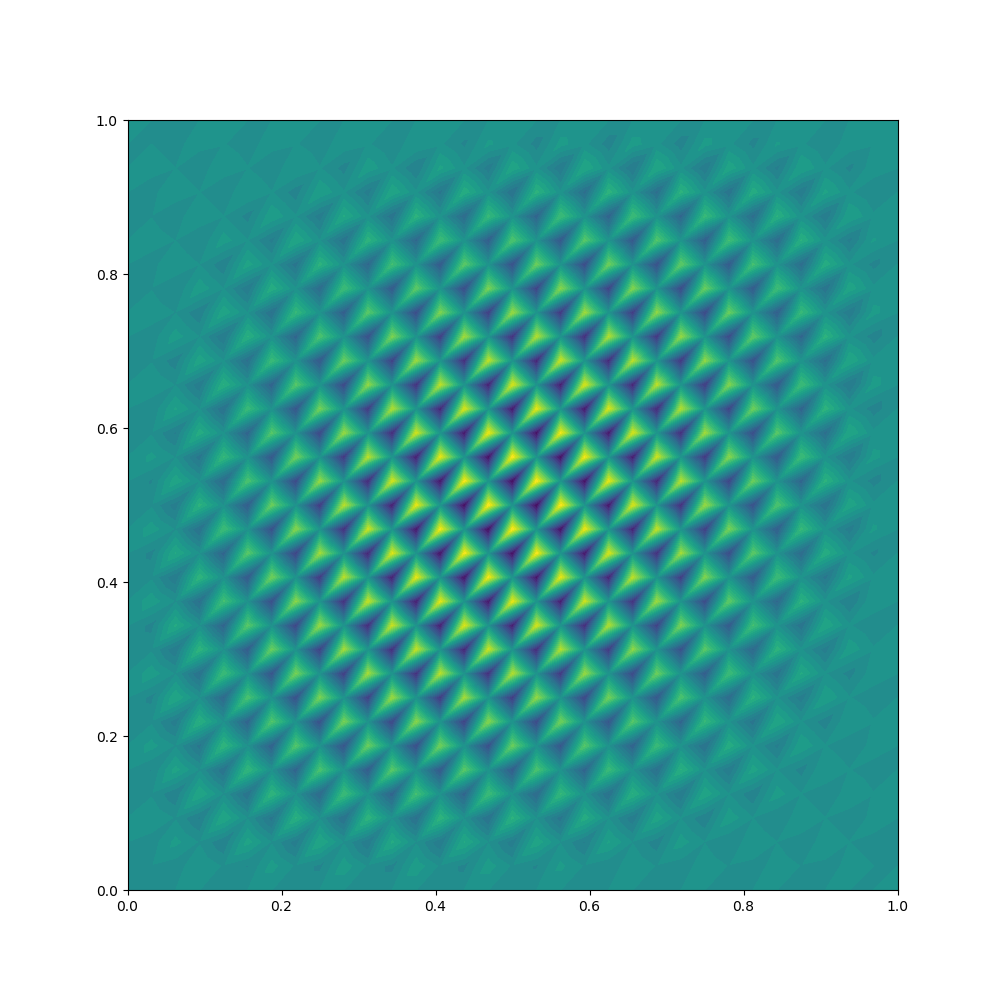}

\meshnonstructured mesh\nopagebreak\\
\includegraphics[width=5cm]{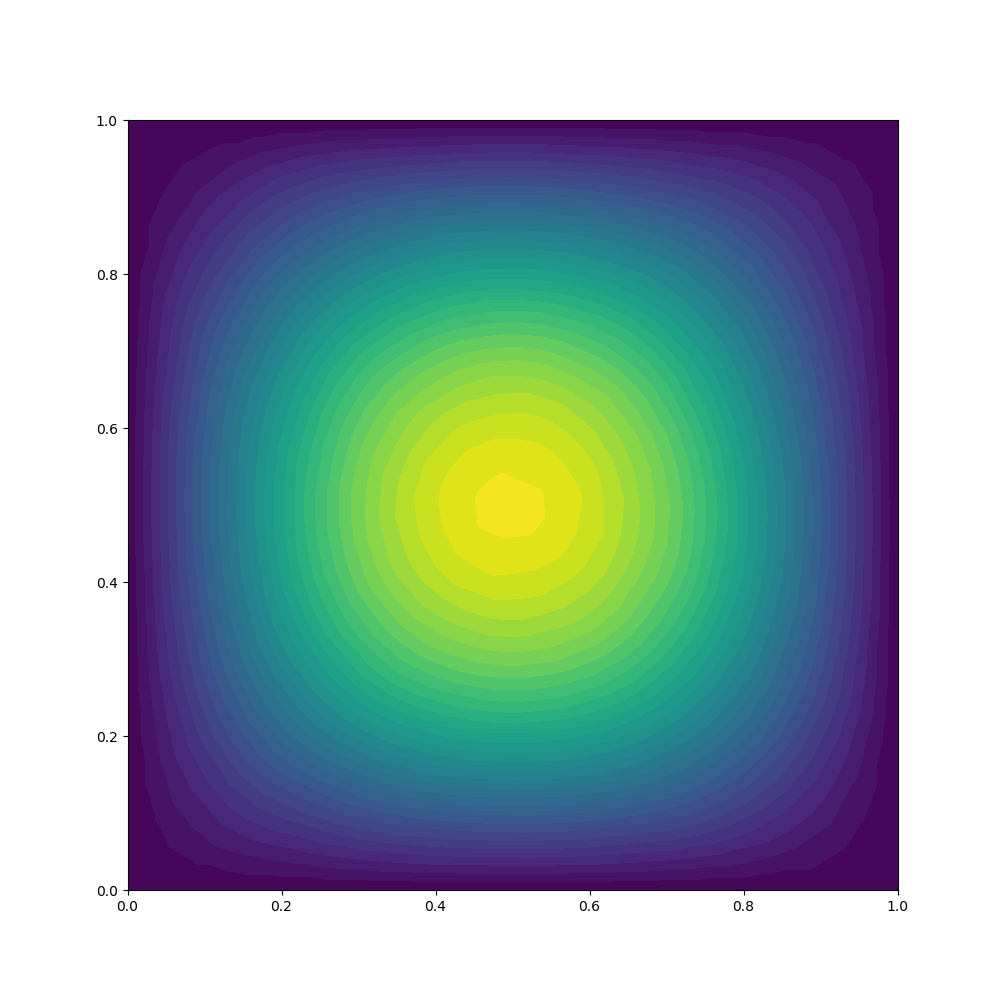}
\includegraphics[width=5cm]{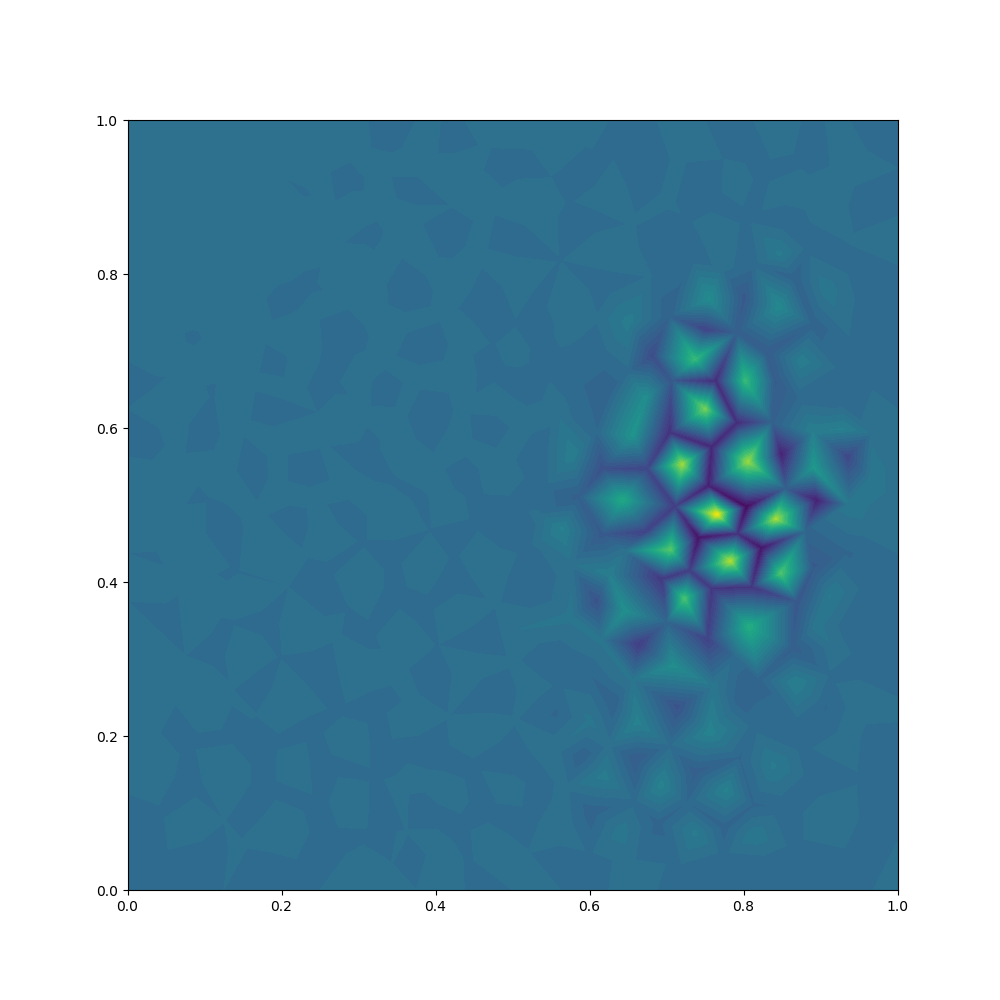}
\caption{Eigenfunctions associated with first (left) and last (right)
eigenvalues $\mu_1$ and $\mu_{N(h)}$ of~\eqref{eq:eiginfsup}}
\label{fg:infsup}
\end{figure}
It is apparent that the eigenfunctions corresponding to the smallest
eigenvalue are highly oscillatory, while the ones associated to the largest
one don't change their sign. This fact can be made more precise by identifying
an appropriate subspace of $U_h$ for which a uniform inf-sup condition holds
true. This will be done in the next section where we discuss possible
splittings of the spaces in order to separate the \emph{stable} part of the
solution from the \emph{unstable} one.

\section{Stable and unstable subspaces}
\label{se:splitting}

The following discussion identifies special subspaces of $\Qh$ for which a
uniform inf-sup condition holds true.
According to what we have seen in the previous section, we are expecting that
the degeneracy of the inf-sup constant is associated with highly oscillatory
eigenfunctions of problem~\eqref{eq:eiginfsup}.

Let $w_k\in\Qh$, $k=1,\dots,N(h)$, be the $k$-th discrete eigenfunction of the
Laplace operator approximated by the standard Galerkin method, that
is
\[
(\grad w_k,\grad v)=\lambda_k(w_k,v)\quad\forall v\in\Qh
\]
with $0<\lambda_1<\lambda_2\le\cdots$.

Our first result shows that if we fix $\bar N$ and restrict $\Qh$ to the space
spanned by $\{w_1,w_2,\dots,w_{\bar N}\}$, then the inf-sup condition holds
with a constant uniform in $h$.

\begin{theorem}
Let $\bar W_h$ be the subspace of $\Qh$ spanned by $\{w_1,w_2,\dots,w_{\bar
N}\}$. Then there exists $\bar\beta>0$ independent of $h$ such that
\begin{equation}
\sup_{\bftau\in\Vh}\frac{(\bftau,\grad w)}{\|\bftau\|_{L^2(\Omega)}}\ge
\bar\beta\qquad\forall w\in\bar W_h.
\label{eq:winfsup}
\end{equation}
\label{th:winfsup}
\end{theorem}

\begin{proof}
First we show that~\eqref{eq:winfsup} is satisfied when $w=w_k$ for a fixed
$k$. We define $\bfvarsigma_k$ as the solution of the following mixed
problem: find $\bfsigma_k\in\Vh$ and $p$ in the space of piecewise constant
functions $\Po$ such that
\[
\left\{
\aligned
&(\bfvarsigma_k,\bftau)+(\div\bftau,p)=0&&\forall\bftau\in\Vh\\
&(\div\bfvarsigma_k,q)=-\lambda_k(w_k,q)&&\forall q\in\Po.
\endaligned
\right.
\]
We then have
\[
\aligned
&\|\bfvarsigma_k\|_{L^2(\Omega)}\le C_\sigma\lambda_k\|w_k\|_{L^2(\Omega)}
=C_\sigma\lambda_k^{1/2}\|\grad w_k\|_{L^2(\Omega)}\\
&\div\bfvarsigma_k=-\lambda_k\projo w_k,
\endaligned
\]
where $\projo$ is the $L^2$ projection onto $\Po$.

By choosing $\bftau=\bfvarsigma_k$ in~\eqref{eq:winfsup} we can conclude
\begin{equation}
\aligned
\sup_{\bftau\in\Vh}\frac{(\bftau,\grad w_k)}{\|\bftau\|_{L^2(\Omega)}}&\ge
\frac{(\bfvarsigma_k,\grad w_k)}{\|\bfvarsigma_k\|_{L^2(\Omega)}}=
-\frac{(\div\bfvarsigma_k,w_k)}{\|\bfvarsigma_k\|_{L^2(\Omega)}}=
\frac{\lambda_k(\projo w_k,w_k)}{\|\bfvarsigma_k\|_{L^2(\Omega)}}\\
&=\frac{\lambda_k(w_k,w_k)-\lambda_k(w_k-\projo w_k,w_k)}
{\|\bfvarsigma_k\|_{L^2(\Omega)}}\\
&=\frac{\|\grad w_k\|^2_{L^2(\Omega)}-\lambda_k(w_k-\projo w_k,w_k-\projo w_k)}
{\|\bfvarsigma_k\|_{L^2(\Omega)}}\\
&\ge\frac{\|\grad w_k\|^2_{L^2(\Omega)}-
\lambda_k\|w_k-\projo w_k\|^2_{L^2(\Omega)}}
{C_\sigma\lambda_k^{1/2}\|\grad w_k\|_{L^2(\Omega)}}\\
&\ge\frac{\|\grad w_k\|^2_{L^2(\Omega)}-
\lambda_k C^2_\Pi h^2\|\grad w_k\|^2_{L^2(\Omega)}}
{C_\sigma\lambda_k^{1/2}\|\grad w_k\|_{L^2(\Omega)}}\\
&=\frac{1-C^2_\Pi h^2\lambda_k}{C_\sigma\lambda_k^{1/2}}
\|\grad w_k\|_{L^2(\Omega)}=\beta_k\|\grad w_k\|_{L^2(\Omega)},
\endaligned
\label{eq:betak}
\end{equation}
where $C_\Pi$ denotes the constant appearing in the approximation property of
$\projo$
\[
\|f-\projo f\|_{L^2(\Omega)}\le C_\Pi h\|\grad f\|_{L^2(\Omega)}.
\]
It follows that $\beta_k$ is bounded below uniformly in $h$ for $k$ fixed and
$h$ small enough.

In order to complete the proof it remains to extend the result to a generic
$w\in\bar W_h$; the restriction of $h$ small enough is removed by comparing
with Theorem~\ref{th:existence}.

We detail how to deal with $w=w_i+w_j$; the generic result follows with
similar arguments by considering a finite linear combination of discrete
eigenfunctions.
We define $\bftau=\bfvarsigma=\bfvarsigma_i+\bfvarsigma_j$
in~\eqref{eq:winfsup}, where $\bfvarsigma_k$ ($k=i,j$) is defined in the
previous step. We have
\begin{equation}
\aligned
\sup_{\bftau\in\Vh}\frac{(\bftau,\grad w)}{\|\bftau\|_{L^2(\Omega)}}&\ge
\frac{(\bfvarsigma,\grad w)}{\|\bfvarsigma\|_{L^2(\Omega)}}=
\frac{(\lambda_i\projo w_i+\lambda_j\projo w_j,w)}
{\|\bfvarsigma\|_{L^2(\Omega)}}\\
&\ge\frac{\|\grad w\|^2_{L^2(\Omega)}-
(\lambda_i(w_i-\projo w_i)+\lambda_j(w_j-\projo w_j),w)}
{C_\sigma\|\lambda_iw_i+\lambda_jw_j\|_{L^2(\Omega)}}.
\endaligned
\label{eq:ortho}
\end{equation}
Let us study separately the numerator and the denominator of the last
expression by starting with the scalar product appearing in the numerator. In
order to make the notation shorted, we denote by $f_k$ ($k=i,j$) the term
$\lambda_k w_k$.
\[
\aligned
(f_i-\projo f_i&+f_j-\projo f_j,w_i+w_j)\\
&= (f_i-\projo f_i,w_i)+(f_j-\projo f_j,w_j)\\
&\quad+
(f_i-\projo f_i,w_j)+(f_j-\projo f_j,w_i)\\
&\le C_\Pi h\left(\|\grad f_i\|_{L^2(\Omega)}\|w_i\|_{L^2(\Omega)}+
\|\grad f_j\|_{L^2(\Omega)}\|w_j\|_{L^2(\Omega)}\right)\\
&\quad+
(f_i-\projo f_i,w_j-\projo w_j)+(f_j-\projo f_j,w_i-\projo w_i)\\
&\le C_\Pi h\left(\max(\lambda_i,\lambda_j)^{1/2}\|\grad
w\|_{L^2(\Omega)}\right)\\
&\quad+
C^2_\Pi h^2
\left(\lambda_i\|\grad w_i\|_{L^2(\Omega)}\|\grad w_j\|_{L^2(\Omega)}
+\lambda_j\|\grad w_j\|_{L^2(\Omega)}\|\grad w_i\|_{L^2(\Omega)}\right)\\
&\le C_\Pi h\left(\max(\lambda_i,\lambda_j)^{1/2}\|\grad
w\|_{L^2(\Omega)}\right)\\
&\quad+
(C^2_\Pi h^2/2)\max(\lambda_i,\lambda_j)\|\grad w\|^2_{L^2(\Omega)},
\endaligned
\]
where we used twice the orthogonality of $\grad w_i$ and $\grad w_j$.
It follows that the numerator in~\eqref{eq:ortho} can be bounded below by a
positive constant times $\|\grad w\|^2_{L^2(\Omega)}$ for $h$ small enough.
For the denominator, we have
\[
\aligned
\|\lambda_iw_i+\lambda_jw_j\|^2_{L^2(\Omega)}&=
\lambda_i^2\|w_i\|^2_{L^2(\Omega)}+\lambda_j^2\|w_j\|^2_{L^2(\Omega)}\\
&=\lambda_i\|\grad w_i\|^2_{L^2(\Omega)}+\lambda_j\|\grad
w_j\|^2_{L^2(\Omega)}\\
&\le\max(\lambda_i,\lambda_j)\|\grad w\|^2_{L^2(\Omega)},
\endaligned
\]
where we used again the orthogonality of $\grad w_i$ and $\grad w_j$ together
with the orthogonality of $w_i$ and $w_j$. It follows that for $h$ small
enough there exists a constant $C$ independent of $h$, but dependent on $i$
and $j$, such that
\[
\frac{(\bfvarsigma,\grad w)}{\|\bfvarsigma\|_{L^2(\Omega)}}\ge C
\|\grad w\|_{L^2(\Omega)}.
\]
\end{proof}

\begin{remark}

The inf-sup constant of the previous theorem depends on the dimension of $\bar
W_h$. In particular, it gives a confirmation that the components of $\Qh$ that
are source of instability are associated with highly oscillating functions.
We now discuss how it is possible to split the solution of
problem~\eqref{eq:dualmixedh} into a stable part and an unstable one in a more
abstract way.

\end{remark}

It can be easily seen that all eigenvalues of~\eqref{eq:svd} are not larger
than $1$, so that we fix a threshold value $\mut$ between $0$ and $1$ and
define an index $\Nht$ (depending on $h$) so that
\[
\aligned
&\mu_i\ge\mut&&\forall i\le\Nht\\
&\mu_i<\mut&&\forall i>\Nht.
\endaligned
\]
We can then introduce the following spaces
\begin{equation}
\aligned
&\Vone=\Span\{\bfsigma_{1,h},\dots,\bfsigma_{\Nht,h}\}\\
&\Vtwo=\Span\{\bfsigma_{\Nht+1,h},\dots,\bfsigma_{N(h),h}\}\\
&\Qone=\Span\{u_{1,h},\dots,u_{\Nht,h}\}\\
&\Qtwo=\Span\{u_{\Nht+1,h},\dots,u_{N(h),h}\}.
\endaligned
\label{eq:splitting}
\end{equation}
It follows that $\Vh=\Vone\oplus\Vtwo\oplus\Vt$ and that
$\Qh=\Qone\oplus\Qtwo$, where $\Vt$ is a remainder space whose dimension is
equal to $\dim\Vh-N(h)\ge0$. This space satisfies the orthogonality
\[
(\bfsigma_{i,h},\bftau)_{L^2(\Omega)}=0\quad\forall i\ \forall\bftau\in\Vt.
\]

\begin{remark}
It turns out that the subspaces $\Vone-\Qone$ provide a stable discretization
of~\eqref{eq:dualmixedh} since the corresponding inf-sup constant is
associated with $(\mut)^{1/2}>0$ which is independent of $h$.
\end{remark}

We can now look at the matrix form of our discrete
problem~\eqref{eq:dualmixedh}
\[
\left(
\begin{matrix}
\sfA&\sfB^\top\\
\sfB&\sfzero
\end{matrix}
\right)
\left(
\begin{matrix}
\sfy\\
\sfx
\end{matrix}
\right)
=
\left(
\begin{matrix}
\sfzero\\
-\sff
\end{matrix}
\right).
\]
The above problem can be written in the following form by using the splitting
of the spaces $\Vh$ and $\Qh$:
\[
\left(
\begin{matrix}
\sfA_{11}&\sfA_{21}^\top&\sfA_{31}^\top&\sfB_{11}^\top&\sfB_{21}^\top\\
\sfA_{21}&\sfA_{22}&\sfA_{32}^\top&\sfB_{12}^\top&\sfB_{22}^\top\\
\sfA_{31}&\sfA_{32}&\sfA_{33}&\sfB_{13}^\top&\sfB_{23}^\top\\
\sfB_{11}&\sfB_{12}&\sfB_{13}&\sfzero&\sfzero\\
\sfB_{21}&\sfB_{22}&\sfB_{23}&\sfzero&\sfzero\\
\end{matrix}
\right)
\left(
\begin{matrix}
\sfy_1\\
\sfy_2\\
\sfy_3\\
\sfx_1\\
\sfx_2
\end{matrix}
\right)
=
\left(
\begin{matrix}
\sfzero\\
\sfzero\\
\sfzero\\
-\sff_1\\
-\sff_2
\end{matrix}
\right).
\]

Some of the matrices involved with this formulation are vanishing. In
particular, it is easy to see that $\sfA_{21}=\sfzero$ due to the
orthogonalities of the $\grad u_i$'s; analogously, $\sfB_{12}=\sfzero$ and
$\sfB_{21}=\sfzero$.
Moreover, due to the orthogonality of $\Vt$ with the rest of $\Vh$, it follows
that $\sfA_{31}=\sfzero$ and $\sfA_{32}=\sfzero$, and also that
$\sfB_{13}=\sfzero$ and $\sfB_{23}=\sfzero$.

Hence the matrix problem is reduced to
\[ 
\left(
\begin{matrix}
\sfA_{11}&\sfzero&\sfzero&\sfB_{11}^\top&\sfzero\\
\sfzero&\sfA_{22}&\sfzero&\sfzero&\sfB_{22}^\top\\
\sfzero&\sfzero&\sfA_{33}&\sfzero&\sfzero\\
\sfB_{11}&\sfzero&\sfzero&\sfzero&\sfzero\\
\sfzero&\sfB_{22}&\sfzero&\sfzero&\sfzero\\
\end{matrix}
\right)
\left(
\begin{matrix}
\sfy_1\\
\sfy_2\\
\sfy_3\\
\sfx_1\\
\sfx_2
\end{matrix}
\right)
=
\left(
\begin{matrix}
\sfzero\\
\sfzero\\
\sfzero\\
-\sff_1\\
-\sff_2
\end{matrix}
\right).
\]

Since $\sfA_{33}$ is invertible, it follows that $\sfy_3=0$ so that the
system can be reduced to the following compact matrix form
\[ 
\left(
\begin{matrix}
\sfA_{11}&\sfzero&\sfB_{11}^\top&\sfzero\\
\sfzero&\sfA_{22}&\sfzero&\sfB_{22}^\top\\
\sfB_{11}&\sfzero&\sfzero&\sfzero\\
\sfzero&\sfB_{22}&\sfzero&\sfzero\\
\end{matrix}
\right)
\left(
\begin{matrix}
\sfy_1\\
\sfy_2\\
\sfx_1\\
\sfx_2
\end{matrix}
\right)
=
\left(
\begin{matrix}
\sfzero\\
\sfzero\\
-\sff_1\\
-\sff_2
\end{matrix}
\right)
\]
where all the non-vanishing blocks are square and diagonal.

It follows that the system decouples into the following two systems:
\[
\left(
\begin{matrix}
\sfA_{11}&\sfB_{11}^\top\\
\sfB_{11}&\sfzero
\end{matrix}
\right)
\left(
\begin{matrix}
\sfy_1\\
\sfx_1
\end{matrix}
\right)
=
\left(
\begin{matrix}
\sfzero\\
-\sff_1
\end{matrix}
\right)
\]
and
\[
\left(
\begin{matrix}
\sfA_{22}&\sfB_{22}^\top\\
\sfB_{22}&\sfzero
\end{matrix}
\right)
\left(
\begin{matrix}
\sfy_2\\
\sfx_2
\end{matrix}
\right)
=      
\left(
\begin{matrix}
\sfzero\\
-\sff_2
\end{matrix}
\right)
\]

This corresponds to the following two variational problems. Find
$\bfsigmaone\in\Vone$ and $\uone\in\Qone$ such that
\begin{equation}
\left\{
\aligned
&(\bfsigmaone,\bftau)+(\bftau,\grad\uone)=0&&\forall\bftau\in\Vone\\
&(\bfsigmaone,\grad v)=-\langle f,v\rangle&&\forall v\in\Qone
\endaligned
\right.
\label{eq:magic1}
\end{equation}
and find $\bfsigmatwo\in\Vtwo$ and $\utwo\in\Qtwo$ such that
\begin{equation}
\left\{
\aligned
&(\bfsigmatwo,\bftau)+(\bftau,\grad\utwo)=0&&\forall\bftau\in\Vtwo\\
&(\bfsigmatwo,\grad v)=-\langle f,v\rangle&&\forall v\in\Qtwo
\endaligned
\right.
\label{eq:magic2}
\end{equation} 

The solution of~\eqref{eq:dualmixedh} is then given simply by
\[
\bfsigma_h=\bfsigmaone+\bfsigmatwo\qquad u_h=\uone+\utwo.
\]

The heuristic idea of the splitting is that~\eqref{eq:magic1} corresponds to a
\emph{stable} problem, where the inf-sup constant is bounded below by
$(\mut)^{1/2}$, while~\eqref{eq:magic2} has an inf-sup constant that goes to
zero as $h$ goes to zero.

We state in the following proposition the results obtained so far.

\begin{proposition}

Let us consider the solution $(\bfsigma_h,u_h)\in\Vh\times\Qh$
of~\eqref{eq:dualmixedh}. Let $\mut$ be a constant between $0$ and $1$
and consider the subspaces $\Vone$, $\Vtwo$, $\Qone$, and $\Qtwo$ introduced
in~\eqref{eq:splitting}. Then it holds
\[
\aligned
\bfsigma_h=\bfsigmaone+\bfsigmatwo\\
u_h=\uone+\utwo,
\endaligned
\]
where $(\bfsigmaone,\uone)\in\Vone\times\Qone$ and
$(\bfsigmatwo,\utwo)\in\Vtwo\times\Qtwo$ solve~\eqref{eq:magic1}
and~\eqref{eq:magic2}, respectively.

\label{pr:splitting}
\end{proposition}

The characterization of the solution $u_h$ can be pushed further by looking
at the matrices involved with~\eqref{eq:magic1} and~\eqref{eq:magic2} and by
testing the systems with $v=u_{i,h}$. Indeed, it turns out that
$\sfA_{11}=-\sfB_{11}$ and $\sfA_{22}=-\sfB_{22}$ have diagonal entries equal
to the eigenvalues $\mu_i$ if the bases of the spaces are chosen as
in~\eqref{eq:splitting}. The following theorem, which is an immediate
consequence of the previous observations, shows how $u_h$ can be
represented in terms of $f$ and the basis of $\Qh$.

\begin{theorem}

The second component $u_h\in\Qh$ of the solution to~\eqref{eq:dualmixedh} can
be represented as
\begin{equation}
u_h=\sum_{i=1}^{N(h)}\frac{\alpha_i}{\mu_i}u_{i,h}
\label{eq:uh}
\end{equation}
where the coefficients $\alpha_i$ are defined as
\begin{equation}
\alpha_i=\langle f,u_{i,h}\rangle.
\label{eq:alpha}
\end{equation}
\label{th:usol}
\end{theorem}

\begin{remark}

It is interesting to observe that, thanks to the orthogonalities of the
$\{u_{i,h}\}$, the coefficients $\{\alpha_i\}$ in Theorem~\ref{th:usol} can be
used also to define the solution $u_G\in\Qh$ of the standard Galerkin method
\[
(\grad u_G,\grad v)=\langle f,v\rangle\quad\forall v\in\Qh
\]
as
\begin{equation}
u_G=\sum_{i=1}^{N(h)}\alpha_iu_{i,h}.
\label{eq:uG}
\end{equation}

It follows that $u_G$ and $u_h$ have similar representations and that the
$i$-th coefficient of their representations differs by a factor equal to
$1/\mu_i$.
\end{remark}

In light of the previous results, it is interesting to see how the
coefficients $\{\alpha_i\}$ behave. To this aim we consider different right
hand sides $f$ in the unit square and compare their behavior.

We start with a smooth $f$ equal to the first eigenfunction of the Poisson
problem, then we take a constant right and side and we conclude with two
approximations of the Dirac delta function: one centered at $(1/3,1/5)$ and
the other centered at the center of the domain. The corresponding results are
reported in Figure~\ref{fg:alpha}.

\begin{figure}
\subcaptionbox*{$f(x,y)=2\pi^2\sin(\pi x)\sin(\pi y)$: \meshcrossed mesh
(left), \meshright mesh (middle), \meshnonstructured mesh (right)}
{
\includegraphics[width=.3\textwidth]{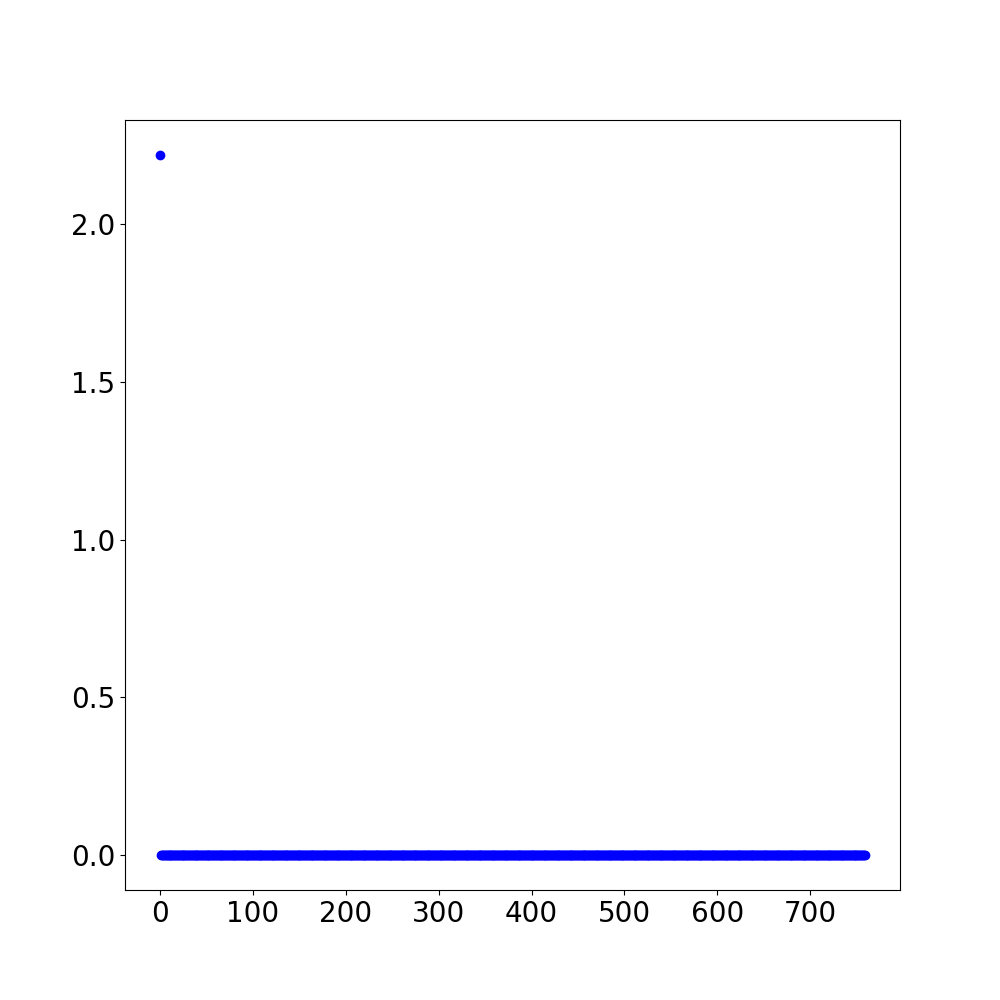}
\includegraphics[width=.3\textwidth]{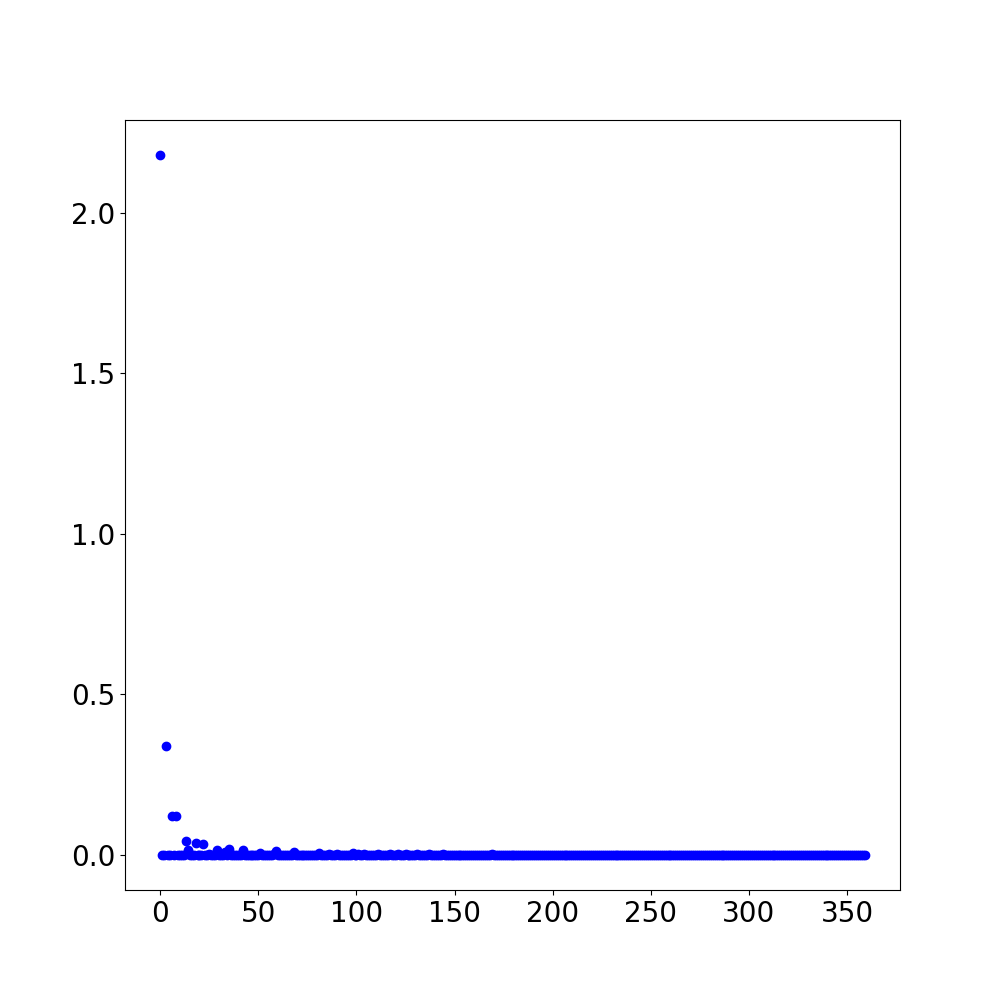}
\includegraphics[width=.3\textwidth]{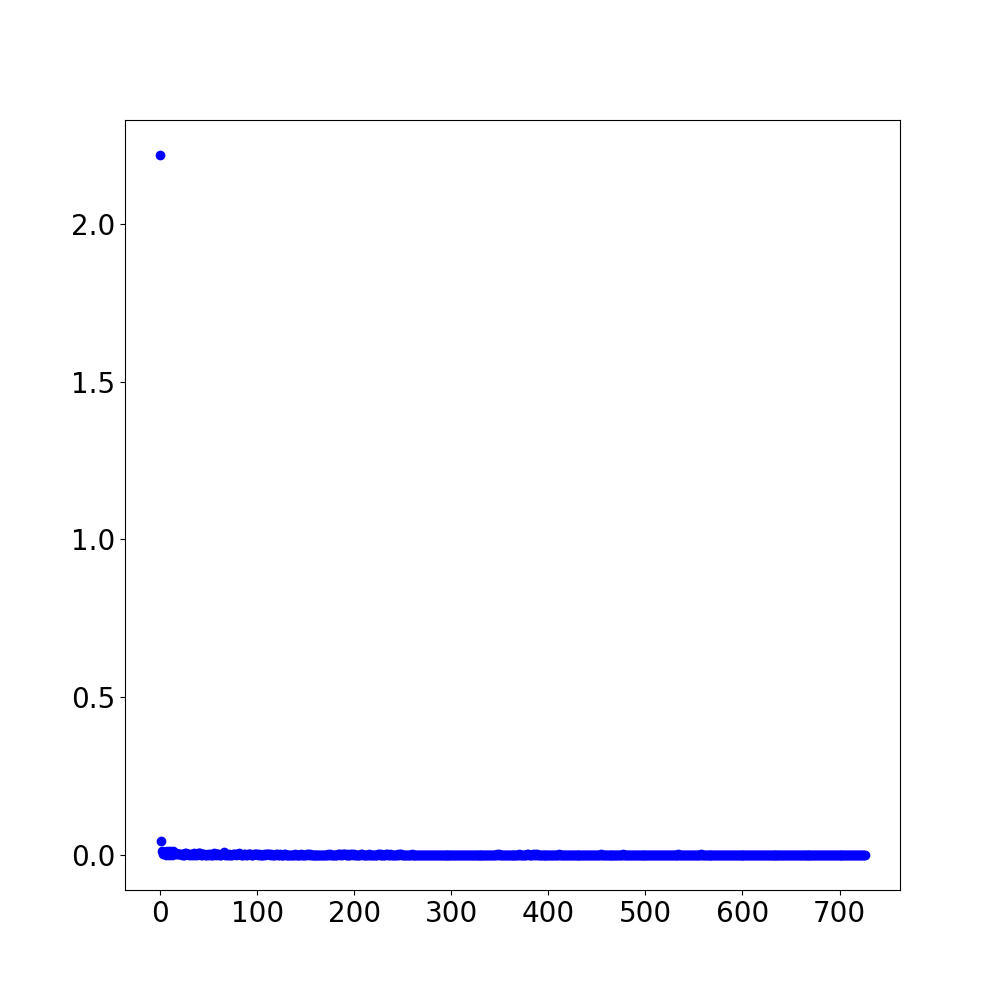}
}

\subcaptionbox*{$f(x,y)=1$: \meshcrossed mesh (left), \meshright mesh
(middle), \meshnonstructured mesh (right)}
{
\includegraphics[width=.3\textwidth]{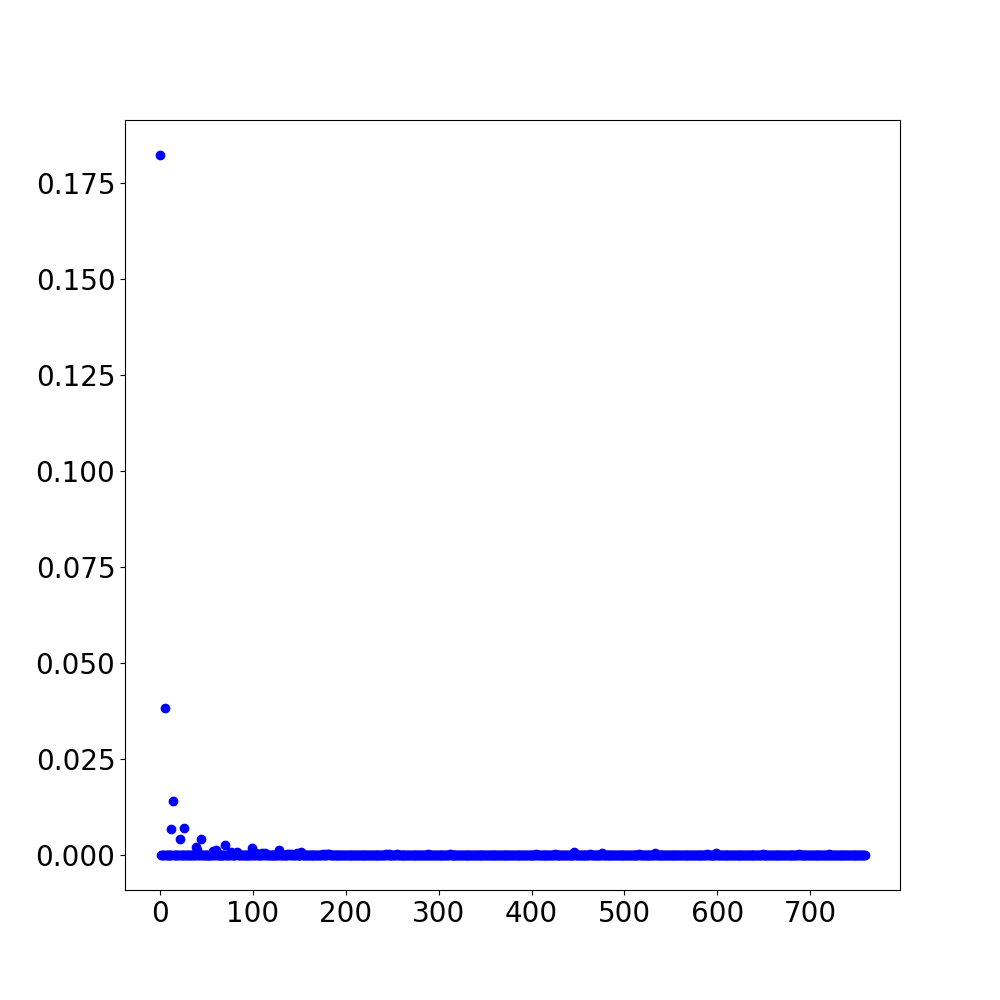}
\includegraphics[width=.3\textwidth]{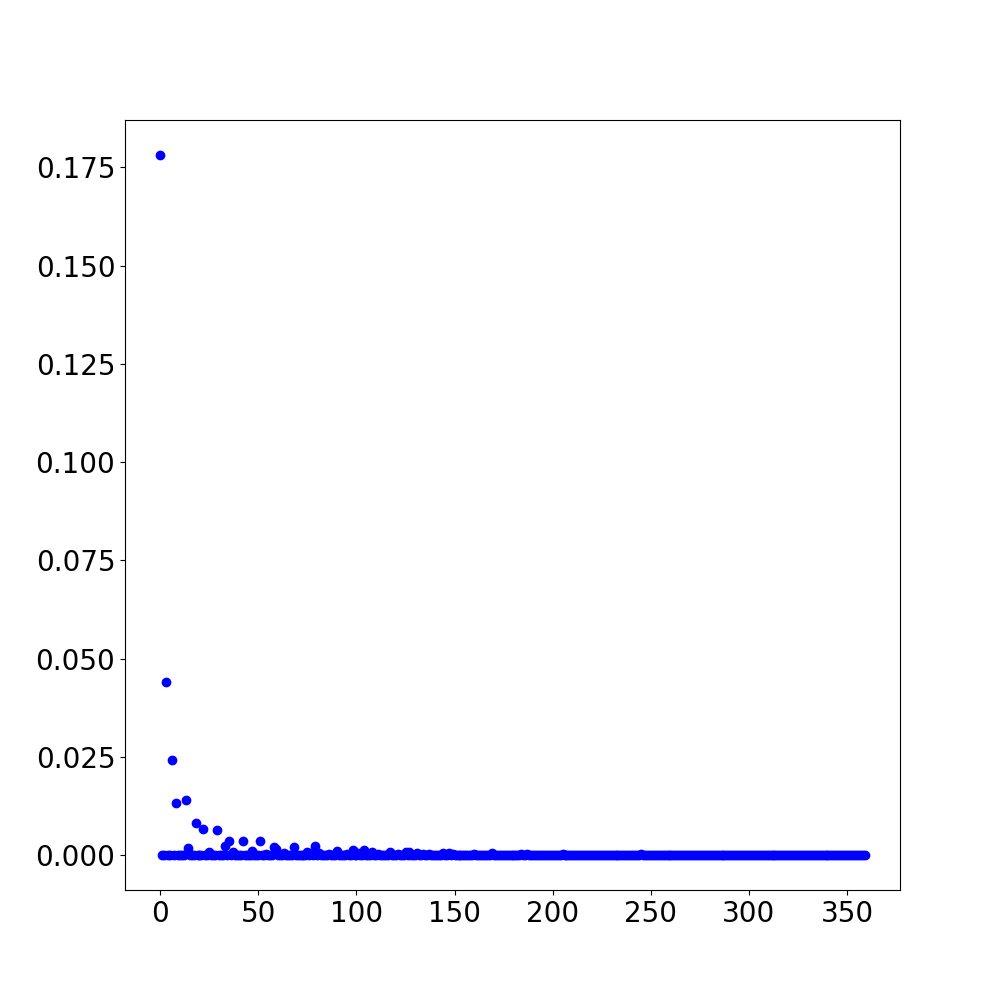}
\includegraphics[width=.3\textwidth]{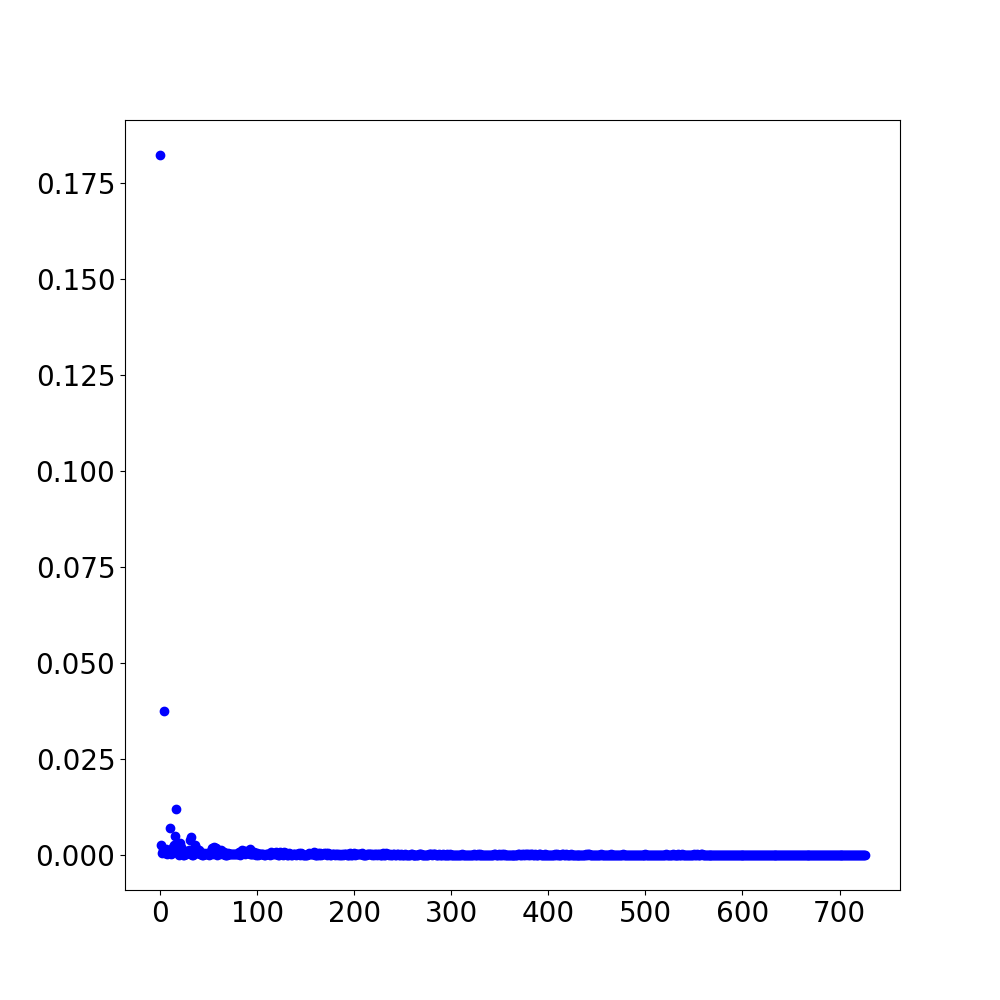}
}

\subcaptionbox*{$f(x,y)$ equal to an approximation to the Dirac delta function
centered at $(1/3,1/5)$: \meshcrossed mesh (left), \meshright mesh (middle),
\meshnonstructured mesh (right)}
{
\includegraphics[width=.3\textwidth]{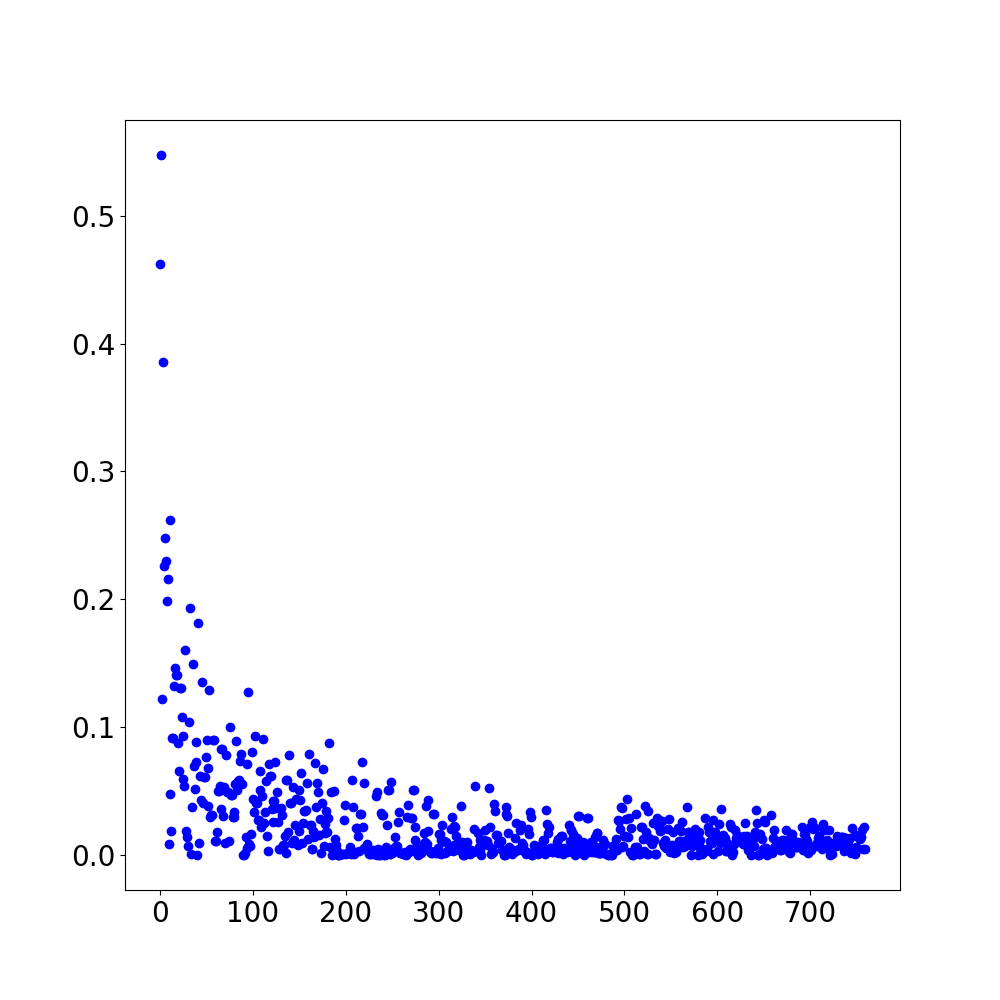}
\includegraphics[width=.3\textwidth]{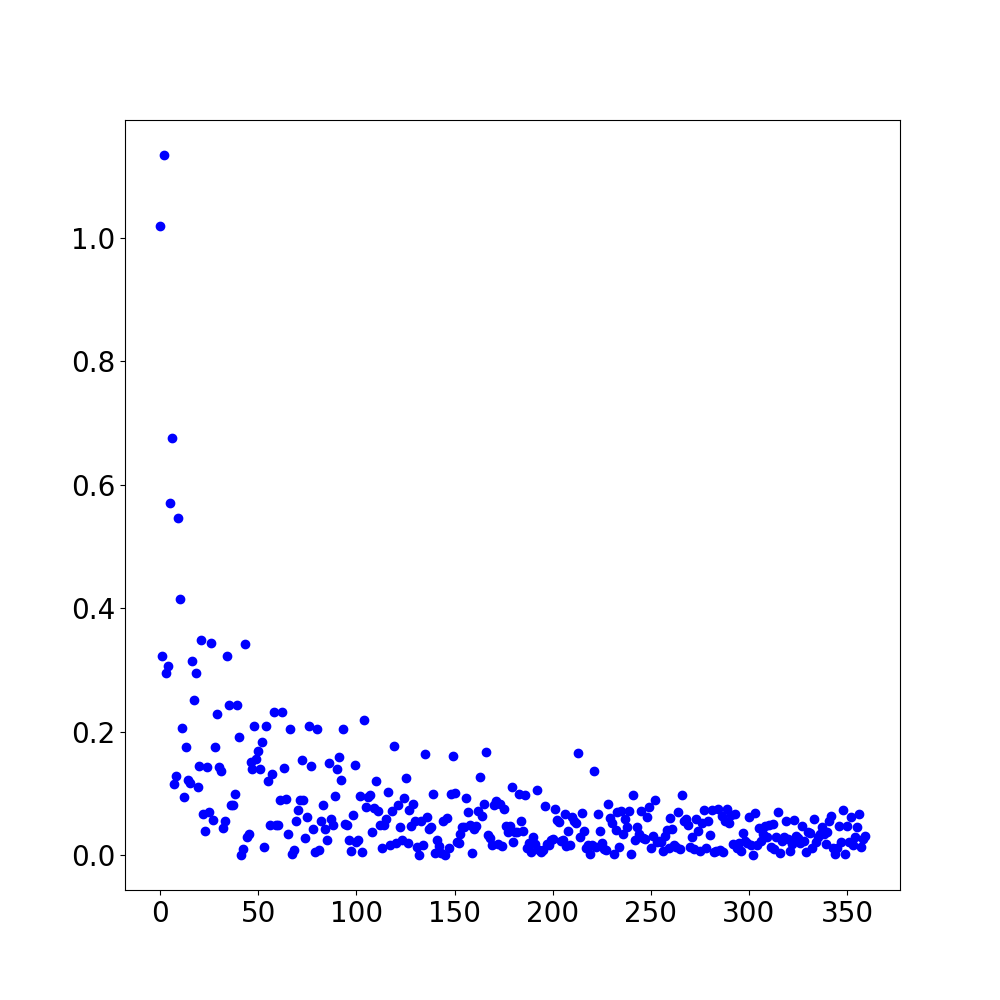}
\includegraphics[width=.3\textwidth]{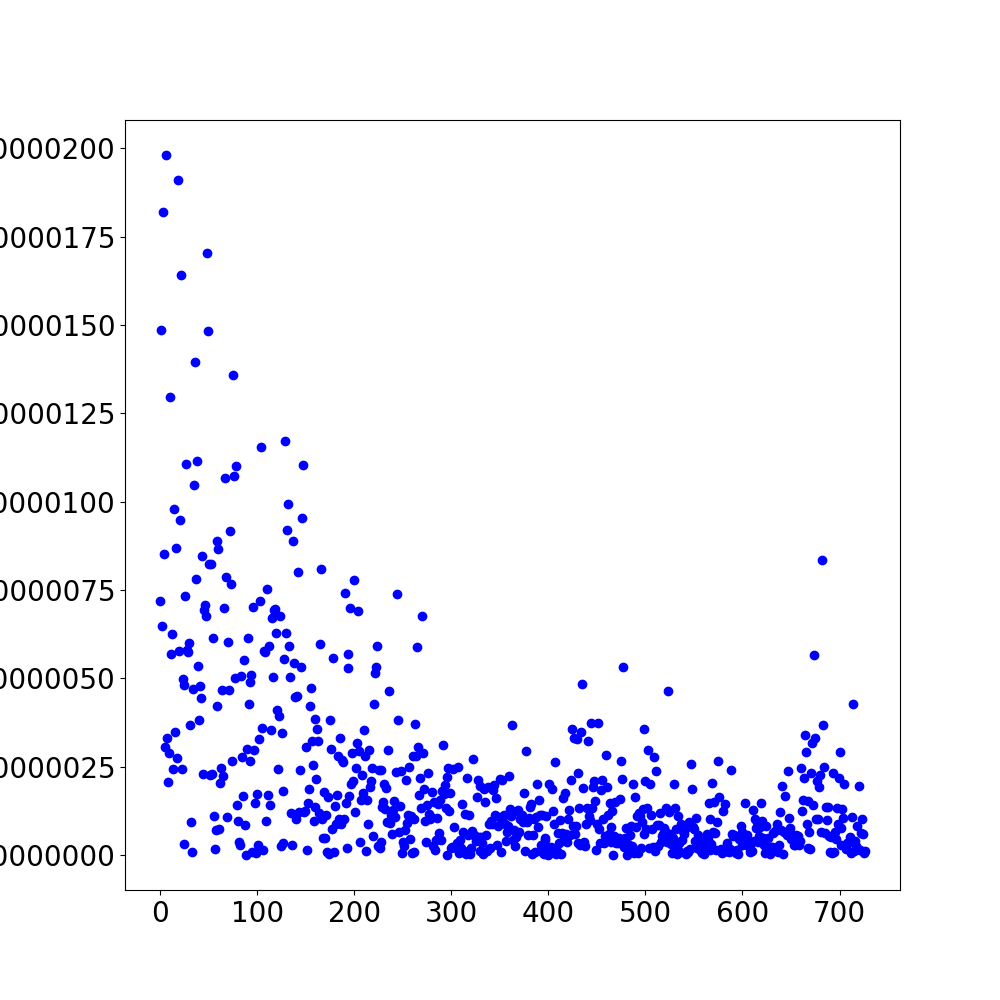}
}

\subcaptionbox*{$f(x,y)$ equal to an approximation to the Dirac delta function
centered at $(1/2,1/2)$: \meshcrossed mesh (left), \meshright mesh (middle),
\meshnonstructured mesh (right)}
{
\includegraphics[width=.3\textwidth]{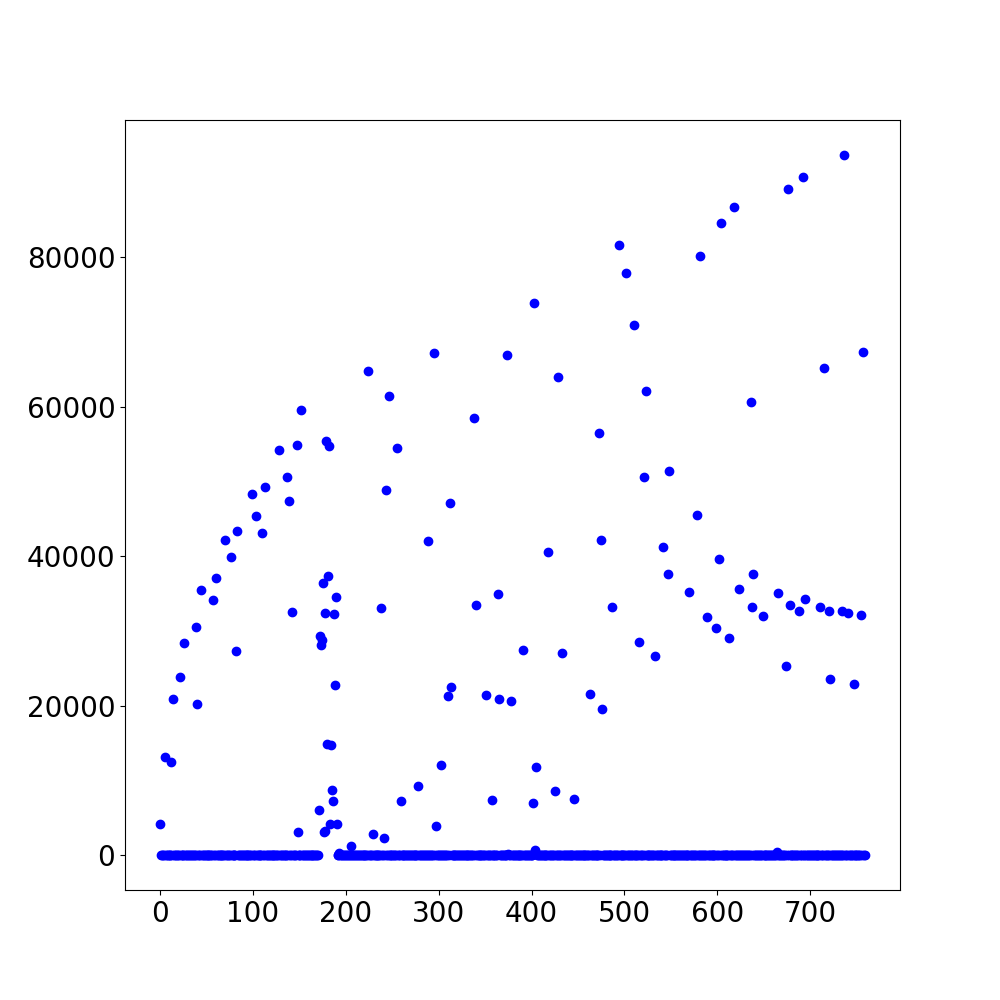}
\includegraphics[width=.3\textwidth]{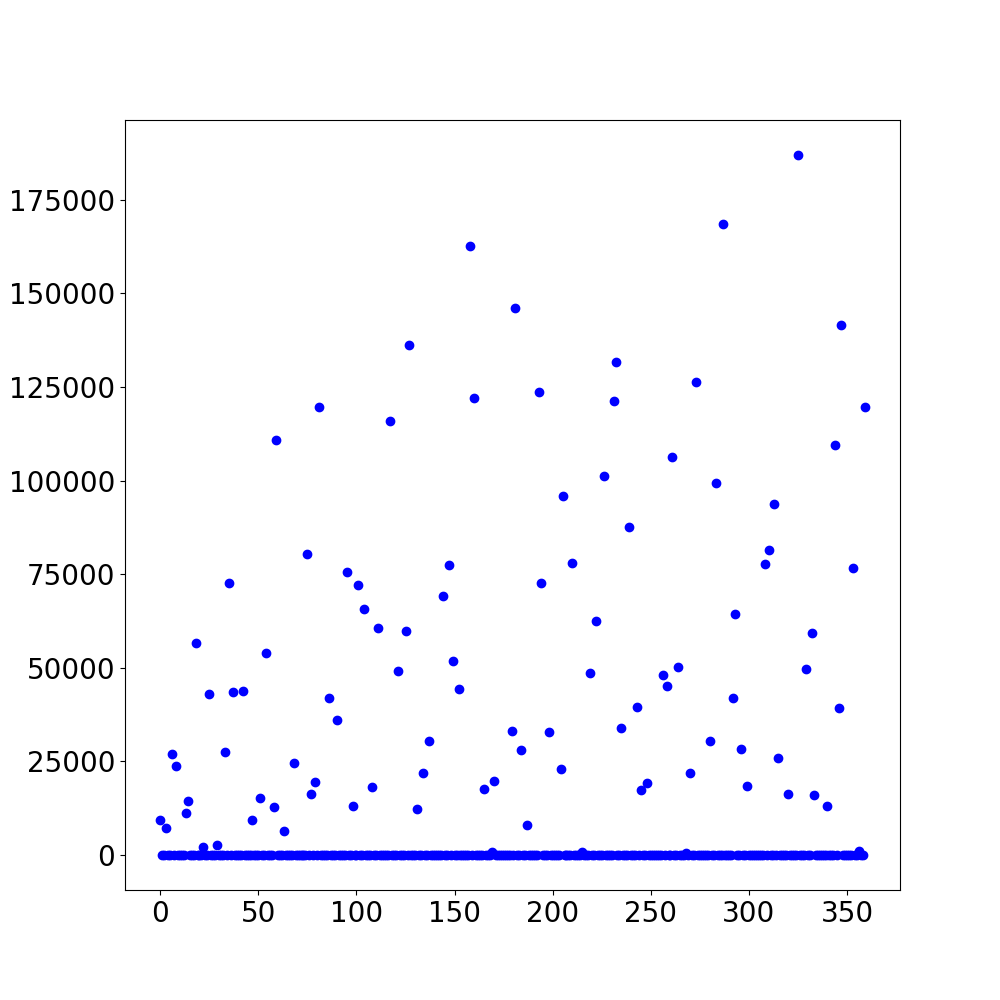}
\includegraphics[width=.3\textwidth]{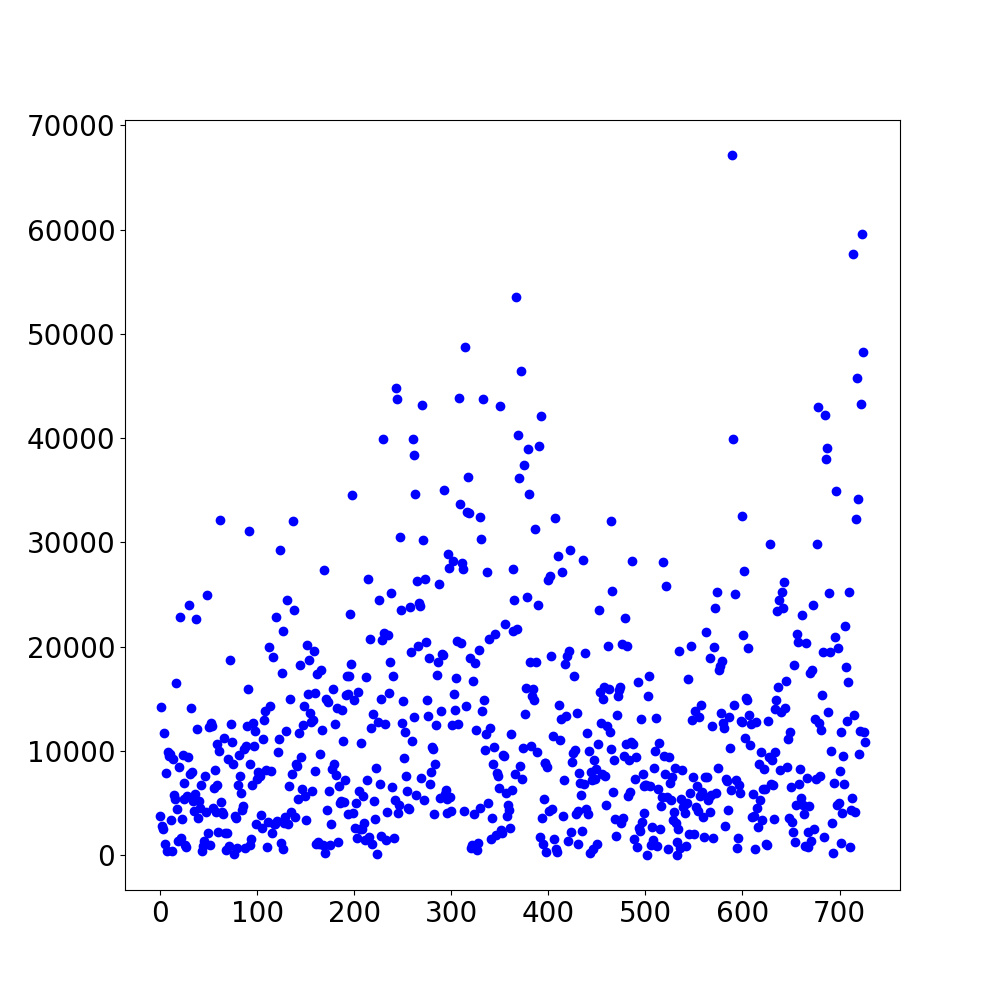}
}

\caption{Coefficients $\{\alpha_i\}$ from~\eqref{eq:alpha} for various choices
of $f$}
\label{fg:alpha}
\end{figure}

As it is natural from the oscillatory behavior of the eigenfunctions, the
coefficients $\alpha_i$ corresponding to smooth functions are different from
zero only for few low values of $i$ and vanish for larger $i$'s, while non
smooth functions contain nonzero coefficients in the right most part of the
spectrum. Moreover, the eigenvalues $\mu_i$ are close to one for small values
of $i$ while become smaller as we move to the right in the spectrum.

Hence, by comparing the two representation of $u_h$ and $U_G$ in~\eqref{eq:uh}
and~\eqref{eq:uG}, it should be expected that their difference is small when
$f$ is smooth and possibly large when $f$ contains large components along the
last part of the spectrum. Indeed in the next section we are going to analyze
the convergence of the scheme when $f$ is smooth enough.

\section{Convergence of $\|\bfsigma-\bfsigma_h\|$ for smooth data}
\label{se:convergence}

Let $\Po$ be the finite element space of piecewise constant functions. We
introduce the following two subspaces of $\Vh$:
\[
\aligned
&Z_h^m(f)=\{\bftau_h\in\Vh:(\div\bftau_h,q_h)=(f,q_h)\ \forall q_h\in\Po\}\\
&Z_h(f)=\{\bftau_h\in\Vh:(\bftau_h,\grad v_h)=-(f,v_h)\ \forall v_h\in\Qh\}.
\endaligned
\]

We have the following crucial result.

\begin{lemma}
If $f$ belongs to $\Po$, then $Z_h^m(f)\subset Z_h(f)$. Moreover, if
$\bfsigma\in\Hdiv$ satisfies $\div\bfsigma=f$, then
\[
\inf_{\bftau_h\in Z_h(f)}\|\bfsigma-\bftau_h\|_{L^2(\Omega)}\le C
\|\bfsigma-\bfsigma^I\|_{L^2(\Omega)},
\]
where $\bfsigma^I\in\Vh$ is the interpolant of $\bfsigma$.
\label{le:crucial}
\end{lemma}

\begin{proof}
If $\div\bfsigma$ is in $\Po$, then $\bfsigma$ is smooth enough so that its
interpolant $\bfsigma^I\in\Vh$ is well defined.
In this setting the definition of $Z_h(f)$ reads
\[
Z_h(f)=\{\bftau_h\in\Vh:(\div\bftau_h,\projo v_h)=(f,\projo v_h)\ \forall
v_h\in\Qh\},
\]
where $\projo$ denotes the $L^2(\Omega)$ projection onto $\Po$.
It follows that if $\bftau_h$ belongs to $Z_h^m(f)$ then it is also in
$Z_h(f)$.

From the standard inf-sup condition that is valid for the spaces $\Vh$ and
$\Po$, Proposition~5.1.3 of~\cite{bbf} gives
\[
\inf_{\bftau_h\in Z_h^m(f)}\|\bfsigma-\bftau_h\|_{\Hdiv}\le C
\inf_{\bftau_h\in\Vh}\|\bfsigma-\bftau_h\|_{\Hdiv},
\]
so that the inclusion $Z_h^m(f)\subset Z_h(f)$ gives, in particular,
\[
\inf_{\bftau_h\in Z_h(f)}\|\bfsigma-\bftau_h\|_{L^2(\Omega)}\le C
\inf_{\bftau_h\in\Vh}\|\bfsigma-\bftau_h\|_{\Hdiv}.
\]
From the commuting diagram property $\div\bfsigma^I=\projo\div\bfsigma$,
observing that $\div\bfsigma$ belongs to $\Po$, we have that
$\div(\bfsigma-\bfsigma^I)=0$, so that we can conclude that the right
hand side can be estimated by $\|\bfsigma-\bfsigma^I\|_{L^2(\Omega)}$.
\end{proof}

The following theorem shows the convergence of the approximation of
$\bfsigma$ given by~\eqref{eq:dualmixedh} in the case when $f$ belongs to
$\Po$.

\begin{theorem}

Let $\bfsigma\in\Hdiv$ be the first component of the solution
of~\eqref{eq:dualmixed} for $f\in\Po$ and $\bfsigma_h$ the corresponding
approximation given by~\eqref{eq:dualmixedh}. Then we have the estimate
\[
\|\bfsigma-\bfsigma_h\|_{L^2(\Omega)}\le
C\left(\|\bfsigma-\bfsigma^I\|_{L^2(\Omega)}+
\inf_{v_h\in\Qh}\|\grad(u-v_h)\|_{L^2(\Omega)}\right),
\]
where $\bfsigma^I\in\Vh$ is the interpolant of $\bfsigma$.
\end{theorem}
\begin{proof}
We estimate $\|\bfsigma_h-\bftau_h\|_{L^2(\Omega)}$ for $\bftau_h\in Z_h(f)$.
The result will then follow from the triangular inequality and
Lemma~\ref{le:crucial}.

From the error equation and the properties of $Z_h(f)$ we have for all
$v_h\in\Qh$
\[
\aligned
\|\bfsigma_h-\bftau_h\|_{L^2(\Omega)}^2&=(\bfsigma_h-\bfsigma,\bfsigma_h-\bftau_h)+
(\bfsigma-\bftau_h,\bfsigma_h-\bftau_h)\\
&=(\bfsigma_h-\bftau_h,\grad(u-u_h))+(\bfsigma-\bftau_h,\bfsigma_h-\bftau_h)\\
&=(\bfsigma_h-\bftau_h,\grad(u-v_h))+(\bfsigma-\bftau_h,\bfsigma_h-\bftau_h)\\
&\le\|\bfsigma_h-\bftau_h\|_{L^2(\Omega)}\left(\|\grad(u-v_h)\|_{L^2(\Omega)}+
\|\bfsigma-\bftau_h\|_{L^2(\Omega)}\right),
\endaligned
\]
from which we obtain the required estimate.
\end{proof}

\begin{corollary}

If $f\in H^1(\Omega)$ and if the Poisson problem has $H^{1+s}(\Omega)$
regularity for some $s\in(0,1]$, then the following optimal error estimate
holds true
\[
\|\bfsigma-\bfsigma_h\|_{L^2(\Omega)}\le C h^s\|f\|_{H^1(\Omega)}.
\]

\end{corollary}

\begin{proof}

The result follows by approximating $f$ with $\projo f$ and using the previous
theorem together with the standard estimate
\[
\|f-\projo f\|_{L^2(\Omega)}\le Ch\|\grad f\|_{L^2(\Omega)}.
\]

\end{proof}

\section{Numerical experiments}
\label{se:num}

In this section we report some numerical results related to the solution of
problem~\eqref{eq:dualmixedh}.
As investigated in Section~\ref{se:splitting} for the variable $u$ and proved
in Section~\ref{se:convergence} for the variable $\bfsigma$, we are expecting
the solution to be convergent when the right hand side $f$ is smooth enough.

\begin{figure}

\subcaptionbox*{$N=20$}
{
\includegraphics[width=5cm]{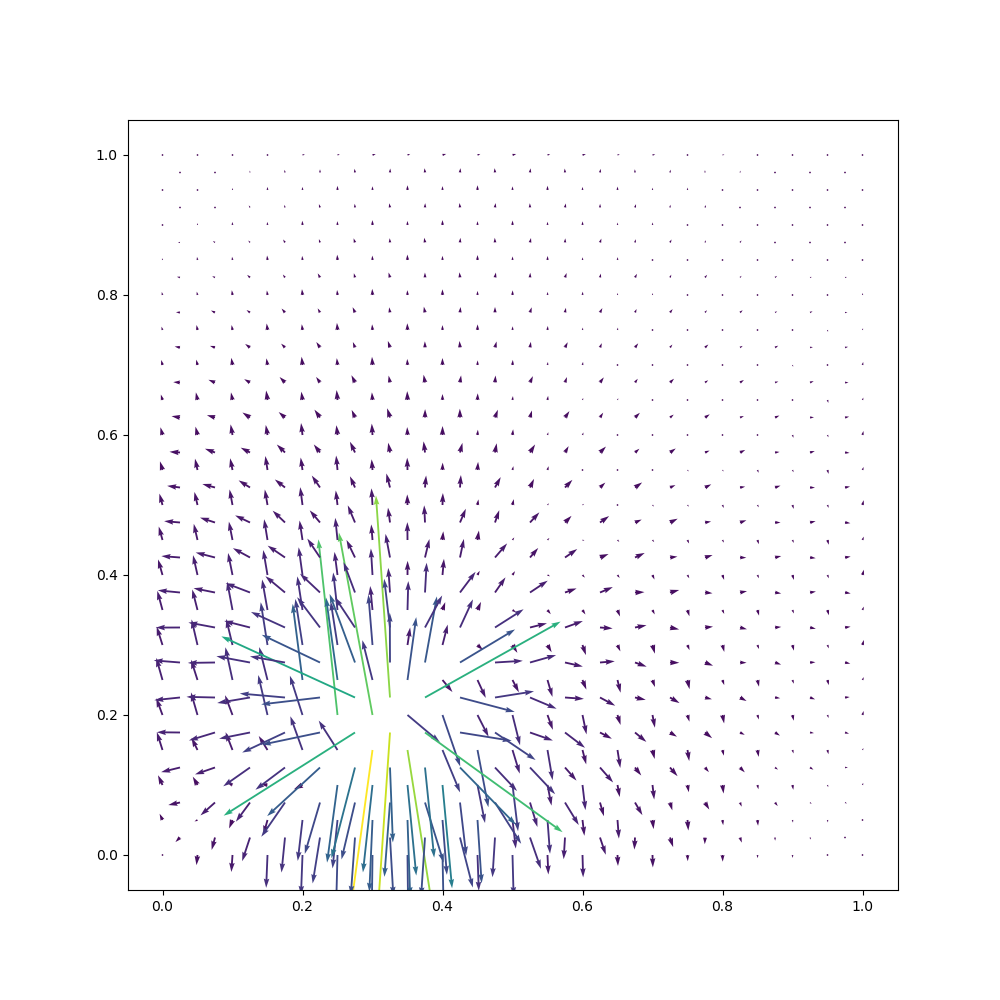}
\includegraphics[width=5cm]{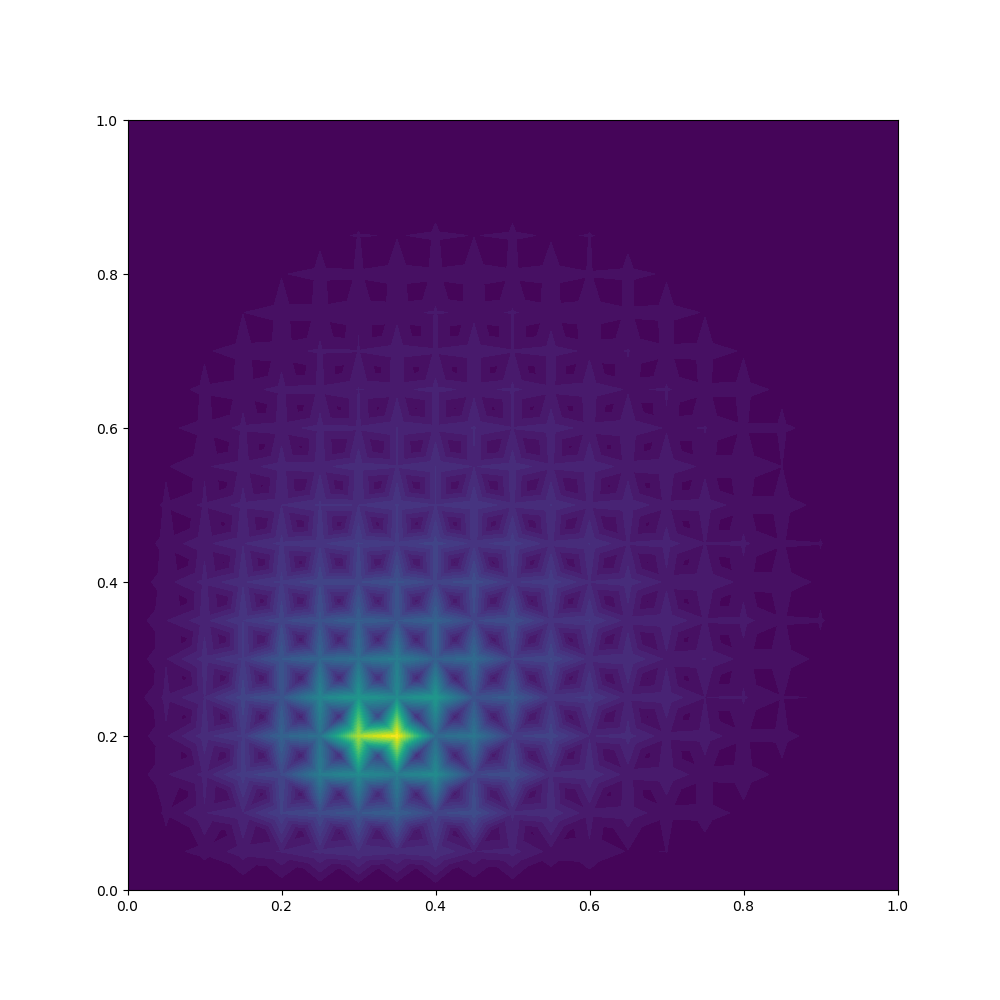}
}

\subcaptionbox*{$N=40$}
{
\includegraphics[width=5cm]{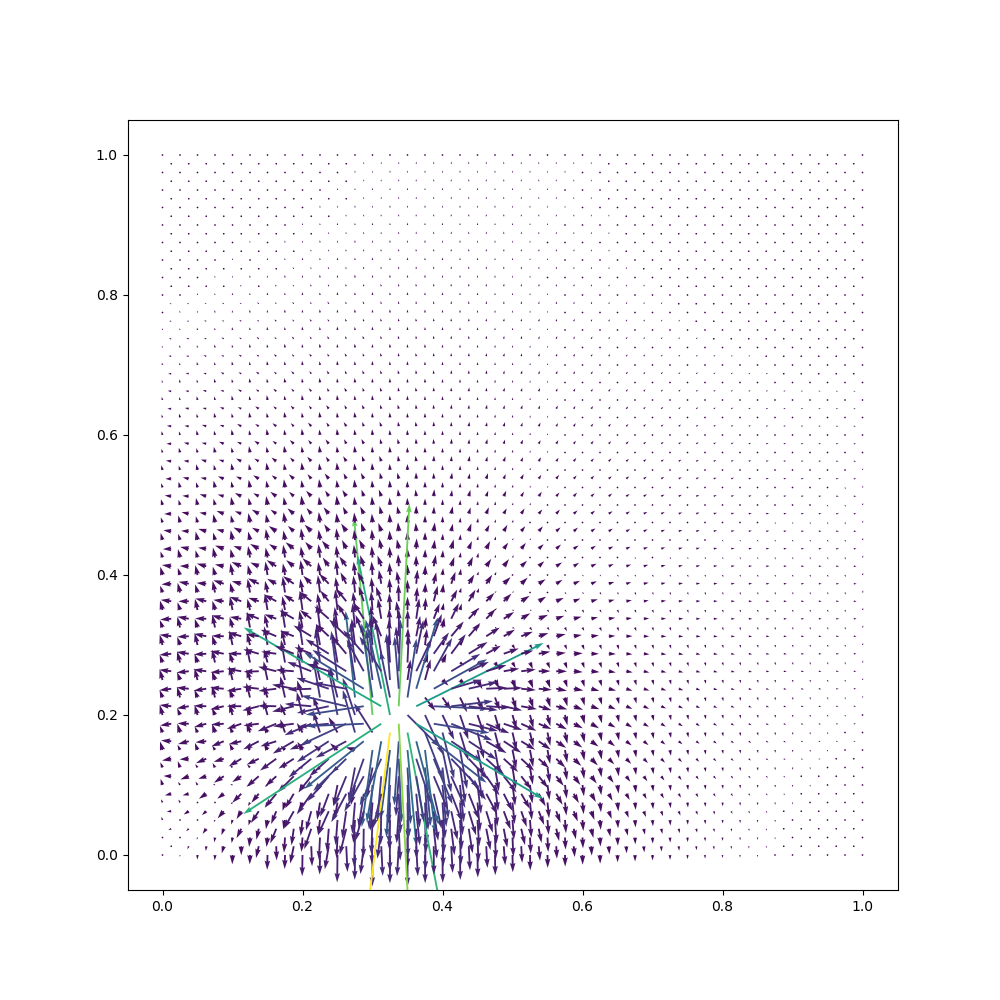}
\includegraphics[width=5cm]{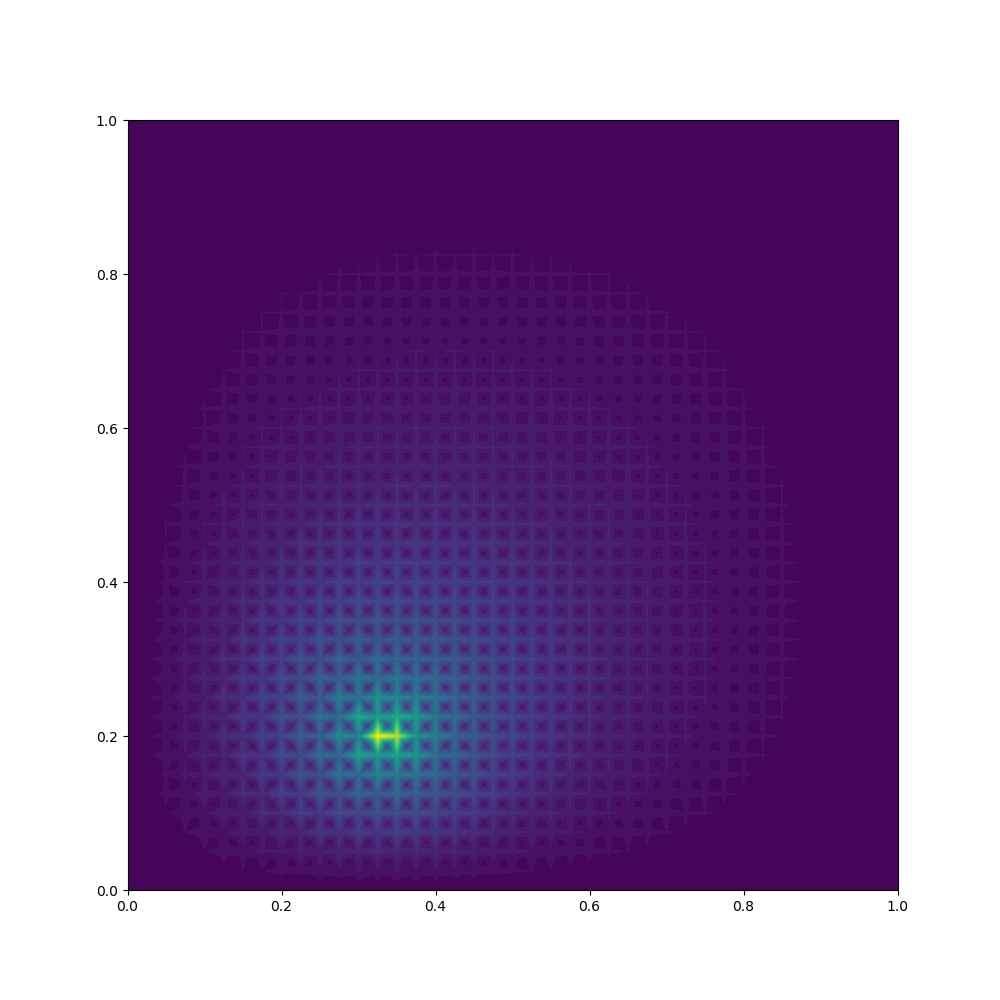}
}

\caption{Solution corresponding to $f$ equal to the approximation of the Dirac
delta function centered at $(1/3,1/5)$ on a coarser and a finer \meshcrossed
mesh}
\label{fg:deltacross}
\end{figure}

\begin{figure}

\subcaptionbox*{$N=20$}
{
\includegraphics[width=5cm]{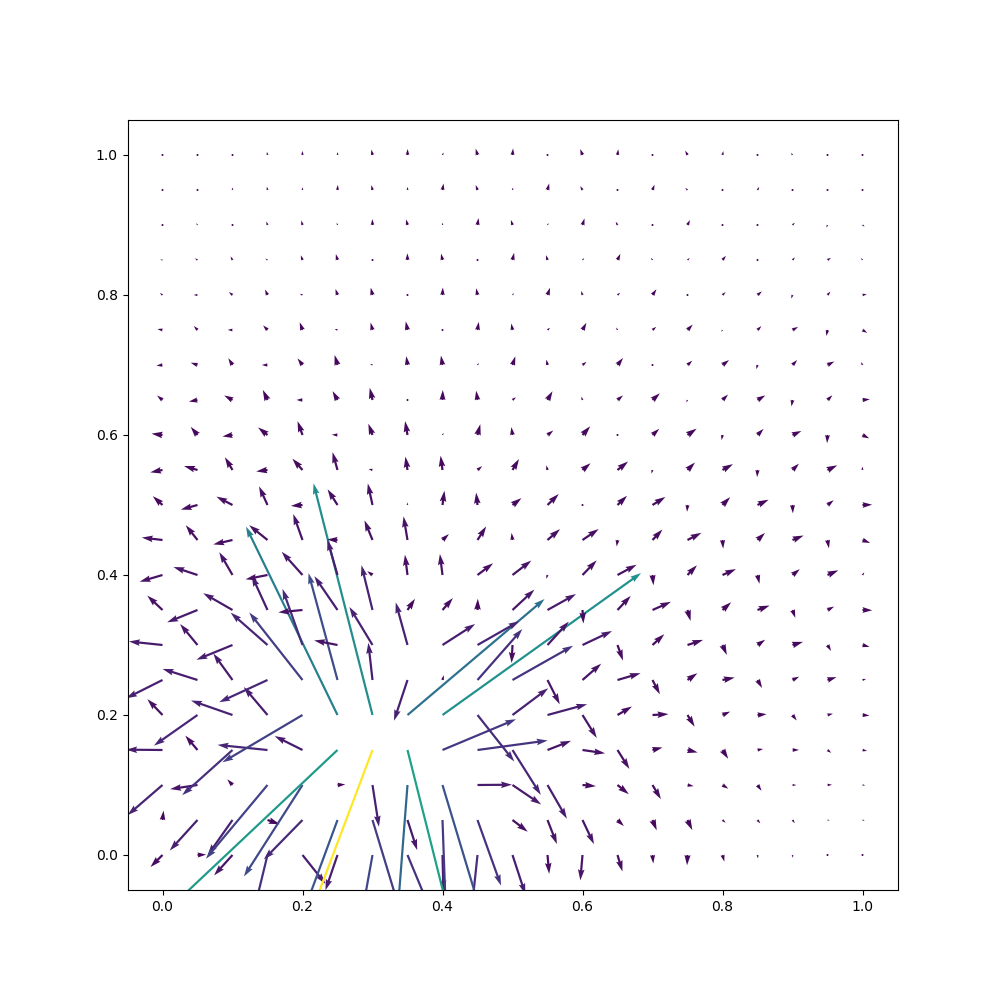}
\includegraphics[width=5cm]{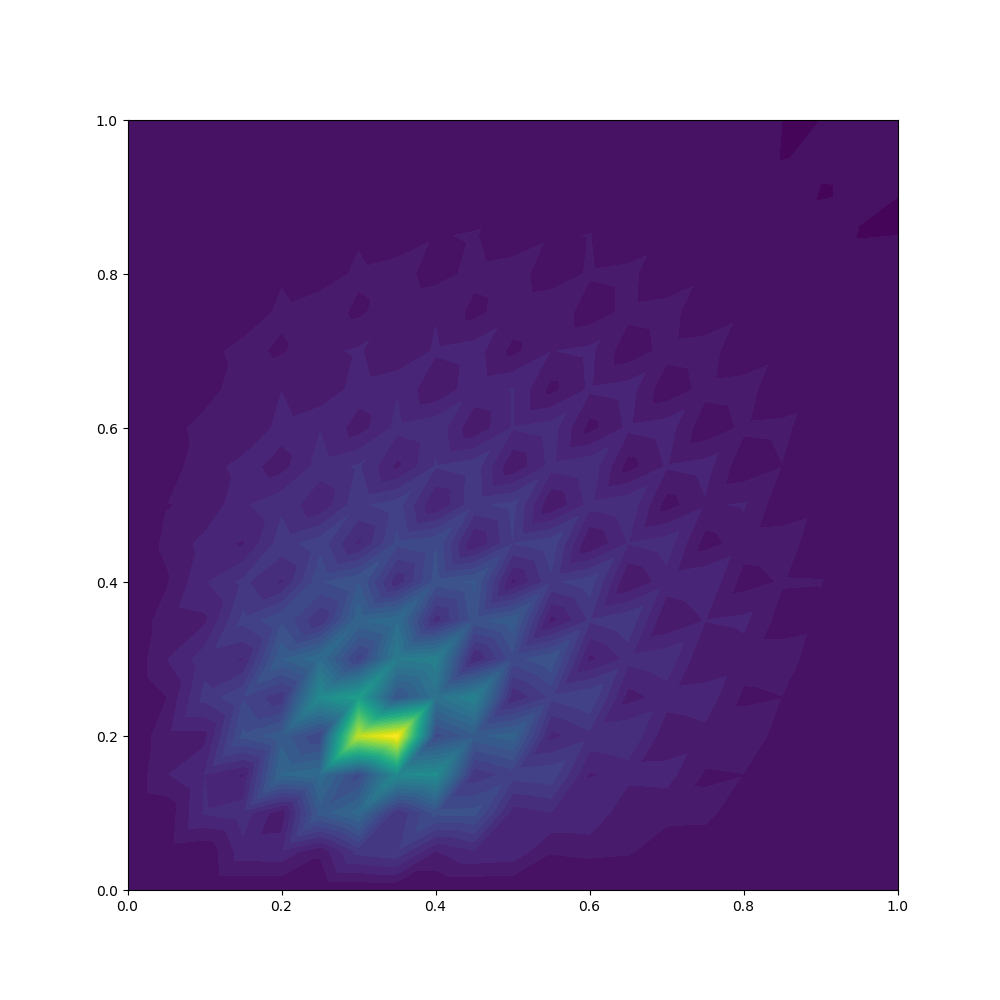}
}

\subcaptionbox*{$N=40$}
{
\includegraphics[width=5cm]{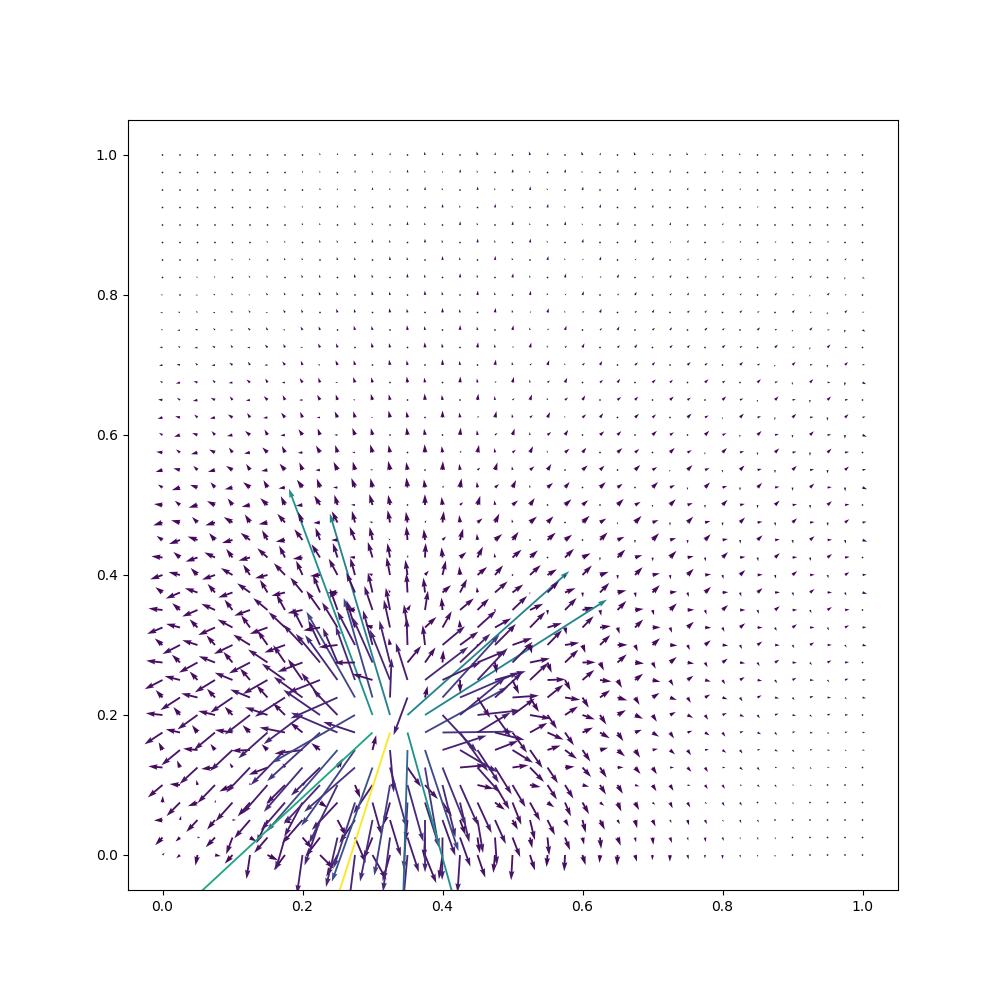}
\includegraphics[width=5cm]{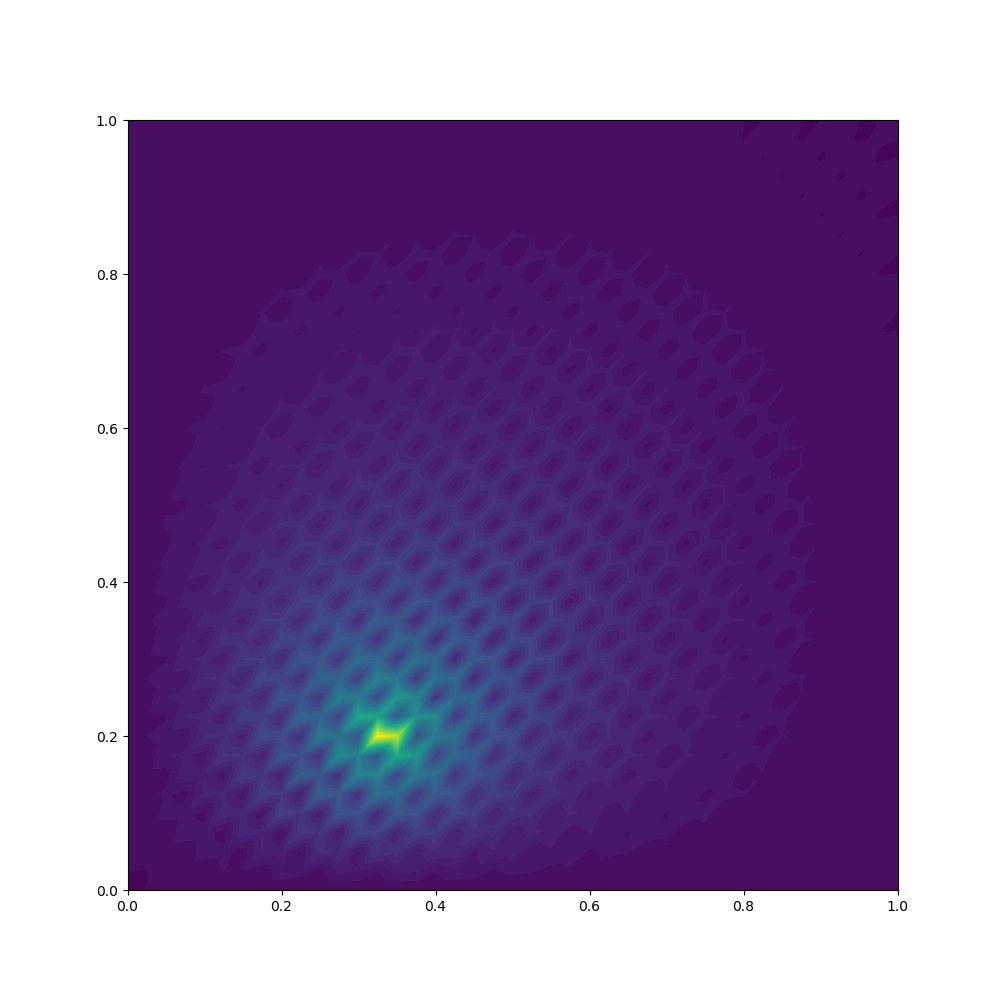}
}

\caption{Solution corresponding to $f$ equal to the approximation of the Dirac
delta function centered at $(1/3,1/5)$ on a coarser and a finer \meshright
mesh}
\label{fg:deltaright}
\end{figure}

\begin{figure}

\subcaptionbox*{$N=20$}
{
\includegraphics[width=5cm]{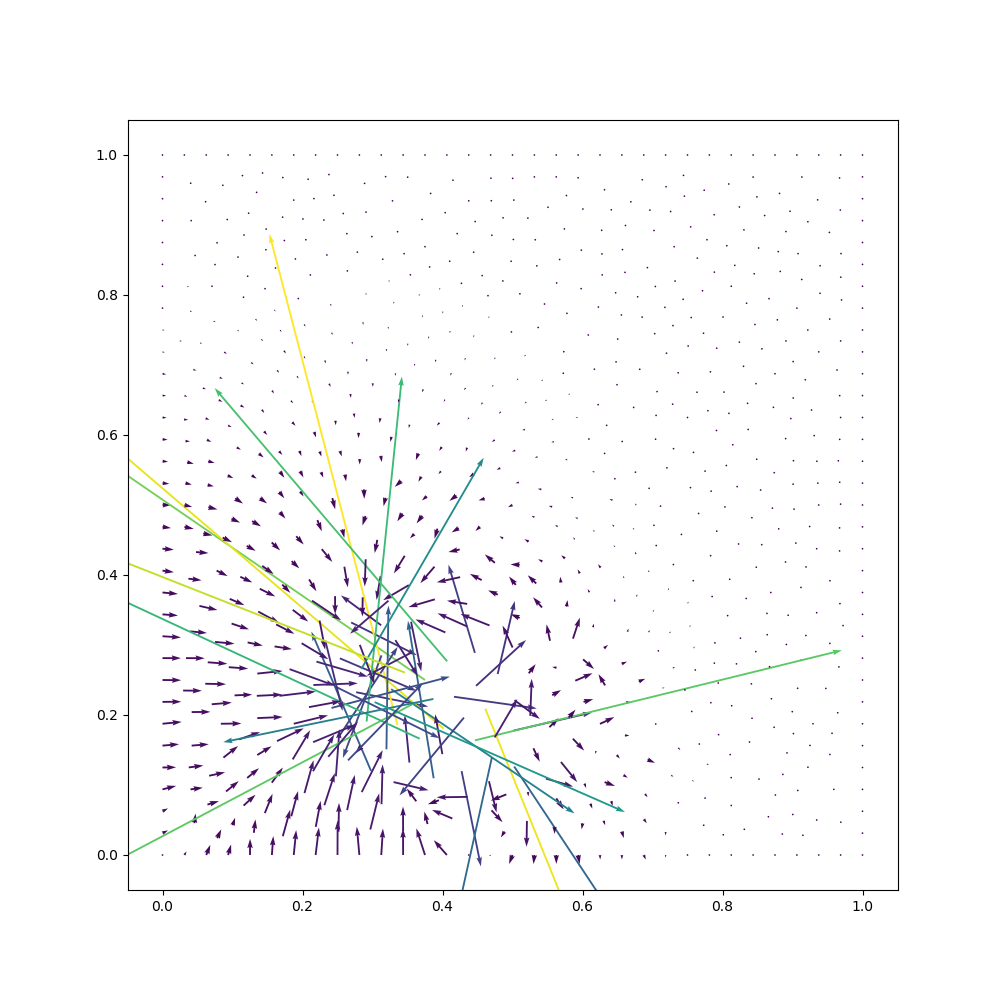}
\includegraphics[width=5cm]{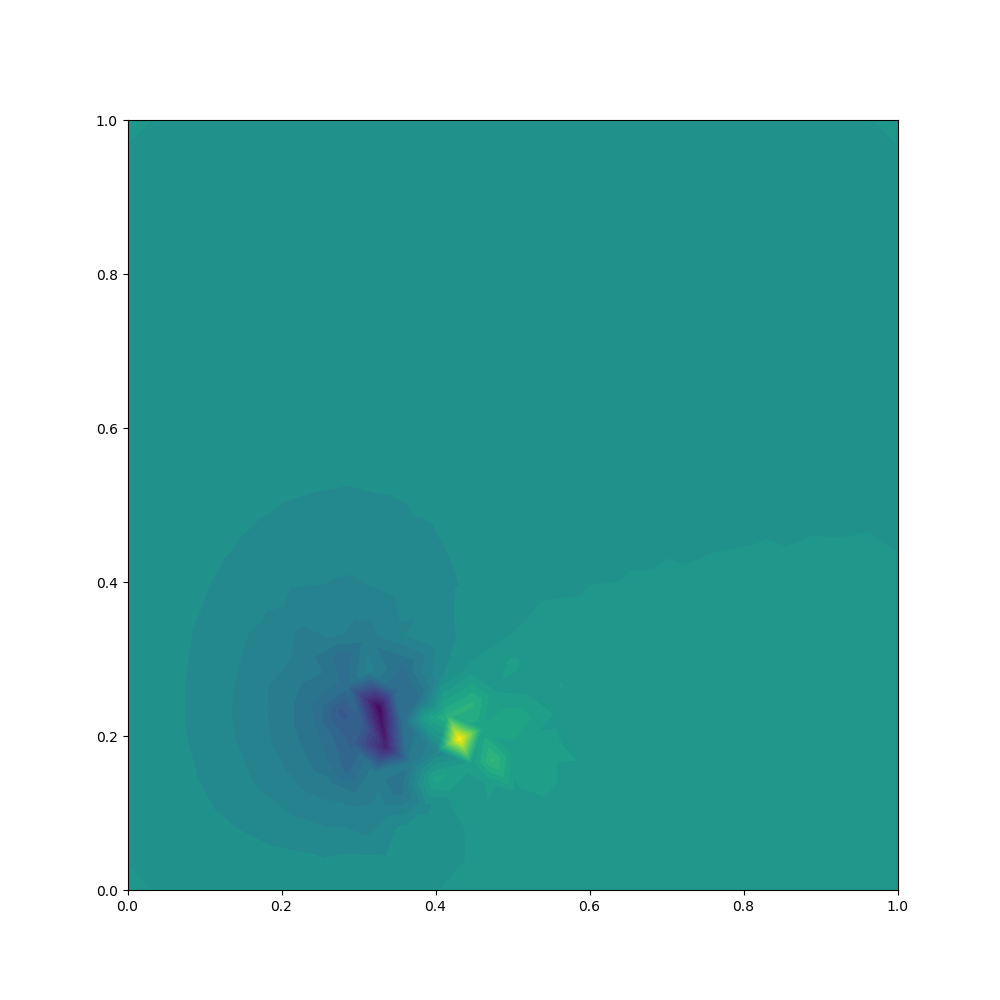}
}

\subcaptionbox*{$N=40$}
{
\includegraphics[width=5cm]{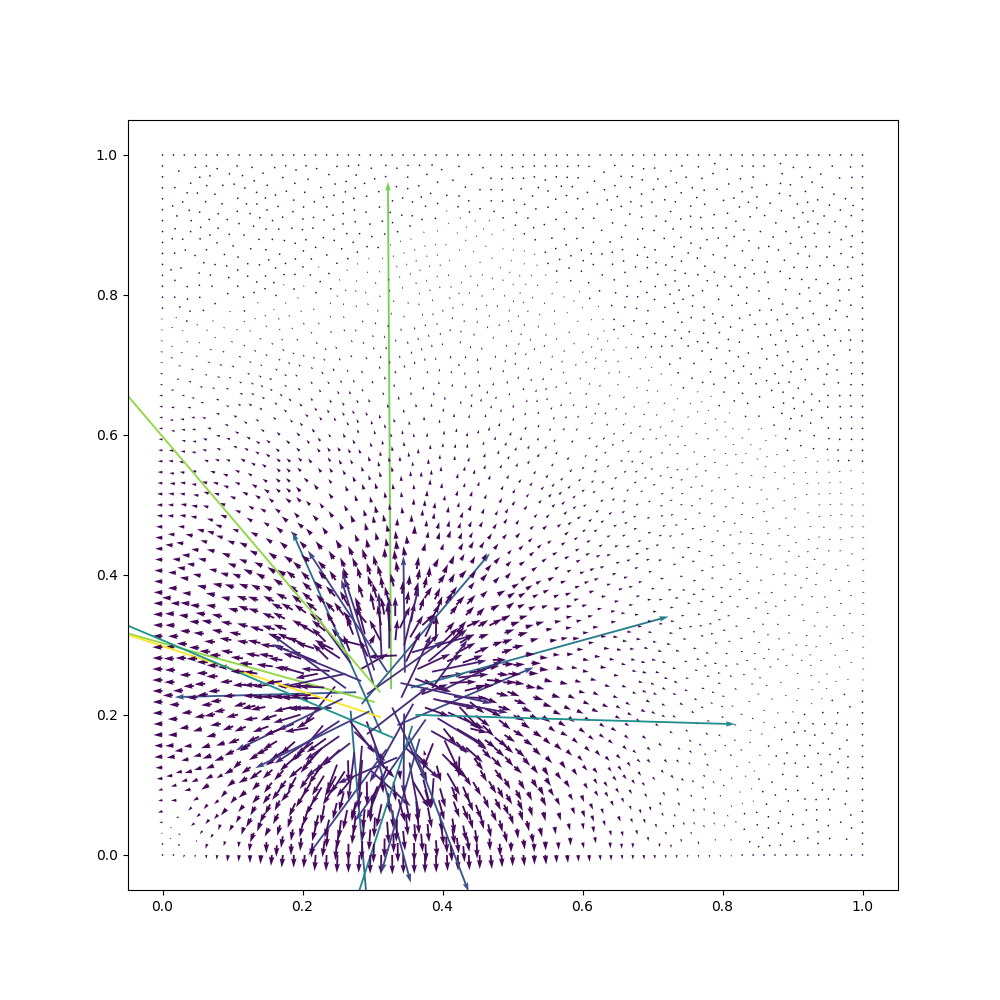}
\includegraphics[width=5cm]{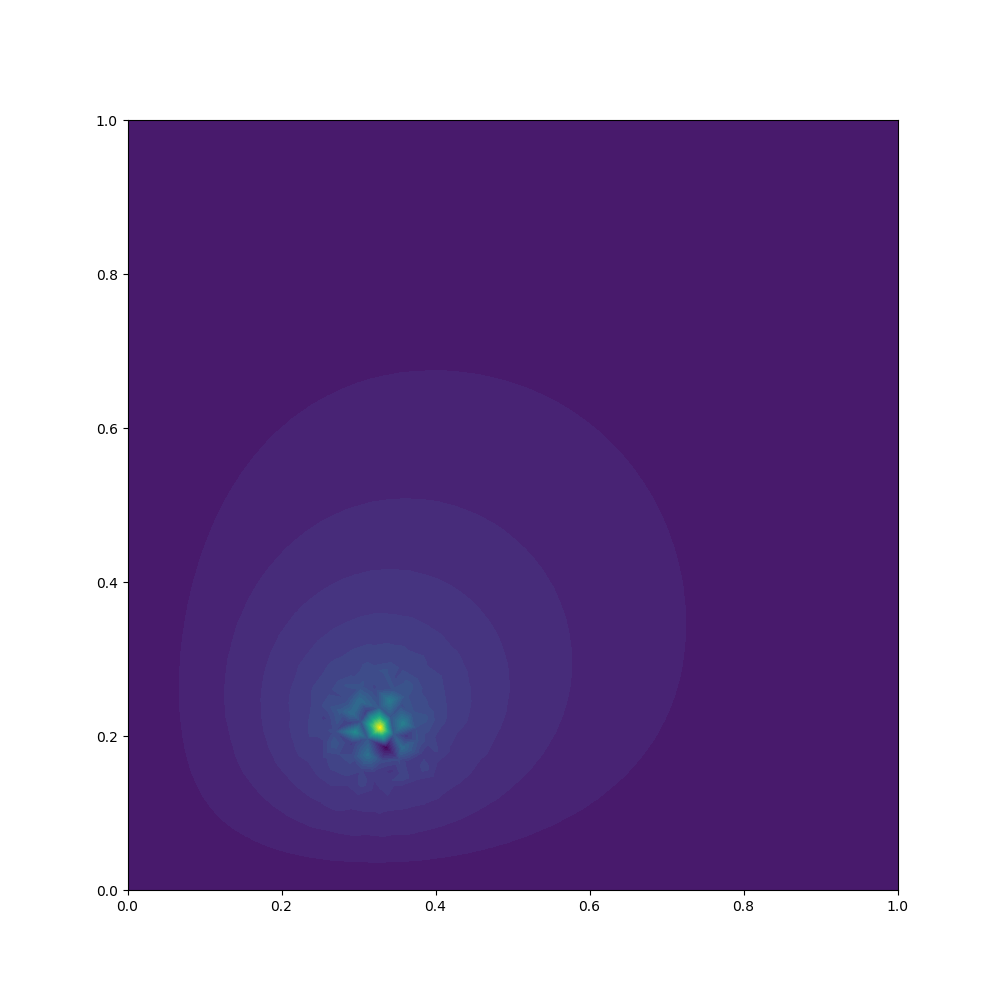}
}

\caption{Solution corresponding to $f$ equal to the approximation of the Dirac
delta function centered at $(1/3,1/5)$ on a coarser and a finer
\meshnonstructured mesh}
\label{fg:deltaunstr}
\end{figure}

We first consider the case of a singular $f$, equal to an approximation of the
Dirac delta function centered at $(1/3,1/5)$.
The results for the \meshcrossed, \meshright, and \meshnonstructured mesh
sequences are reported in Figures~\ref{fg:deltacross}, \ref{fg:deltaright},
and~\ref{fg:deltaunstr}, respectively.
It is clear that the solution is affected by spurious components that are
pretty much related to oscillating eigenfunctions like the one reported in
Figure~\ref{fg:infsup}.

\begin{figure}

\subcaptionbox*{$N=20$}
{
\includegraphics[width=5cm]{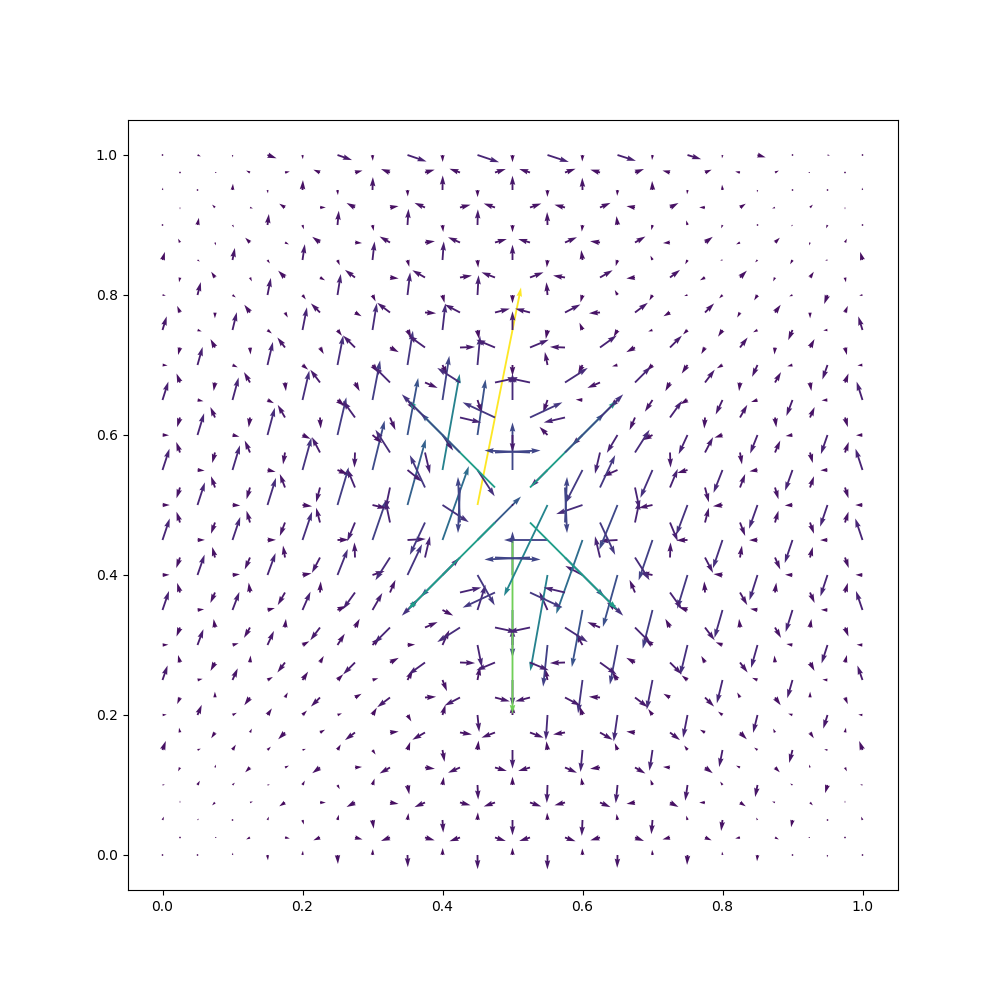}
\includegraphics[width=5cm]{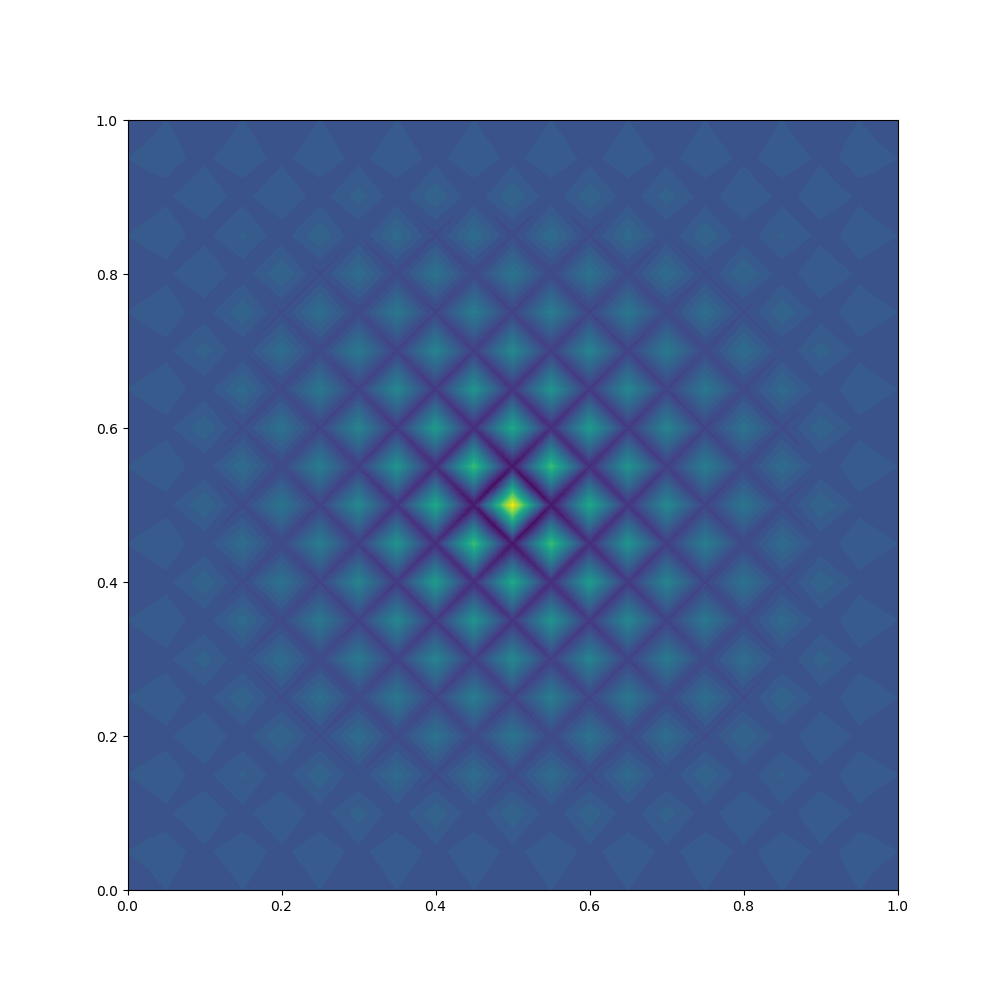}
}

\subcaptionbox*{$N=40$}
{
\includegraphics[width=5cm]{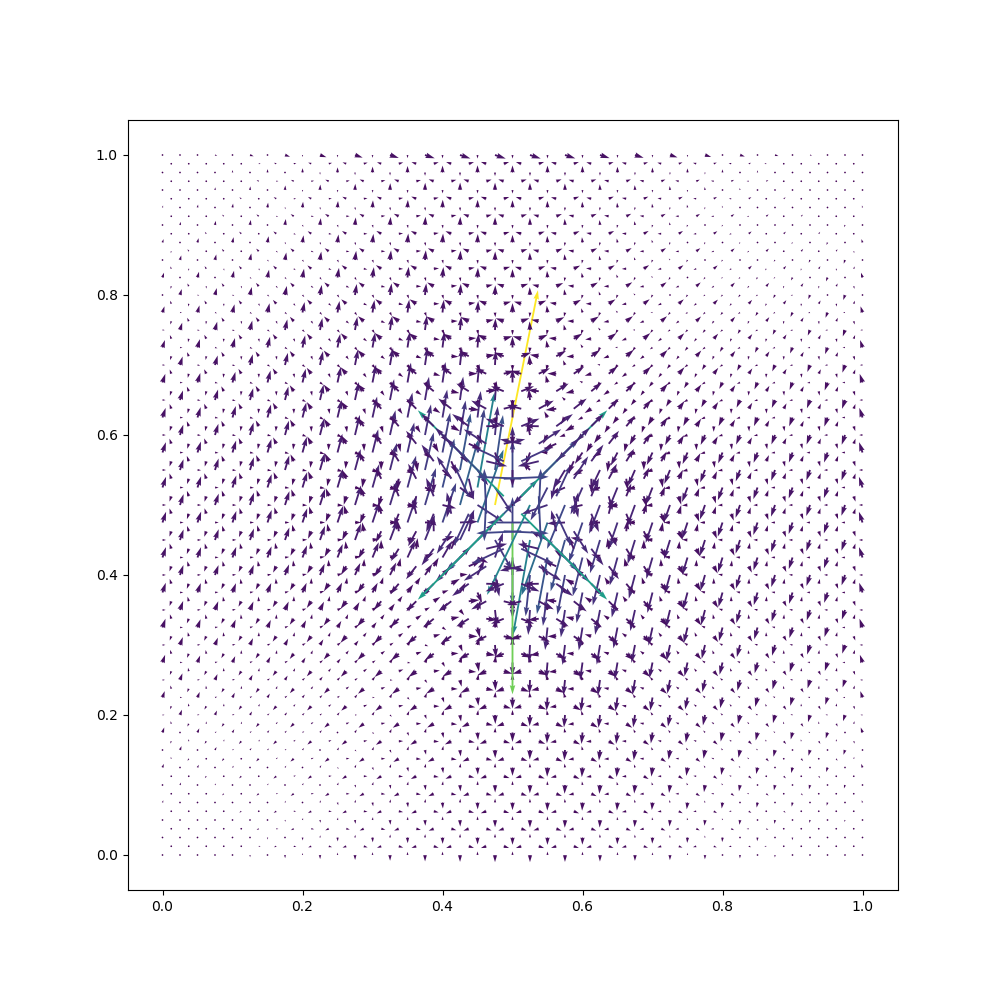}
\includegraphics[width=5cm]{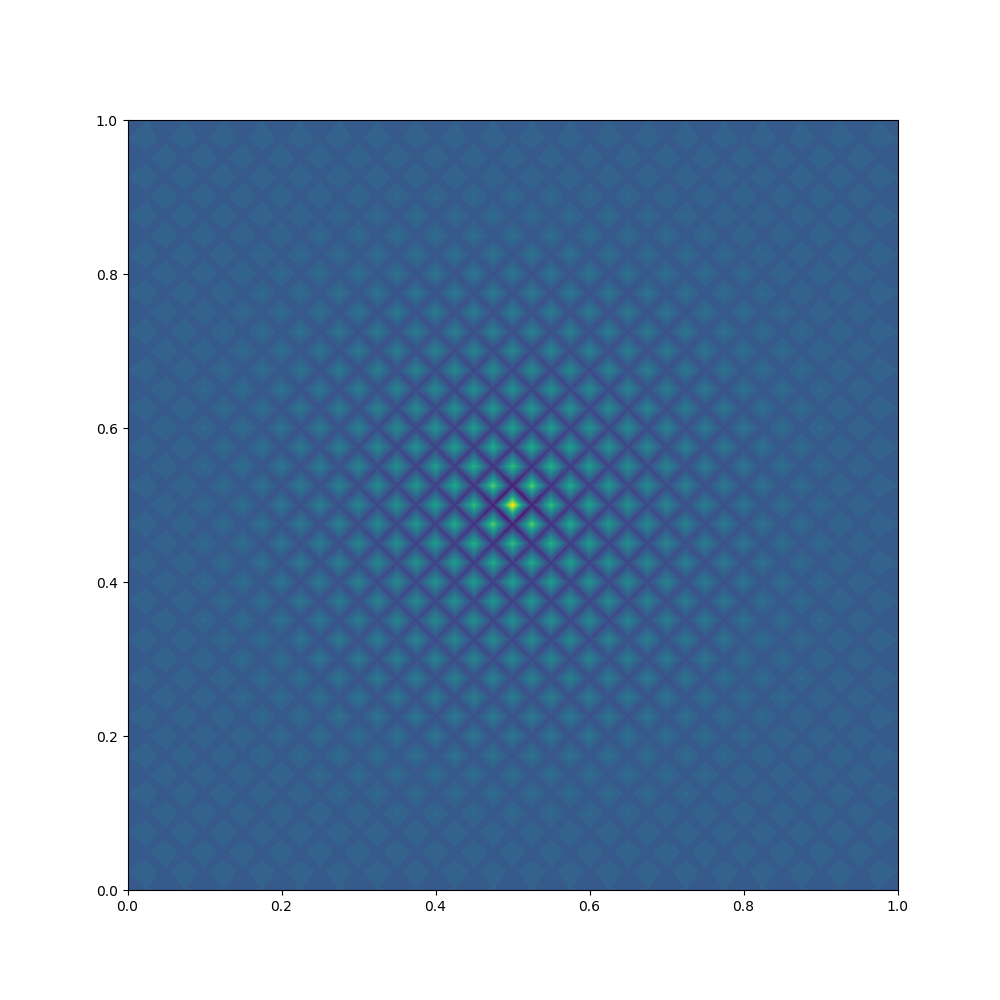}
}

\caption{Solution corresponding to $f$ equal to the approximation of the Dirac
delta function centered at $(1/2,1/2)$ on a coarser and a finer \meshcrossed
mesh}
\label{fg:deltaccross}
\end{figure}

\begin{figure}

\subcaptionbox*{$N=20$}
{
\includegraphics[width=5cm]{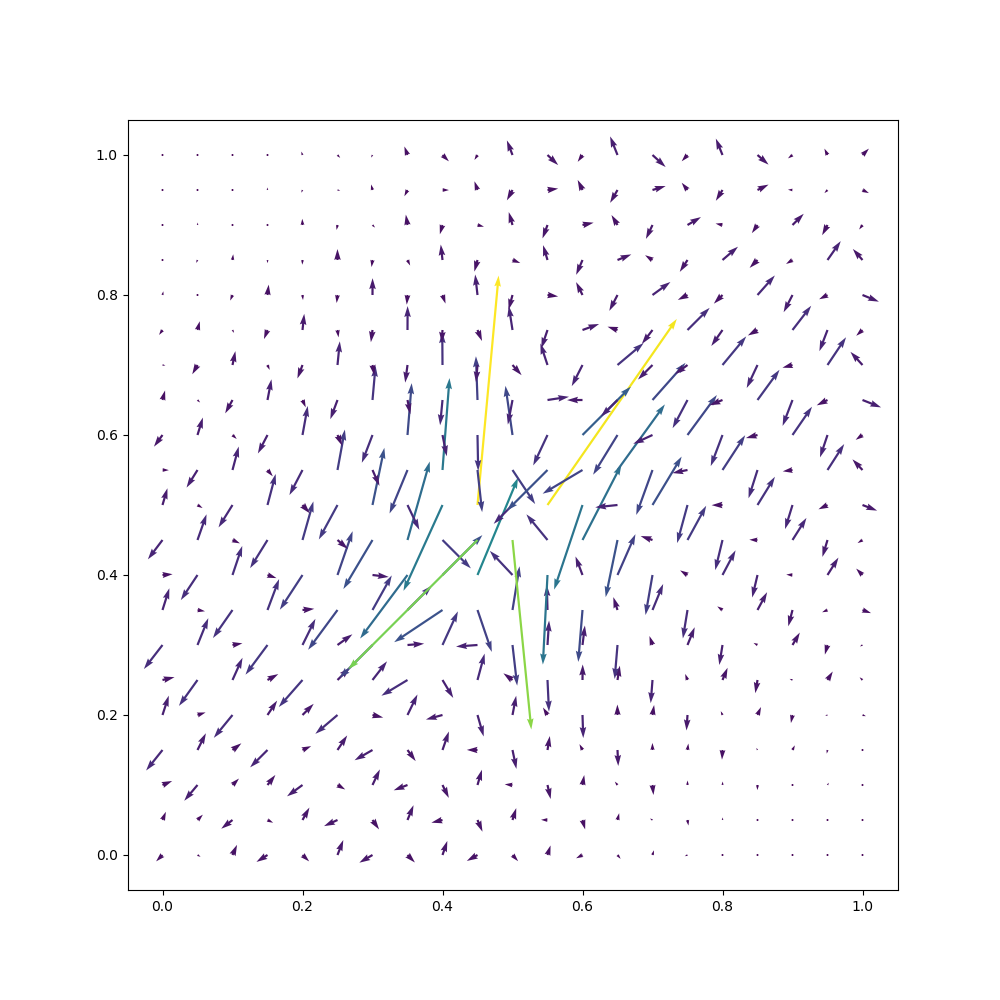}
\includegraphics[width=5cm]{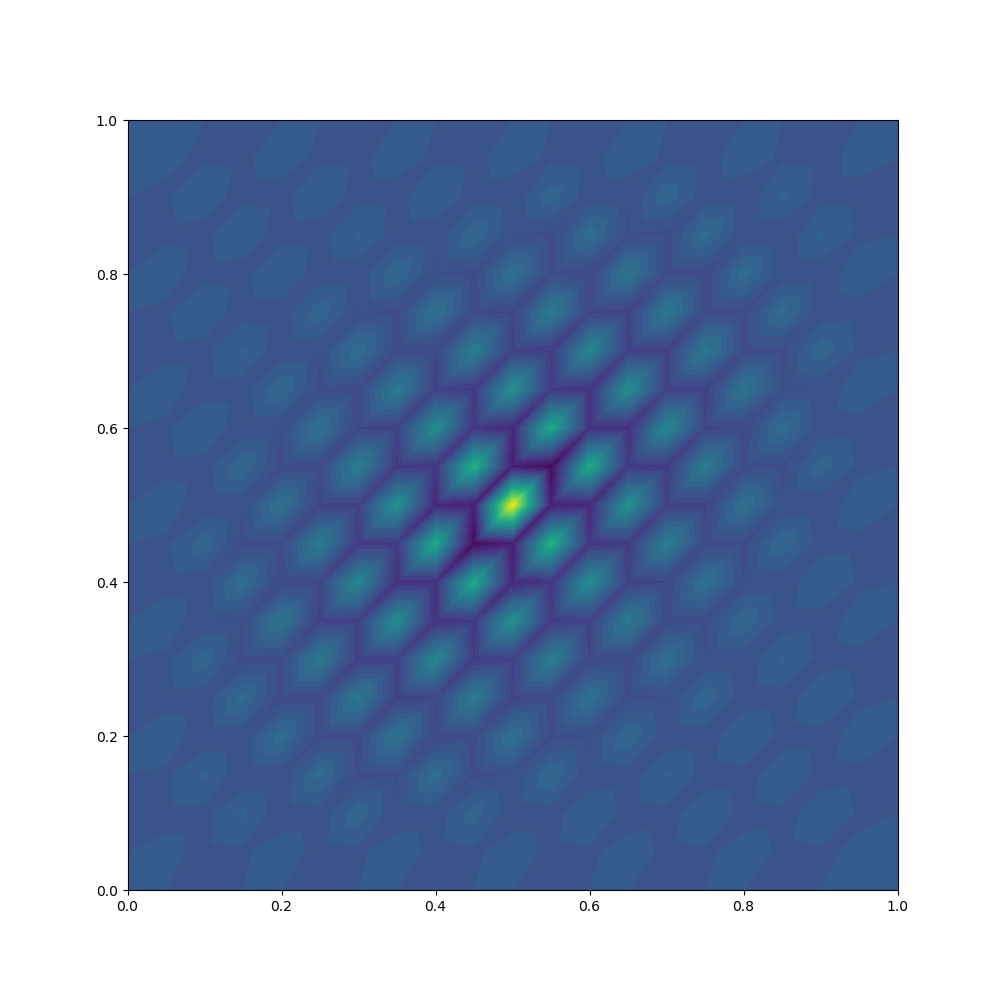}
}

\subcaptionbox*{$N=40$}
{
\includegraphics[width=5cm]{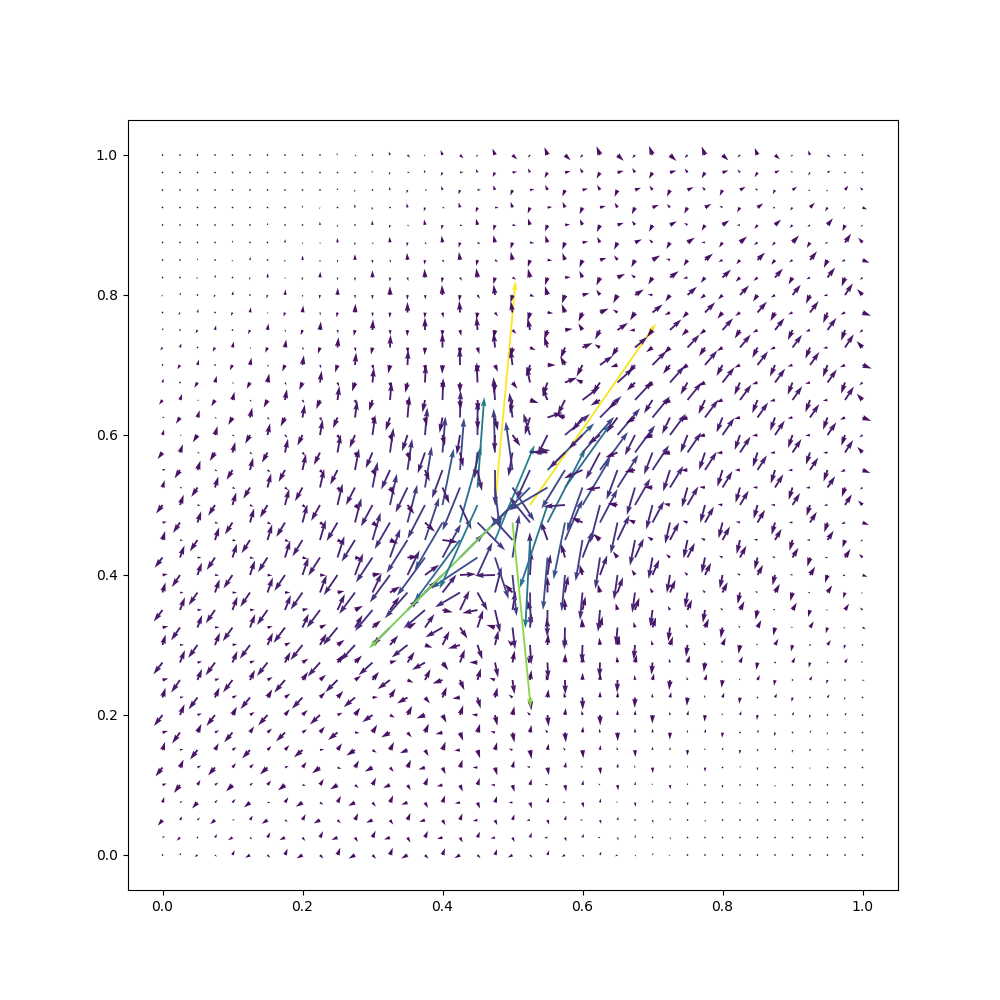}
\includegraphics[width=5cm]{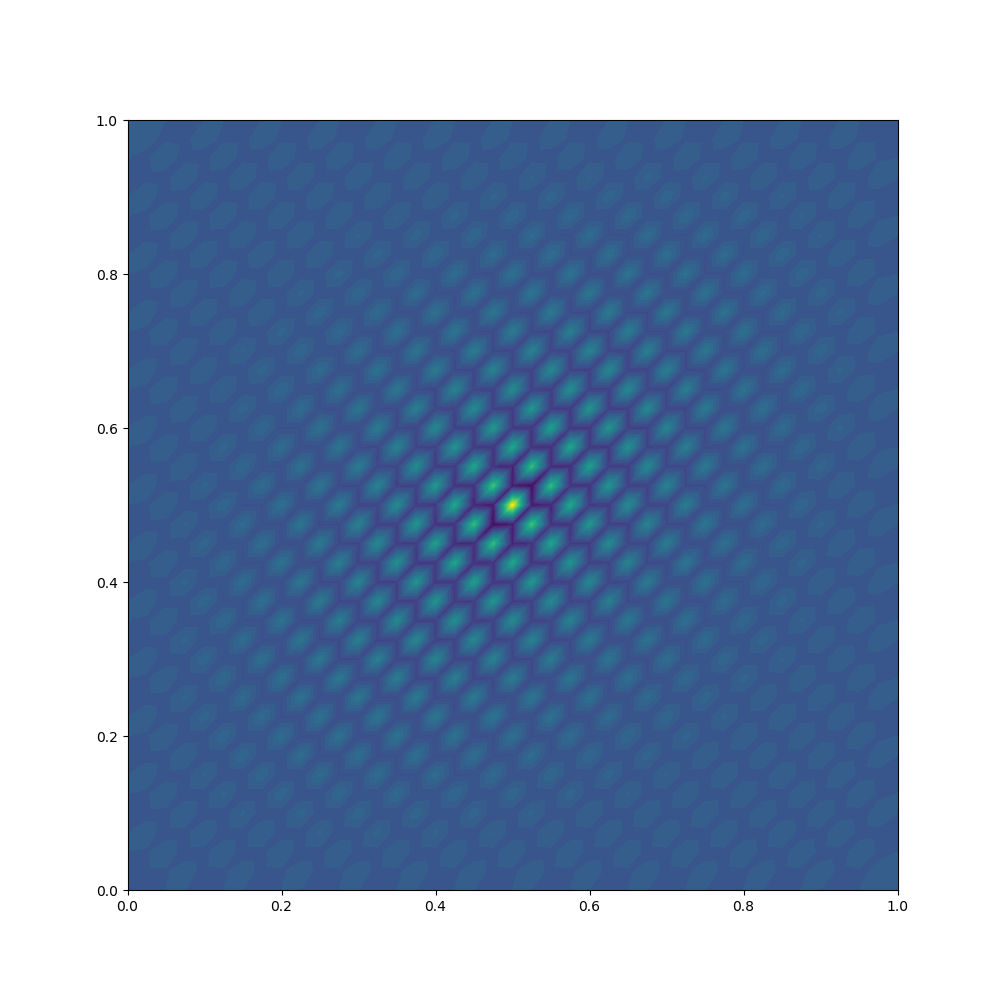}
}

\caption{Solution corresponding to $f$ equal to the approximation of the Dirac
delta function centered at $(1/2,1/2)$ on a coarser and a finer \meshright
mesh}
\label{fg:deltacright}
\end{figure}

\begin{figure}

\subcaptionbox*{$N=20$}
{
\includegraphics[width=5cm]{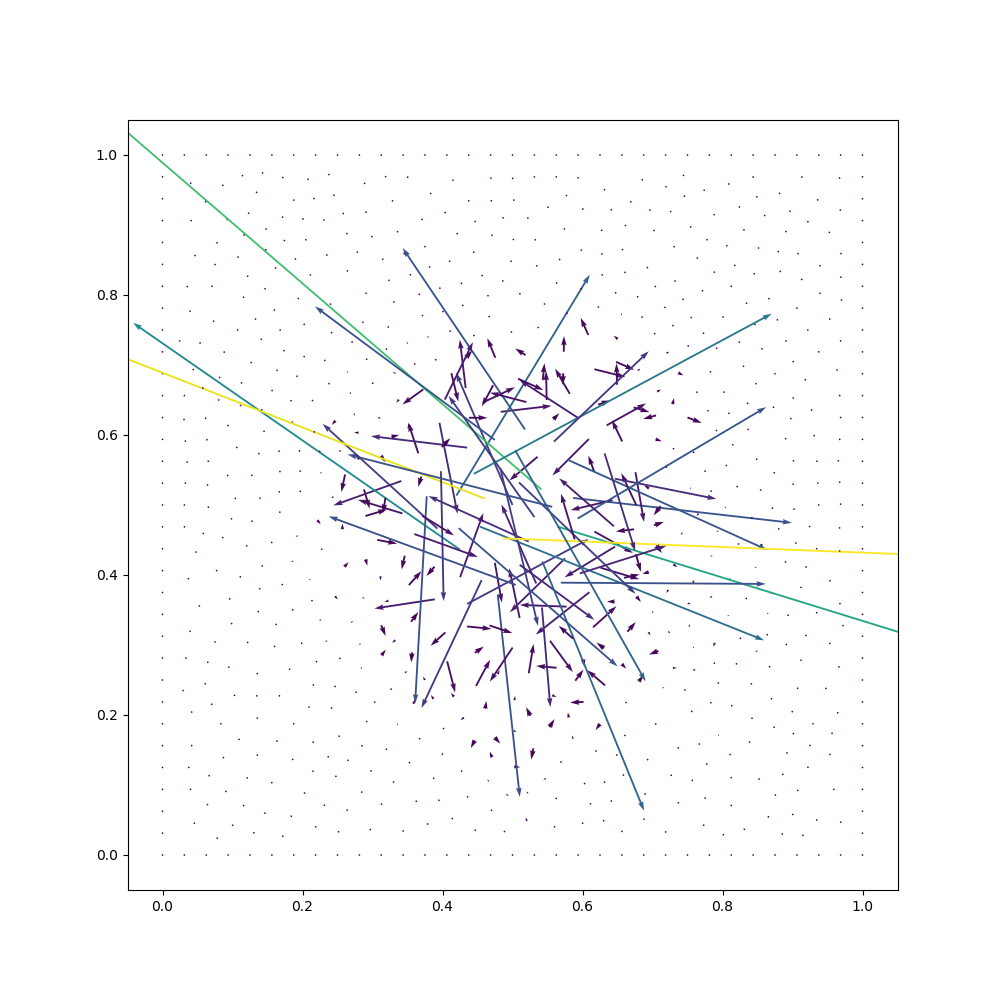}
\includegraphics[width=5cm]{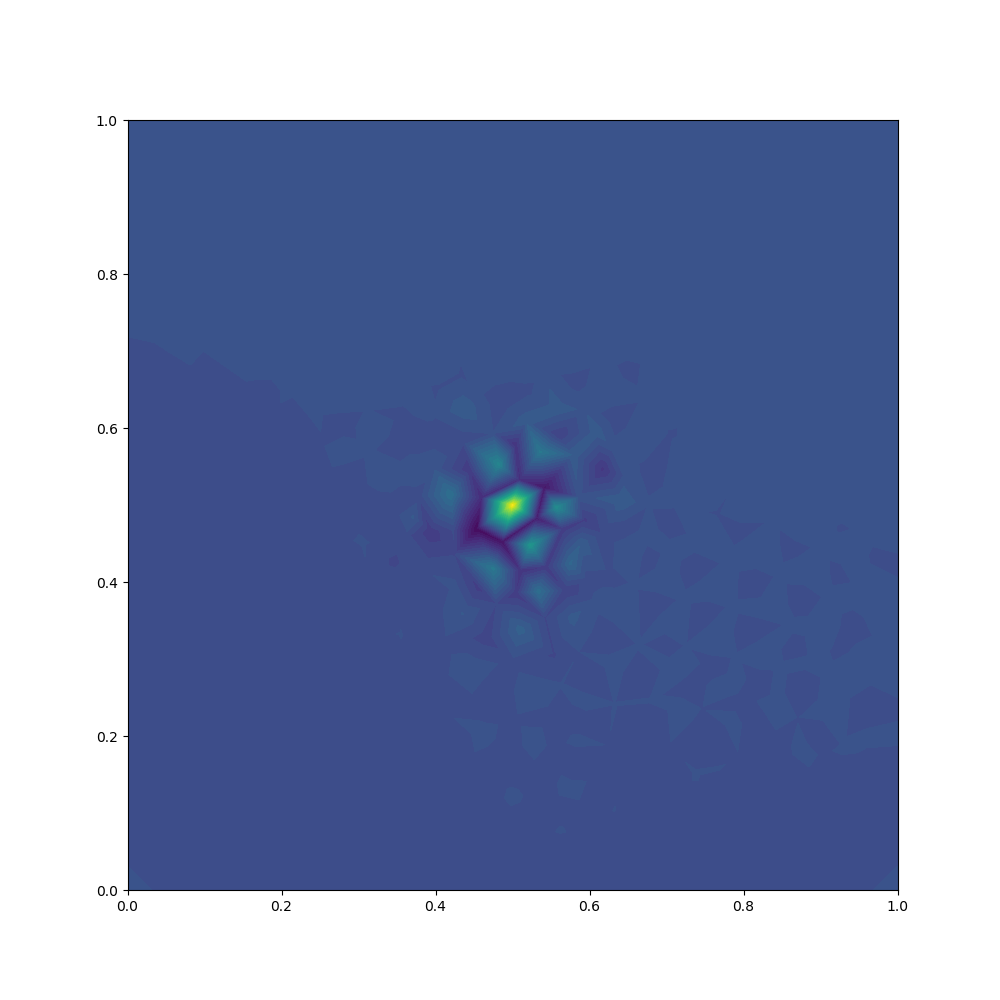}
}

\subcaptionbox*{$N=40$}
{
\includegraphics[width=5cm]{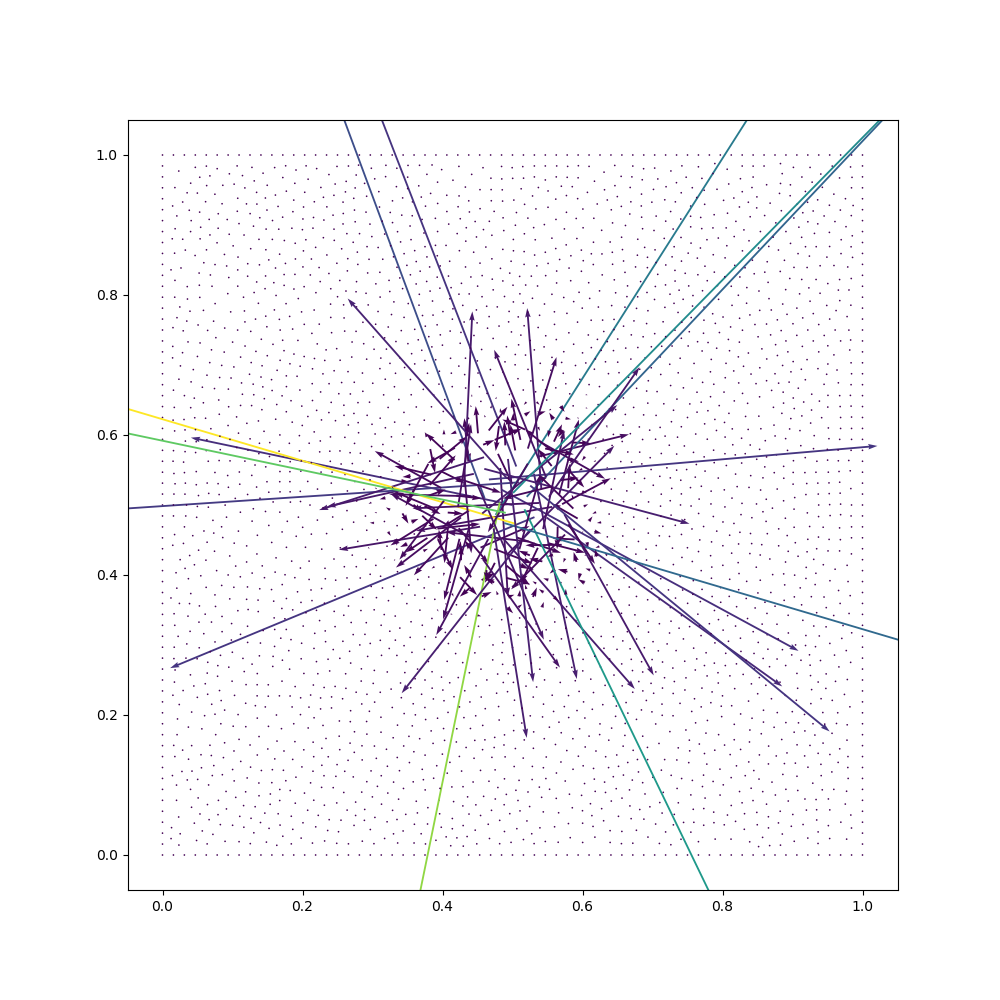}
\includegraphics[width=5cm]{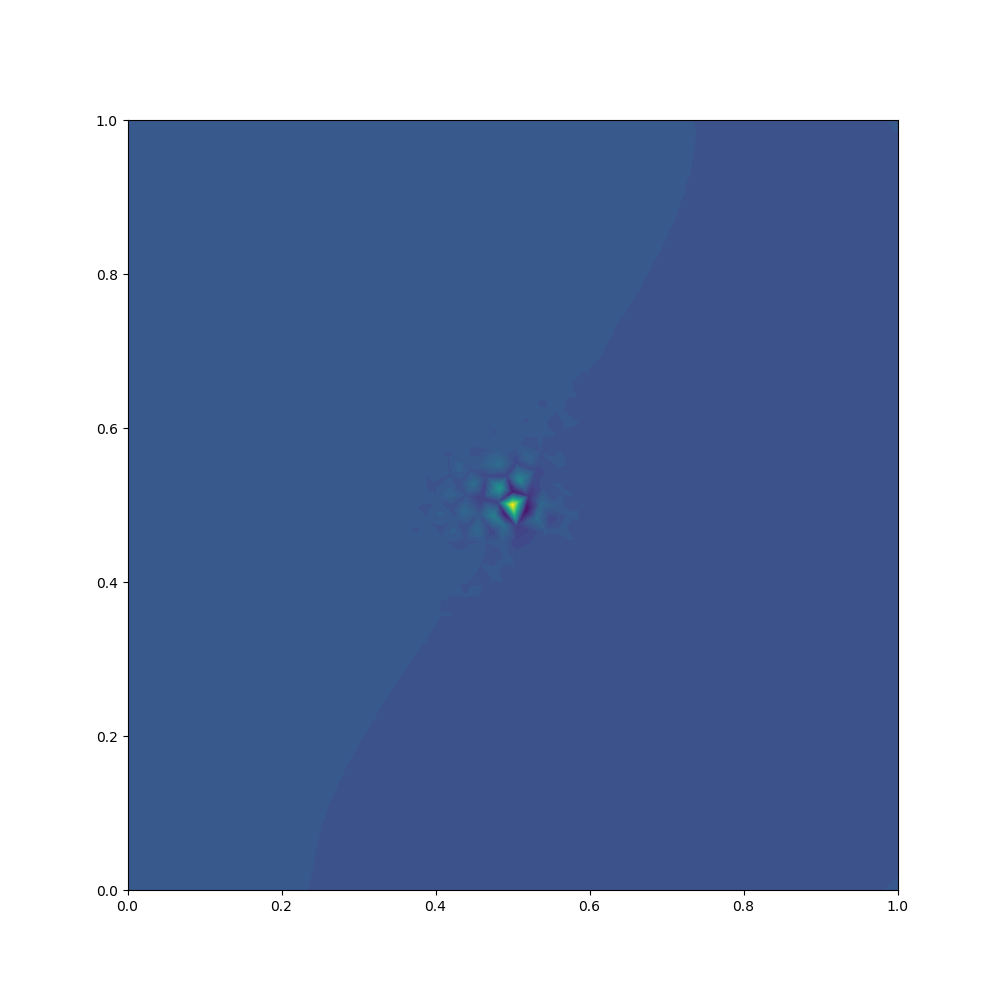}
}

\caption{Solution corresponding to $f$ equal to the approximation of the Dirac
delta function centered at $(1/2,1/2)$ on a coarser and a finer 
\meshnonstructured mesh}
\label{fg:deltacunstr}
\end{figure}

The situation is even more apparent if the Dirac delta function is centered at
the point $(1/2,1/2)$, that is at the center of our domain $\Omega$. The
corresponding solutions are reported in Figures~\ref{fg:deltaccross},
\ref{fg:deltacright}, and~\ref{fg:deltacunstr}.

\begin{figure}

\subcaptionbox*{\meshcrossed mesh}
{
\includegraphics[width=5cm]{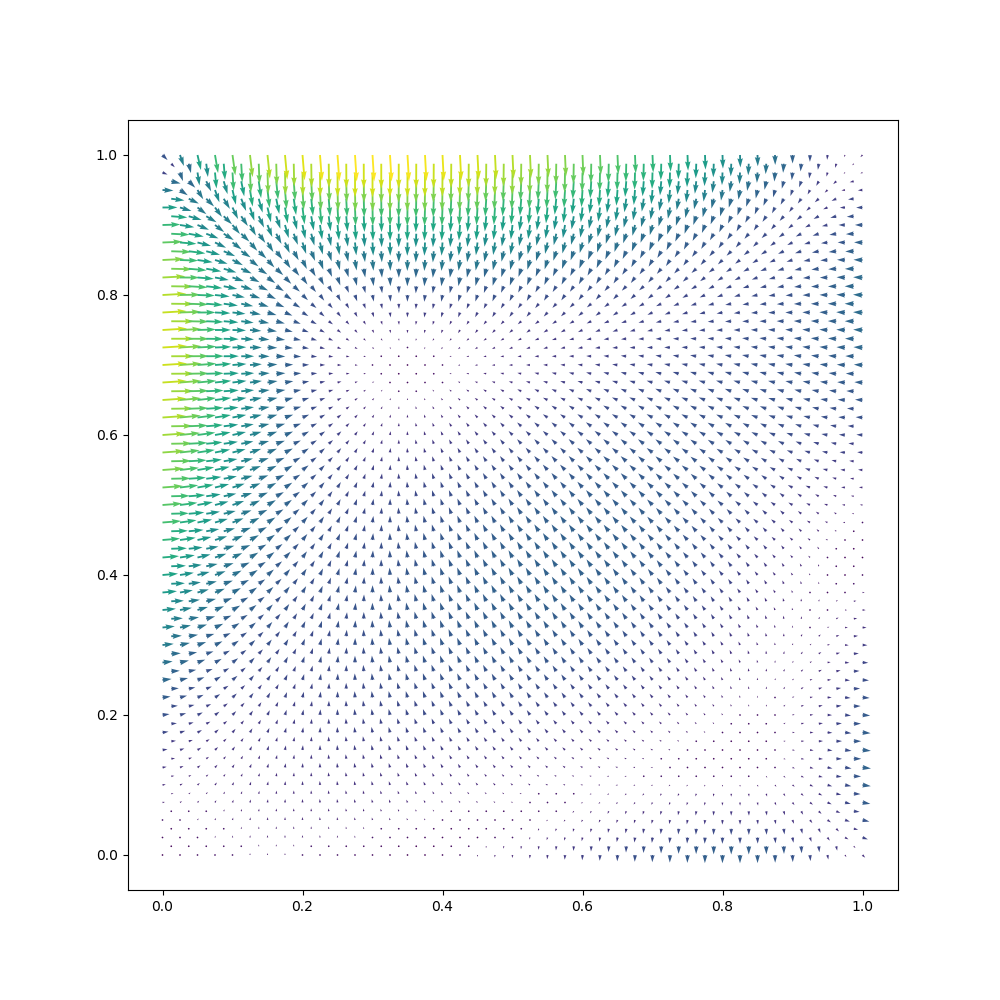}
\includegraphics[width=5cm]{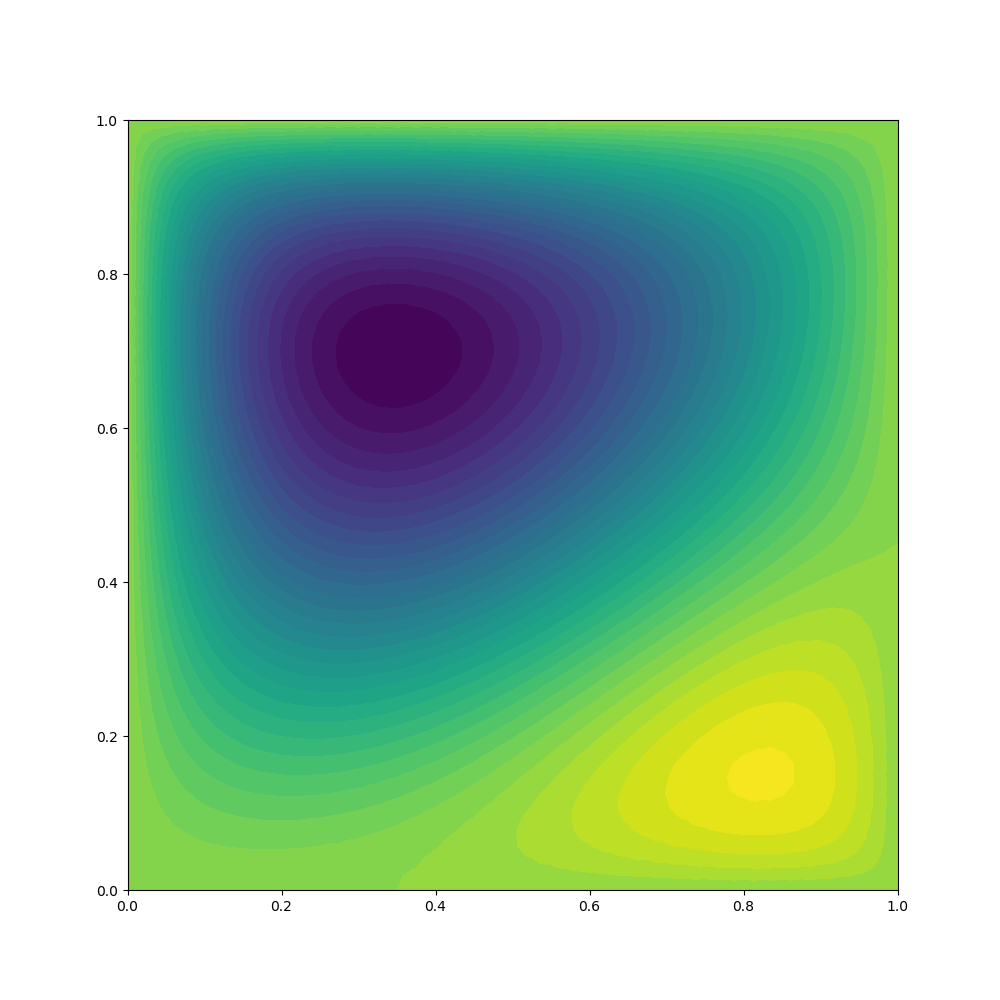}
}

\subcaptionbox*{\meshright mesh}
{
\includegraphics[width=5cm]{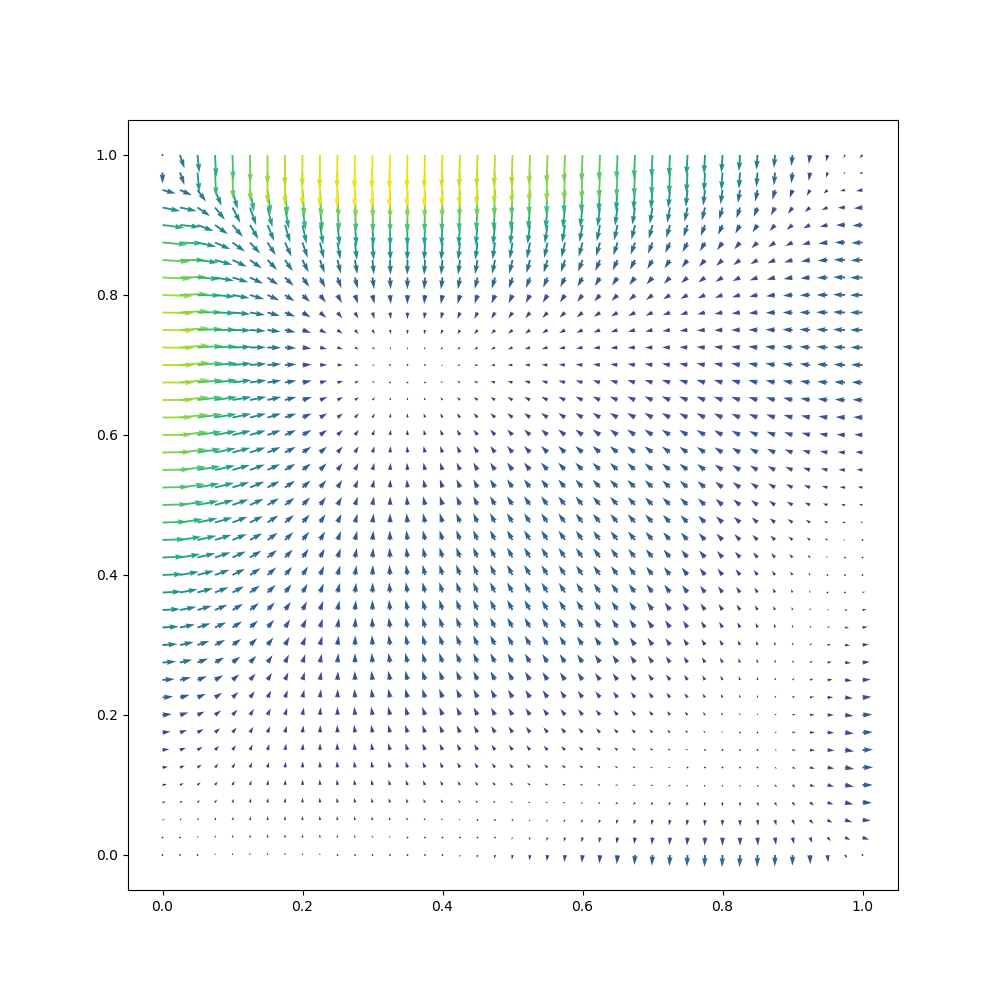}
\includegraphics[width=5cm]{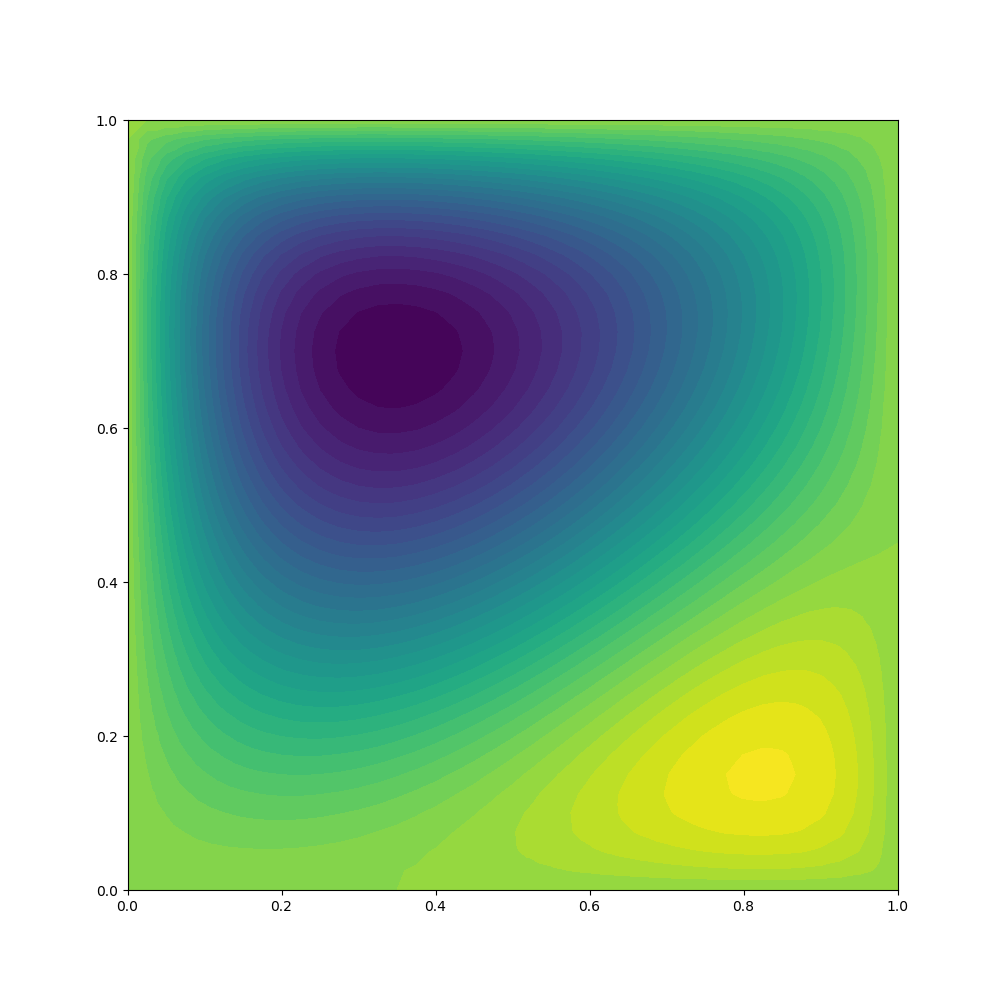}
}

\subcaptionbox*{\meshnonstructured mesh}
{
\includegraphics[width=5cm]{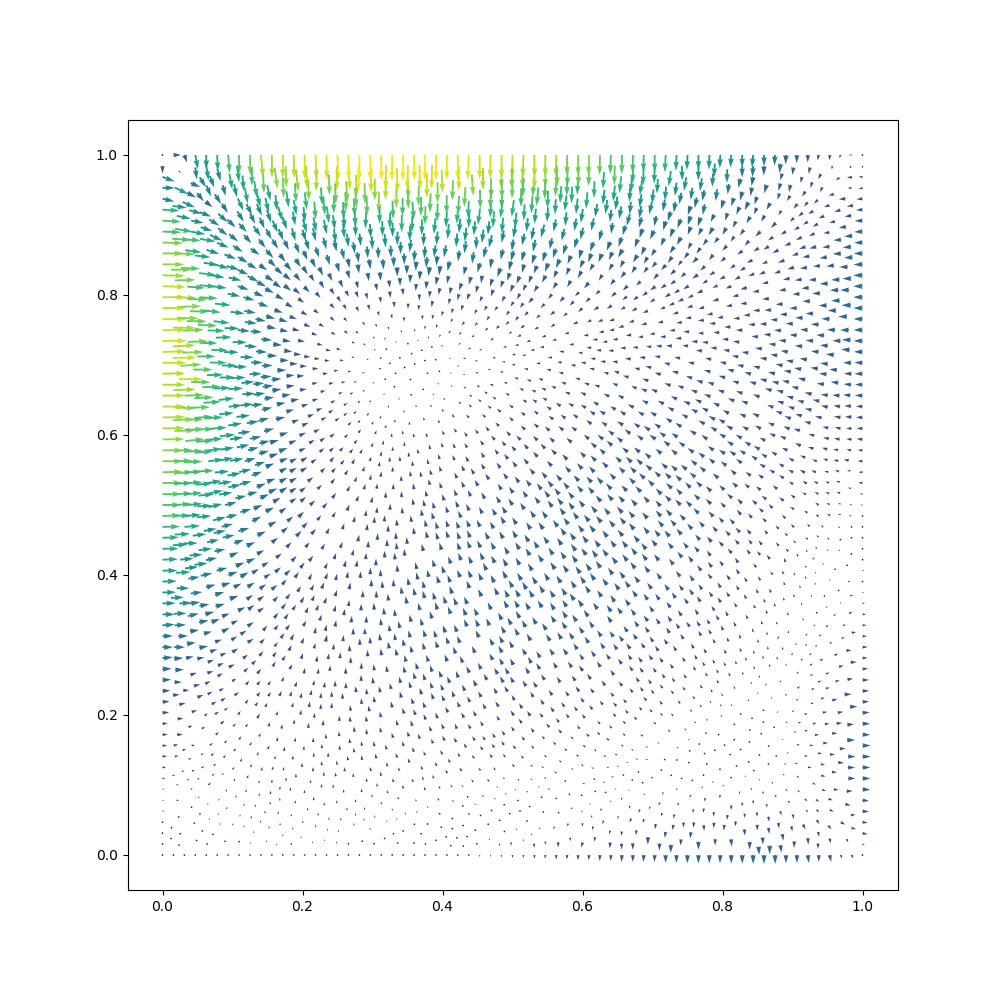}
\includegraphics[width=5cm]{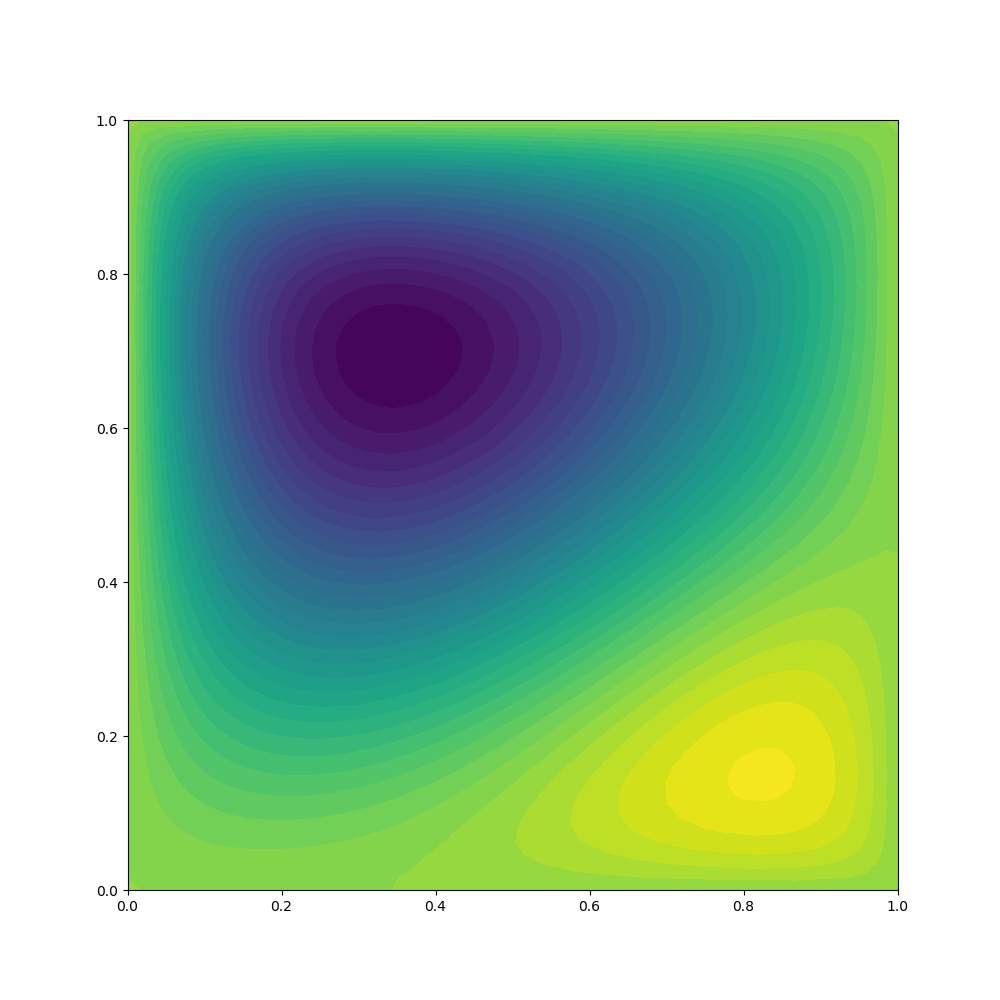}
}
\caption{Solution corresponding to the smooth right hand side
$f(x,y)=x-3y+\sin(x)$}
\label{fg:smooth}
\end{figure}

Figure~\ref{fg:smooth} shows the solution in a case when $f$ is smooth. We
take $f(x,y)=x-3y+\sin(x)$ and we can observe from the figures that the
discrete solutions do not present any spurious oscillation.

We conclude this section by reporting the rates of convergence in some cases
when the exact solutions are known.

Table~\ref{tb:smooth} shows the rates of convergence with respect of the
meshsize in the case when the solution is one of the Laplace eigenfunctions,
namely $u(x,y)=\sin(\pi x)\sin(2\pi y)$

\begin{table}
\begin{tabular}{ r|l|l|l }
dof's & $\|\bfsigma-\bfsigma_h\|_0$ & $\|u-u_h\|_0$ & $\|\grad(u-u_h)\|_0$\\
\hline
3281 & 7.178e-02 & 4.415e-03 & 7.324e-02\\
  12961 & 3.389e-02 (1.09) & 1.215e-03 (1.88) & 4.081e-02 (0.85)\\
  51521 & 1.695e-02 (1.00) & 3.029e-04 (2.01) & 2.036e-02 (1.01)\\
 205441 & 8.341e-03 (1.02) & 7.697e-05 (1.98) & 1.023e-02 (1.00)\\
 820481 & 4.173e-03 (1.00) & 1.920e-05 (2.01) & 5.094e-03 (1.01)\\
3279361 & 2.083e-03 (1.00) & 4.751e-06 (2.02) & 2.513e-03 (1.02)\\
13112321 & 1.051e-03 (0.99) & 1.182e-06 (2.01) & 1.263e-03 (0.99)
\end{tabular}
\caption{Rate of convergence with respect to $h$: smooth solution on the
\meshcrossed mesh in a square}
\label{tb:smooth}
\end{table}

The last example is related to a case where the solution is singular due to a
reentrant corner on the L-shaped domain
$\Omega=(-1,1)^2\setminus(0,1)\times(-1,0)$. We take as exact solution the
harmonic function $u(\rho,\theta)=\rho^{2/3}\sin((2/3)\theta)$, where
$(\rho,\theta)$ are the polar coordinates centered at the origin. We consider
$f=0$ and Dirichlet boundary conditions given by the exact solution. In
particular, the solution is vanishing along the two sides of the reentrant
corner meeting at the origin. It is well known that $u\in
H^{5/3-\epsilon}(\Omega)$ for $\epsilon>0$ but $u\not\in H^{5/3}(\Omega)$.
The rates of convergence are shown in Table~\ref{tb:singular} and are the
expected ones: approximately order $2/3$ for the energy norm and the
suboptimal order $4/3$ for the error $\|u-u_h\|$ in $L^2(\Omega)$.
\begin{table}
\begin{tabular}{ r|l|l|l }
dof's & $\|\bfsigma-\bfsigma_h\|_0$ & $\|u-u_h\|_0$ & $\|\grad(u-u_h)\|_0$\\
\hline
1155 & 7.181e-02 & 3.970e-03 & 1.136e-01\\
   4485 & 4.864e-02 (0.57) & 1.302e-03 (1.64) & 7.063e-02 (0.70)\\
  17673 & 3.096e-02 (0.66) & 4.392e-04 (1.59) & 4.457e-02 (0.67)\\
  70161 & 1.967e-02 (0.66) & 1.518e-04 (1.54) & 2.810e-02 (0.67)\\
 279585 & 1.246e-02 (0.66) & 5.395e-05 (1.50) & 1.779e-02 (0.66)\\
1116225 & 7.875e-03 (0.66) & 1.963e-05 (1.46) & 1.125e-02 (0.66)\\
4460673 & 4.972e-03 (0.66) & 7.304e-06 (1.43) & 7.102e-03 (0.66)
\end{tabular}
\caption{Rate of convergence with respect to $h$: singular solution on a
\meshnonstructured mesh in the L-shaped domain}
\label{tb:singular}
\end{table}
Figure~\ref{fg:Lshaped} shows a typical mesh for this computation and
Figure~\ref{fg:singular} reports the solution $(\bfsigma_h,u_h)$ computed
on the same mesh.

\begin{figure}
\includegraphics[width=5cm]{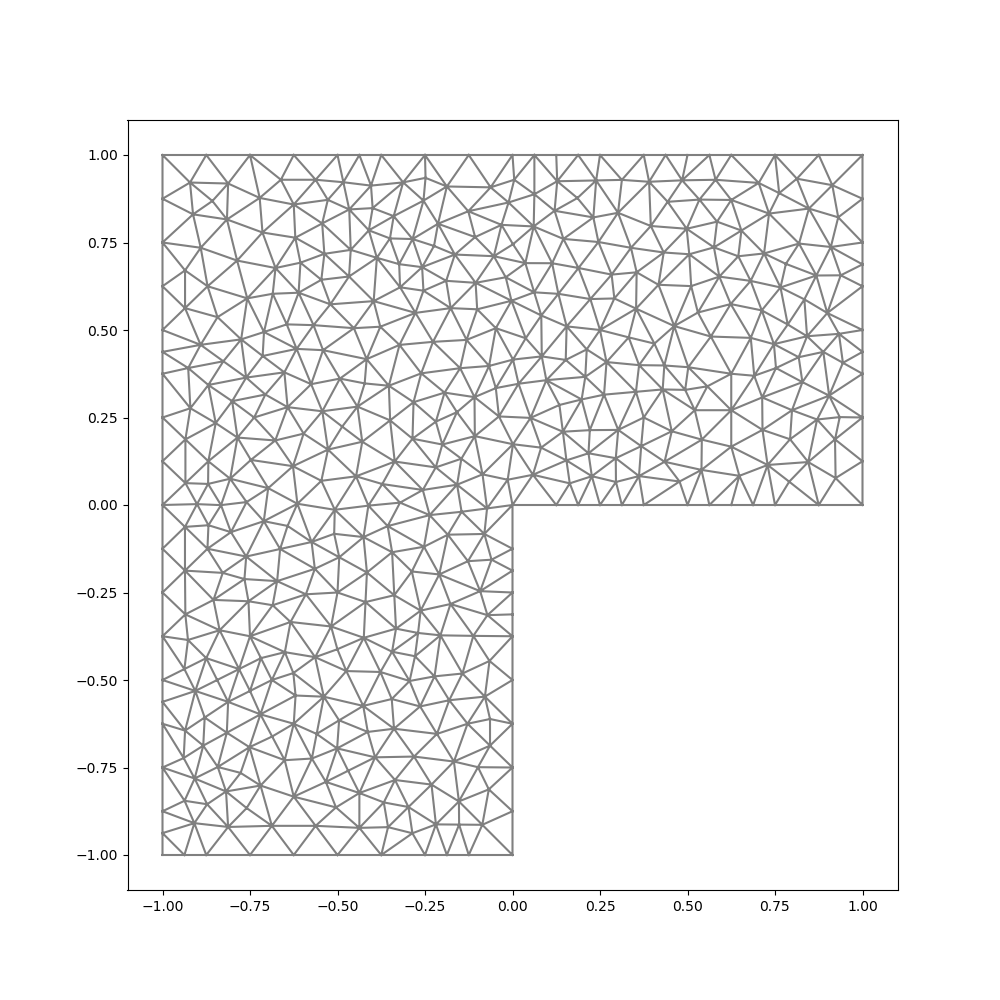}
\caption{\meshnonstructured mesh of the L-shaped domain}
\label{fg:Lshaped}
\end{figure}

\begin{figure}
\includegraphics[width=5cm]{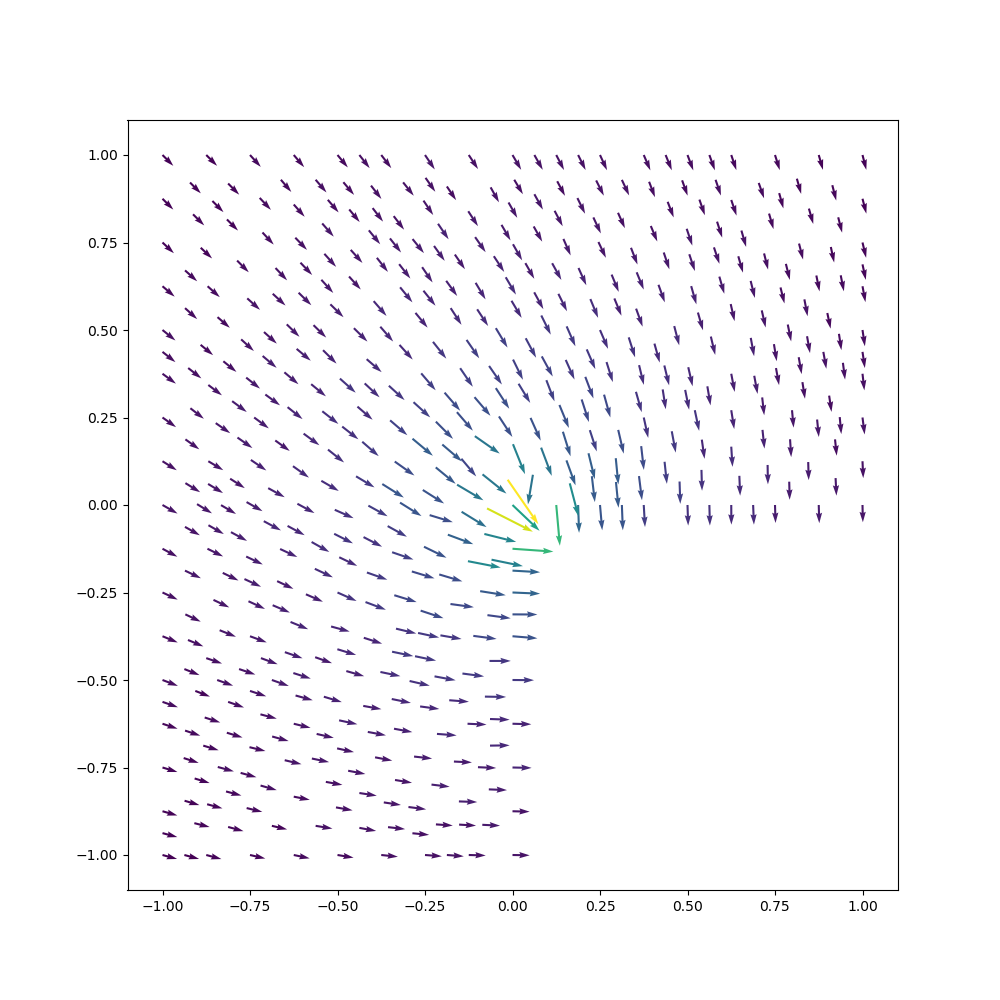}
\includegraphics[width=5cm]{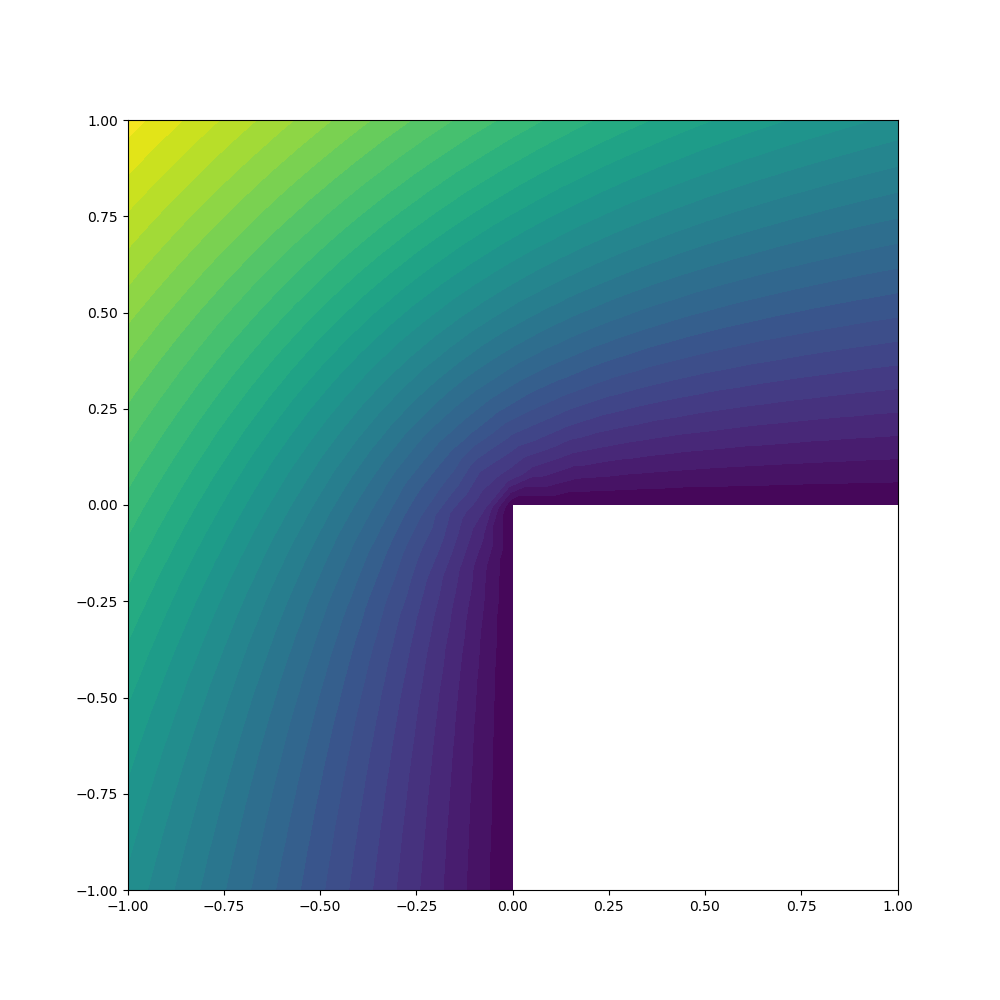}
\caption{Singular solution approximated on the \meshnonstructured mesh in the
L-shaped domain}
\label{fg:singular}
\end{figure}

\appendix

\section{The asymptotic behavior of the inf-sup constant}
\label{ap:asympt}

In this section we discuss theoretically the asymptotic behavior of the
inf-sup constant that has been studied numerically in Section~\ref{se:infsuph}.

We start by some remarks related to the continuous inf-sup condition: given
$u\in H^1_0(\Omega)$ there exists $\bfsigma\in L^2(\Omega)^2$ such that
\begin{equation}
\aligned
&(\bfsigma,\grad u)=\|\grad u\|_{L^2(\Omega)}^2\\
&\|\bfsigma\|_{L^2(\Omega)}\le C\|\grad u\|_{L^2(\Omega)}.
\endaligned
\label{eq:iscont}
\end{equation}
This can be easily achieved by defining $\bfsigma=\grad u$ and the constant
$C$ is equal to one in this case. The same consideration could lead to a
uniform discrete inf-sup condition if we had the inclusion
$\grad\Qh\subset\Vh$. In our case, however, the inclusion is not satisfied
because the space $\Vh$ is in $\Hdiv$. The natural question is then if it is
possible, at the continuous level, to find $\bfsigma\in\Hdiv$
satisfying~\eqref{eq:iscont}. This is certainly true if $\grad u$ belongs to
$\Hdiv$ which is the case, for instance, when $u$ is the solution of a Poisson
problem $-\Delta u=g$ for some $g$ in $L^2(\Omega)$. We state this result in
the following proposition.

\begin{proposition}
Let $u\in H^1_0(\Omega)$ be the solution of the problem
\[
(\grad u,\grad v)=(g,v)\quad\forall v\in H^1_0(\Omega)
\]
for some $g\in L^2(\Omega)$. Then $\bfsigma=\grad u$ belongs to $\Hdiv$ and
satisfies~\eqref{eq:iscont} with $C=1$.

\end{proposition}

We now try to mimic this proposition at the discrete level. Since the
divergence of $\Vh$ is piecewise constant, it is natural to start by
considering $u_h\in\Qh$ that solves
\begin{equation}
(\grad u_h,\grad v)=(g_h,v)\quad\forall v\in\Qh
\label{eq:poissonappendix}
\end{equation}
for some piecewise constant right hand side $g_h$. If we define
$\bfsigma_h=\grad u_h$ we have that $\bfsigma_h$ is not in $\Hdiv$ since
$\grad u_h$ has only the tangential component continuous across the elements
while we would need the normal one. We are then led to some sort of
equilibration strategy, in the spirit
of~\cite{PS,Braess,BraSch:08,Martin,survey}. The optimal way to achieve the
equilibration is to solve a global problem and to define $\bfsigma_h\in\Vh$ as
the solution of the mixed problem: find $\bfsigma_h\in\Vh$ and $p_h\in\Po$
such that
\begin{equation}
\left\{
\aligned
&(\bfsigma_h,\bftau)+(\div\bftau,p_h)=0&&\forall\bftau\in\Vh\\
&(\div\bfsigma_h,q)=-(g_h,q)&&\forall q\in\Po
\endaligned
\right.
\label{eq:mixedappendix}
\end{equation}
Then we have $\div\bfsigma_h=-g_h$ which implies
\[
(\bfsigma,\grad u_h)=(g_h,u_h)=\|\grad u_h\|_{L^2(\Omega)}^2
\]
Hence, we get a bound for the discrete inf-sup condition if we can estimate
the constant $C(h)$ for which it holds
\[
\|\bfsigma_h\|_{L^2(\Omega)}\le C(h)\|\grad u_h\|_{L^2(\Omega)}.
\]
Clearly, from~\eqref{eq:mixedappendix} we have
$\|\bfsigma_h\|_{L^2(\Omega)}\le C\|g_h\|_{L^2(\Omega)}$ so that we need a
bound of $\|g_h\|_{L^2(\Omega)}$ in terms of $\|\grad u_h\|_{L^2(\Omega)}$.
Since $u_h$ solves~\eqref{eq:poissonappendix}, this bound is related to an
inf-sup condition between the spaces $\Qh$ and $\Po$. More precisely, the
following lemma holds true.

\begin{lemma}
Let $\Qzero$ be the subspace of $\Qh$ defined as the solutions
of~\eqref{eq:poissonappendix} for some $g_h\in\Po$. Let $\zeta(h)>0$ be such
that the following inf-sup condition holds true
\begin{equation}
\inf_{w_h\in\Qh}\sup_{g_h\in\Po}
\frac{(g_h,w_h)}{\|g_h\|_{L^2(\Omega)}\|w_h\|_{L^2(\Omega)}}\ge\zeta(h).
\label{eq:p1p0}
\end{equation}
Then the inf-sup condition holds true for the space $\Qzero$
\[
\inf_{v\in\Qzero}\sup_{\bftau\in\Vh}\frac{(\bftau,\grad v)}
{\|v\|_{H^1(\Omega)}\|\bftau\|_{L^2(\Omega)}}\ge C\zeta(h).
\]
\end{lemma}

The inf-sup constant $\zeta(h)$ in~\eqref{eq:p1p0} can be investigated
theoretically on special meshes or in the case of a one dimensional domain. It
turns out that its behavior is similar to the one of $\beta_h$
in~\eqref{eq:betah}. For instance, Table~\ref{tb:p1p0} shows the value of the
lowest eigenvalue associated with the inf-sup constant (computed as in
Section~\ref{se:infsuph}). It turns out that with a uniform mesh $\zeta(h)$
decays as O(h) while with a non structured mesh the decay is less pronounced.

\begin{table}
\begin{tabular}{ l|l|l }
h &\meshnonstructured mesh&\meshright mesh\\
\hline
$1/2$  & 0.66666667 &0.66666667\\
$1/2^2$& 0.20980372 &0.33333333\\
$1/2^3$& 0.09283759 &0.11409783\\
$1/2^4$& 0.08061165 &0.03137791\\
$1/2^5$& 0.06760737 &0.00803861\\
$1/2^6$& 0.0626653  &0.00202209
\end{tabular}
\caption{Lowest eigenvalue associated with the inf-sup constant
in~\eqref{eq:p1p0}, corresponding to the square of $\zeta(h)$}
\label{tb:p1p0}
\end{table}

We conclude this appendix by showing the behavior of the \emph{worse} $w_h$
in~\eqref{eq:p1p0} as it comes out from the numerical experiments. In one
dimension, where the behavior of $\zeta(h)$ is $O(h^2)$ in agreement
with~\cite{1D}, we get the highly oscillating function plotted in
Figure~\ref{fg:1D}.
\begin{figure}
\includegraphics[width=5cm]{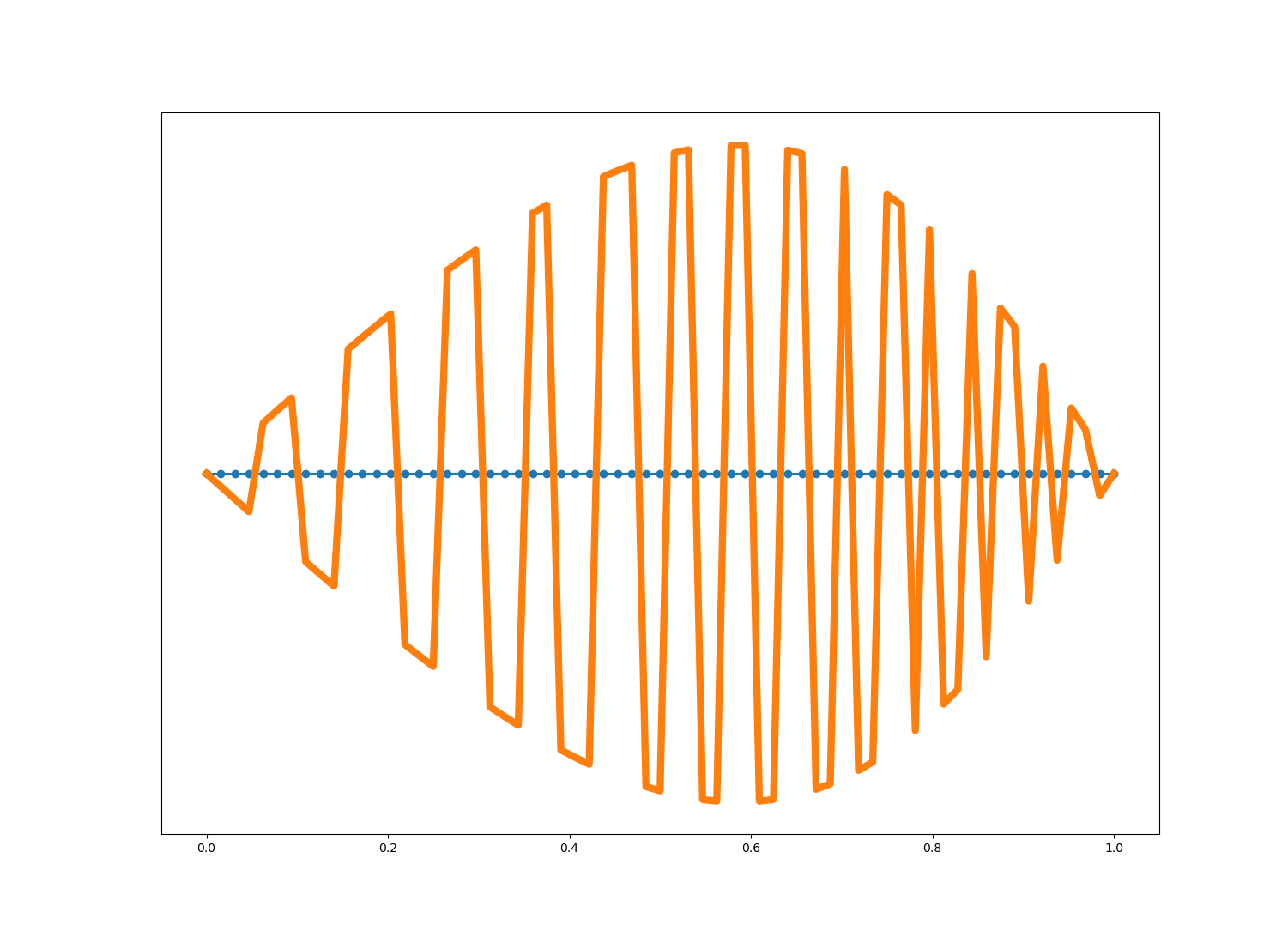}
\caption{The function $w_h$ corresponding to the first singular value
of~\eqref{eq:p1p0} in one dimension}
\label{fg:1D}
\end{figure}
The analogous functions in the two dimensional \meshright and
\meshnonstructured meshes are plotted in Figures~\ref{fg:2Dr}
and~\ref{fg:2Du}, respectively.

\begin{figure}
\includegraphics[width=5cm]{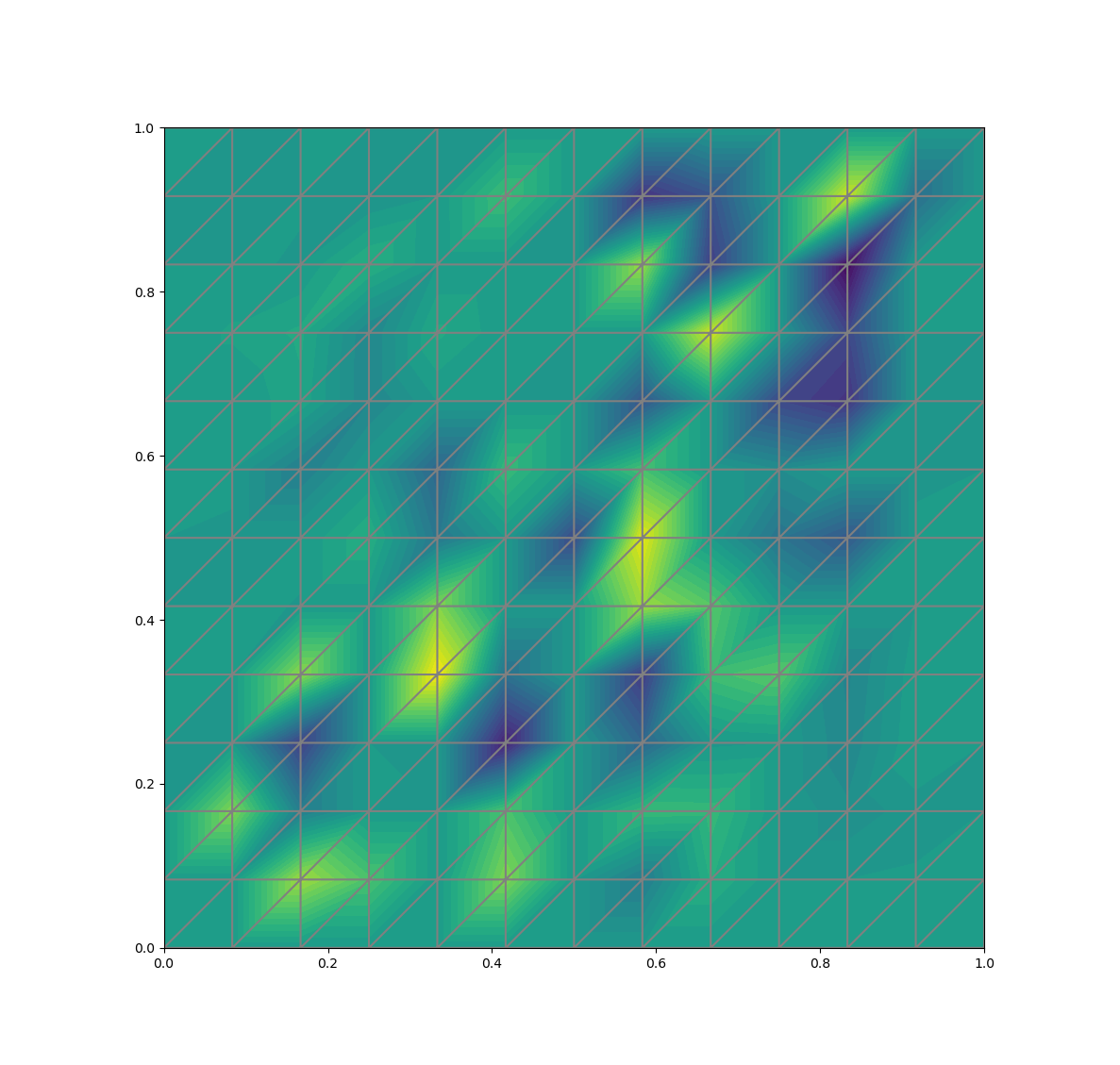}
\caption{The function $w_h$ corresponding to the first singular value
of~\eqref{eq:p1p0} on the \meshright mesh}
\label{fg:2Dr}
\end{figure}

\begin{figure}
\includegraphics[width=5cm]{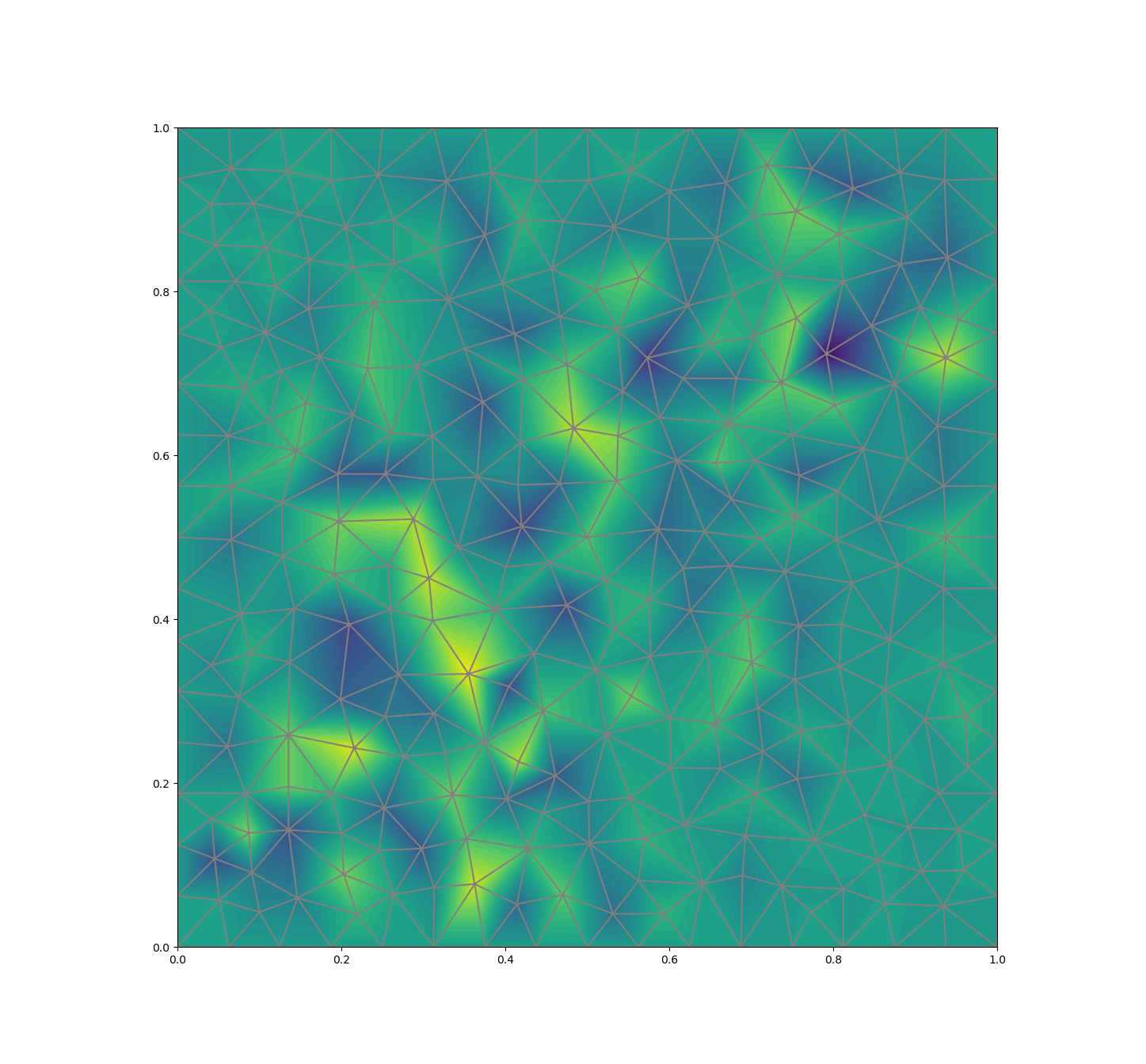}
\caption{The function $w_h$ corresponding to the first singular value
of~\eqref{eq:p1p0} on the \meshnonstructured mesh}
\label{fg:2Du}
\end{figure}

\section{Some connections with flux equilibration}
\label{ap:fleur}

We have seen that the inf-sup constant $\beta_h$ in general is not bounded
below independently of $h$. We have also seen in several parts of this paper
some analogies between the inf-sup condition and the well known flux
equilibration strategy used, for instance, in the a posteriori analysis of
standard finite elements (see, for
instance,~\cite{PS,Braess,BraSch:08,Martin,survey}).
In this appendix we explore this connections in more detail.

To this aim, we make use of the following natural modification of the usual
Fortin trick~\cite{bbf}.

\begin{proposition}

Let $\Qzero$ be a subspace of $\Qh$ and assume that there exists a linear
projection $\Pi:\grad(\Qzero)\to\Vh$ such that for any
$\bftau\in\grad(\Qzero)$
\begin{equation}
\aligned
&(\bftau-\Pi\bftau,\grad v)=0&&\forall v\in\Qzero\\
&\|\Pi\bftau\|_{L^2(\Omega)}\le C_{\Pi}\|\bftau\|_{L^2(\Omega)}.
\endaligned
\label{eq:fortin}
\end{equation}
Then the following inf-sup condition holds true
\begin{equation}
\inf_{v\in\Qzero}\sup_{\bftau\in\Vh}
\frac{(\bftau,\grad v)}{\||\bftau\|_{L^2(\Omega)}\|v\|_{H^1(\Omega)}}
\ge\beta
\label{eq:infsup0}
\end{equation}
with $\beta=1/(C_PC_\Pi)$, being $C_P$ the Poincar\'e constant.

\label{pr:fortin}
\end{proposition}

\begin{proof}

Take $v\in\Qzero$, then we have
\[
\aligned
\|v\|_{H^1(\Omega)}&\le C_P\|\grad v\|_{L^2(\Omega)}
=C_P\frac{(\grad v,\grad v)}{\|\grad v\|_{L^2(\Omega)}}\\
&\le C_P\sup_{\bftau\in\grad(\Qzero)}
\frac{(\bftau,\grad v)}{\|\bftau\|_{L^2(\Omega)}}\\
&=C_P\sup_{\bftau\in\grad(\Qzero)}
\frac{(\Pi\bftau,\grad v)}{\|\bftau\|_{L^2(\Omega)}}
\le{C_P}{C_\Pi}\sup_{\bftau\in\grad(\Qzero)}
\frac{(\Pi\bftau,\grad v)}{\|\Pi\bftau\|_{L^2(\Omega)}}\\
&\le{C_P}{C_\Pi}\sup_{\bftau\in\Vh}
\frac{(\bftau,\grad v)}{\|\bftau\|_{L^2(\Omega)}},
\endaligned
\]
which implies the inf-sup condition~\eqref{eq:infsup0} with
$\beta=1/(C_P C_\Pi)$.

\end{proof}

Looking carefully at the properties of the Fortin projector
in~\eqref{eq:fortin}, we can see that the construction of $\Pi$ consists in
finding a mapping from gradients of continuous piecewise polynomials into the
lowest order Raviart--Thomas space in $\Hdiv$. This problem has been widely
studied, for instance, in the framework of flux equilibration of finite
element spaces.

In this appendix we try to identify a suitable subspace $\Wone$ of $\Qh$ that
could be used as $\Qzero$ in Proposition~\ref{pr:fortin}.
We start with the following heuristic reasoning. Let us take a function
$f\in L^2(\Omega)$ and for each $h$ consider the solution $u_h\in\Qh$ of the
following standard Galerkin problem
\begin{equation}
(\grad u_h,\grad v_h)=(\projo f,v_h)\qquad\forall v_h\in\Qh.
\label{eq:cont}
\end{equation}
Then the most natural way to find an equilibrated $\bfsigma_h\in\Vh$ is to
solve the following global mixed problem: find
$(\bfsigma_h,p_h)\in\Vh\times\Po$ such that
\[
\left\{
\aligned
&(\bfsigma_h,\bftau_h)+(\div\bftau_h,p_h)=0&&\forall\bftau_h\in\Vh\\
&(\div\bfsigma_h,q_h)=-(f,q_h)&&\forall q_h\in\Po.
\endaligned
\right.
\]
By considering $\bftau=\grad u_h$ and $\Pi\bftau=\bfsigma_h$ and observing
that $\div\bfsigma_h=-\projo f$, we then have
\[
(\bftau-\Pi\bftau,\grad v_h)=(\projo f,v)+(\div\bfsigma_h,v)=0
\qquad\forall v\in U_h,
\]
that is, the first Fortin condition in~\eqref{eq:fortin} is satisfied.
Moreover, we can bound $\bfsigma_h$ by triangular inequality as follows
\[
\|\bfsigma_h\|_0\le\|\bfsigma_h-\grad u_h\|_0+\|\grad u_h\|_0.
\]
The first term on the right hand side is going to zero and can be bounded by
\[
\|\bfsigma_h-\grad u_h\|_0\le h^s\|\projo f\|_0
\]
if the solution of the continuous problem corresponding to~\eqref{eq:cont} has
regularity $H^{1+s}(\Omega)$.
In order to estimate $\|\bfsigma_h\|_0$ uniformly in terms of
$\|\grad u_h\|_0$ we then need to bound $\|\projo f\|_0$ uniformly by
$h^{-s}\|\grad u_h\|_0$. Unfortunately this cannot be done in general.

The equilibration technique is used in the next theorem to show how the
Fortin operator restricted to a suitable subspace of $\Qh$ behaves
asymptotically in $h$.

\begin{theorem}

Let $\Wone$ be the following subspace of $\Qh$:
\[
\Wone=\{w\in\Qh\,|\,\exists g_0\in\Po:(\grad w,\grad v)=(g_0,v)\
\forall v\in U_h\},
\]
where $\Po$ is the space of piecewise constant functions. Let $\zeta(h)$ be
the inf-sup constant introduced in~\eqref{eq:p1p0} and $\rho(h)$ a function of
$h$ so that the inverse inequality $\|\grad v\|_0\le\rho(h)\|v\|_0$ holds
true. Then there exists a
Fortin operator $\Pi$ as in Proposition~\ref{pr:fortin} with
$\Qzero=\Wone$ and $C_\Pi\le Ch\rho(h)/\zeta(h)$.
\label{th:fleur}
\end{theorem}

\begin{proof}
We denote by $\mathcal{T}_h$ our triangulation of $\Omega$ and by
$\mathcal{E}_h$ the skeleton of the edges.
Given $u_h\in\Wone$, we consider $\bftau=\grad u_h$ so that the flux
reconstruction procedure provided in~\cite{BraSch:08} and described with
explicit formulas in \cite{cism} will give a function
$\Pi\bftau=\bfsigma_h\in\Vh$.
The construction is local and is performed as
\[
\bfsigma_h = \nabla u_h+ \sum\limits_{ z \in \mathcal{V}} \bfsigma_{z}
\]
where $\bfsigma_z$ is a \emph{discontinuous} $RT_0$ function ($DRT_0$)
supported on the patch 
\[
\omega_z = \{T \in \mathcal T:z \text{ is a vertex of } T \}.
\]
As the space $DRT_0$ consists of edge basis functions, we will also consider
the set
\[
\mathcal E_z = \{E \in \mathcal E : z \text{ is a vertex of } E \}.
\]

Let $g_0\in\Po$ be such that $(\grad u_h,\grad v)=(g_0,v)$ $\forall v\in U_h$
as in the definition of $\Wone$.
Following the standard procedure of flux reconstruction we have that
$\bfsigma_h$ belongs to $\Vh$ and that the Fortin property
$(\grad u_h-\bfsigma_h,\grad v)=0$ ($\forall v\in\Wone$) is satisfied if the
$\bfsigma_{z}$ are chosen such that
\begin{equation}
\aligned
&\div\bfsigma_z|_T 
= -\frac{1}{|T|} (g_0,\phi_z)_T 
&&\forall T \in \omega_z\\
&\jump{\bfsigma_z \cdot \bn} {E}= - \frac{1}{2}\jump{\grad u_h \cdot \bn}{E}
&&\forall E \in \mathcal E_z\\
&\bfsigma_z\cdot\bn=0&&\text{on }\partial\omega_z
\endaligned
\label{strrecopatch2}
\end{equation}
where $\phi_z$ is the linear nodal basis function corresponding to the
node $z$ and vanishing on the boundary of the patch $\omega_z$.

Indeed, the second equation in~\eqref{strrecopatch2} guarantees that
$\jump{\bfsigma_h}{E}=0$ for all internal edges $E$, since each edge belongs
to exactly two patches. Hence $\bfsigma_h$ belongs to $\Hdiv$ and we can
evaluate its divergence element by element. From the first equation
in~\eqref{strrecopatch2} we have
\[
\left(\div\sum\limits_{ z \in \mathcal{V}} \bfsigma_{z}\right)\Bigg|_T=
-\frac1{|T|}\sum_{i=1}^3(g_0,\phi_{z_i})_T=-\frac1{|T|}(g_0,1)_T=g_0|_T,
\]
where $z_i$ ($i=1,2,3$) are the three vertices of $T$. From
$(\div\grad u_h)|_T=0$ it follows $\div\bfsigma_h=-g_0$, so that we get the
Fortin property
\[
(\bfsigma_h,\grad v)=(g_0,v)=(\grad u_h,\grad v)\qquad\forall v\in\Qh.
\]

The edge basis functions considered for the space $\Vh$ are supported in the
two adjacent triangles of the edge $E$. 
We denote by $T^-_E$ and $T^+_E$ the two triangles adjacent to $E$ and define
an edge oriented basis
$\{\bfpsi_E^-\}_{E\in\mathcal{E}_h}
\cup\{\bfpsi_E^+\}_{E\in\mathcal{E}_h,E\not\subset\partial\Omega}$
for $DRT_0$, using the basis functions 
\[
\bfpsi_E^-(\bx) = \begin{cases}
-\frac{1}{2|T^-_E| }(\bx - P^-_E)  &\text{on } T^-_E\\
0 &\text{elsewhere}
\end{cases} 
\text{ and }
\bfpsi_E^+(\bx)
= 
\begin{cases}
\frac{1}{2|T^+_E| }(\bx - P^+_E)  &  \text{on } T^+_E\\
0 & \text{elsewhere}
\end{cases}
\]
where $P^-_E$ and $P^+_E$ are the vertices of $T^-_E$ and $T^+_E$,
respectively, not shared by the two triangles.
This basis uses a similar representation of the one presented
in~\cite{BahCar:05} although here
\[
\div\bfpsi_E^\pm =\pm\frac{1}{|T^\pm_E|}
\]
and thus   
\[
\int_E\bfpsi^\pm_E\cdot\bn_E=\pm\left(\div\bfpsi_E^\pm,1\right)_{T^\pm_E}=1,
\]
where the normal $\bn_E$ is pointing from $T^+_E$ to $T^-_E$. It follows
\[
\bfpsi^\pm_E\cdot\bn_E=\frac{1}{|E|}.
\]
These basis functions allow for the following construction: 
\[
\bfsigma_z = \sum\limits_{E \in \mathcal E_z} 
(\tau_{E,z}^- \bfpsi_E^-+\tau_{E,z}^+ \bfpsi_E^+)
\]
where the coefficients $\tau_{E,z}^+$ and $\tau_{E,z}^-$ will be chosen so
that the equilibration conditions~\eqref{strrecopatch2} hold.

For the computation of the coefficients, we index the triangles in
the patch from $1$ to $n_z$ and define $T_0:=T_{n_z}$ and $ T_{n_z+1}:=T_{1}$.
Further, we define $E_i = T_i \cap T_{i+1}$ and $T^+_{E_i} = T_{i+1}$. Then, 
\[
\div\bfsigma_z|_{T_i} = 
\tau_{E_{i-1},z}^+ \frac{1}{|T_{i}|}
-\tau_{E_i,z}^-  \frac{1}{|T_{i}|}
\]
and the first condition in~\eqref{strrecopatch2} reads
\[
\tau_{E_i,z}^-
= (g_0,\phi_z)_{T_i} 
+ \tau^+_{E_{i-1},z} \quad \forall i=1,\dots,n_z.
\]
The second condition in~\eqref{strrecopatch2} implies
\[
\tau_{E_i,z}^+ - \tau_{E_i,z}^-= 
-\frac{|E_i|}{2} \jump{\grad u_h \cdot \bn}{E_i} \quad \forall i=1,\dots,n_z.
\]
This leads to 
\[
\tau_{E_i,z}^+-\tau^+_{E_{i-1},z}=
-\frac{|E_i|}{2}\jump{\grad u_h \cdot \bn}{E_i}
+(g_0,\phi_z)_{T_i} 
\quad\forall i=1,\dots,n_z
\]
and thus to
\[
\tau_{E_i,z}^+=  \tau^+_{E_{0},z}
+\sum\limits_{j=1}^{i}\left(-\frac{|E_j|}{2}\jump{\grad u_h \cdot \bn}{E_j}
+(g_0,\phi_z)_{T_j}\right) 
\quad \forall i=1,\dots,n_z.
\]
We can bound the two terms on the right hand side that are involved in the
summation as follows:
\[
\left|\frac{|E_j|}{2}\jump{\nabla u_h\cdot\bn}{E_j}\right|\le
C\|\grad u_h \|_{0,\omega_z}
\]
and
\[
|(g_0,\phi_z)_{T_j}|\le
\|g_0\|_{0,T_j}\|\phi_z\|_{0,T_j}\le Ch\|g_0\|_{0,T_j}.
\]
Choosing $\tau^+_{E_{0},z} =0$ we then have
%
\[
\aligned
&|\tau_{E_i,z}^+| \le C(\|\grad u_h \|_{0,\omega_z}+h\|g_0\|_{0,\omega_z})\\
&|\tau_{E_i,z}^-| \le C(\|\grad u_h \|_{0,\omega_z}+h\|g_0\|_{0,\omega_z})
&&\forall i=1,\dots,n_z.
\endaligned
\]
By a scaling argument or by using a suitable quadrature rule we see that
\[
\|\bfpsi_E^\pm \|_0^2=\frac{1}{4|T_E^\pm|^2}\|\bx-P_E^\pm\|_0^2
\le Ch^2/|T_E^\pm|\le C
\]
so that it holds
$\|\bfsigma_z\|_{0,\omega_z}\le
C(\|\grad u_h\|_{0,\omega_z}+h\|g_0\|_{0,\omega_z})$.
By putting all the patches together and considering that the intersections
between patches contain a bounded number of elements and that each element
belongs to a bounded number of patches we get
\[
\|\bfsigma_h\|_{L^2(\Omega)}\le
C(\|\grad u_h\|_{L^2(\Omega)}+h\|g_0\|_{L^2(\Omega)}).
\]

It remains to estimate $g_h$, which can be done by considering the inf-sup
constant discussed in Appendix~\ref{ap:asympt} concerning the $P_0-P_1$
element~\eqref{eq:p1p0}. By the definition of $g_0$ we have
\[
(g_0,v_h)=(\grad u_h,\grad v_h)\quad\forall v_h\in\Qh
\]
and hence
\[
\aligned
\zeta(h)\|g_0\|_0&\le\sup_{v_h\in\Qh}\frac{(g_0,v_h)}{\|v\|_0}
=\sup_{v_h\in\Qh}\frac{(\grad u_h,\grad v_h)}{\|v\|_0}
\le\sup_{v_h\in\Qh}\frac{\|\grad u_h\|_0\|\grad v_h\|_0}{\|v_h\|_0}\\
&\le\rho(h)\|\grad u_h\|_0.
\endaligned
\]
Finally, we arrive at the final estimate
\[
\|\bfsigma_h\|_{L^2(\Omega)}\le C\|\grad u_h\|_{L^2(\Omega)}
(1+h\rho(h)/\zeta(h)),
\]
which leads to the bound $C_\Pi\le Ch\rho(h)/\zeta(h)$.
\end{proof}

Although the above considerations do not provide a rigorous proof that the
inf-sup constant is vanishing with $h$, they give a clear indication that we
should not expect the constant $\beta_h$ to be uniformly bounded away from
zero. Indeed, the behavior of the inf-sup constant depends on the chosen mesh
as shown in Section~\ref{se:infsuph}: if the mesh is quasiuniform then
$h\rho(h)$ is bounded from above and below so that we have a confirmation that
the inf-sup constant cannot be better than $\zeta(h)$.

\section*{Acknowledgments}

The first author gratefully acknowledges support by the Deutsche
Forschungsgemeinschaft in the Priority Program SPP 1748 \textit{Reliable simulation
techniques in solid mechanics, Development of non standard discretization methods,
mechanical and mathematical analysis} under the project number BE 6511/1-1.
The second author is member of the INdAM Research group GNCS and his research
is partially supported by IMATI/CNR and by PRIN/MIUR.

\bibliographystyle{amsplain}
\bibliography{ref}

\end{document}